\title{Symplectic Divisorial Capping in Dimension 4}

\documentclass[11pt]{article}

\usepackage{amsfonts}
\usepackage{amsthm}
\usepackage{latexsym, amssymb, bbm}
\usepackage{xypic}
\usepackage{xcolor}
\usepackage{graphicx}
\usepackage{amsmath}
\usepackage{enumerate}
\usepackage[all]{xy}
\usepackage[margin=1.7in]{geometry}
\usepackage[toc,page]{appendix}

\usepackage{amsmath,amsthm,amsfonts,amssymb,graphicx,psfrag}
\usepackage{color}
\usepackage{hyperref}

\usepackage{tikz}
\usetikzlibrary{shapes.geometric, arrows}
\tikzstyle{startstop} = [rectangle, rounded corners, minimum width=0.5cm, minimum height=0.5cm,text centered, text width=2cm, draw=black, fill=white!30]
\tikzstyle{startstop2} = [rectangle, rounded corners, minimum width=0.5cm, minimum height=0.5cm,text centered, text width=2cm, draw=black, fill=white!30]
\tikzstyle{startstop3} = [rectangle, rounded corners, minimum width=0.5cm, minimum height=0.5cm,text centered, text width=2.5cm, draw=black, fill=white!30]
\tikzstyle{arrow} = [thick,->,>=stealth]

\newtheorem{thm}{Theorem}[section]

\newtheorem{prop}[thm]{Proposition}
\newtheorem{lemma}[thm]{Lemma}
\newtheorem{rmk}[thm]{Remark}

\newtheorem{defn}[thm]{Definition}
\newtheorem{corr}[thm]{Corollary}
\newtheorem{eg}[thm]{Example}

\newtheorem{fig}[thm]{Figure}

\begin{document}

\author{Tian-Jun Li and Cheuk Yu Mak\thanks{Both authors are supported by NSF-grant DMS 1065927.}}

\AtEndDocument{\bigskip{\footnotesize%
  \textsc{School of Mathematics, University of Minnesota, Minneapolis, MN 55455} \par  
  \textit{E-mail address}, Tian-Jun Li: \texttt{tjli@math.umn.edu} \par
  \addvspace{\medskipamount}
  \textsc{School of Mathematics, University of Minnesota, Minneapolis, MN 55455} \par  
  \textit{E-mail address}, Cheuk Yu Mak: \texttt{makxx041@math.umn.edu} \par

}}

\date{\today}

\maketitle

\begin{abstract}
We investigate the notion of symplectic divisorial compactification for symplectic 4-manifolds with either convex or concave type boundary.
This is motivated by the notion of compactifying divisors for open algebraic surfaces.
We give a sufficient and necessary criterion, which is simple and also works in higher dimensions, to determine whether an arbitrarily small concave/convex neighborhood exist for an $\omega$-orthogonal symplectic divisor (a symplectic plumbing).
If deformation of symplectic form is allowed, we show that a symplectic divisor has either a concave or convex neighborhood whenever the symplectic form is exact on the boundary of its plumbing. 
As an application, we classify symplectic compactifying divisors having finite boundary fundamental group.
We also obtain a finiteness result of fillings when the boundary can be capped by a symplectic divisor with finite boundary fundamental group.
\end{abstract}

\tableofcontents

\section{Introduction}
In this paper, a {\it symplectic divisor} refers to a connected configuration of finitely many closed embedded symplectic surfaces $D=C_1 \cup \dots \cup C_k$ 
in   a symplectic 4 dimensional manifold (possibly with boundary or non-compact) $(W, \omega)$. $D$ is further required to have the following properties:   $D$ has  empty intersection with $ \partial W$, no three $C_i$ intersect at a point, and any intersection between two surfaces is transversal and positive.
The orientation of each $C_i$ is chosen to be positive with respect to $\omega$.
Since we are interested in the germ of a symplectic divisor, $W$ is sometimes omitted in the writing and $(D,\omega)$, or simply $D$,  is used to denote a symplectic divisor.

A closed regular neighborhood of $D$ is called a plumbing of $D$. The plumbings are well defined up to orientation preserving diffeomorphism, so we can introduce
topological invariants of $D$ using any of its plumbings. In particular, $b_2^{\pm}(D)$ is  defined as $b_2^{\pm}$ of a plumbing. Similarly, we  
define  the {\it boundary} of the divisor $D$, and we call  the fundamental group of the boundary {\it boundary fundamental group} of $D$.
In the same vein,  when  $\omega$ is exact on the boundary of a plumbing, we say that $\omega$ is exact on the boundary of $D$.

A plumbing $P(D)$ of $D$ is called a {\it concave (resp. convex) neigborhood} if $P(D)$ is a strong concave (resp. convex) filling of its boundary.
A symplectic divisor $D$ is called {\it concave}  (resp. {\it convex}) if   for any neighborhood $N$ of $D$, there is a concave (resp. convex) neighborhood $P(D) \subset N$ for the divisor.
Through out this paper, all concave (resp. convex) fillings are symplectic strong concave (resp. strong convex) fillings and we simply call it cappings or concave fillings (resp. fillings or convex fillings).

\begin{defn}
Suppose that $D$ is a concave (resp. convex) divisor.  If a symplectic gluing (\cite{Et98}) can be performed for a concave (resp. convex) neighborhood of $D$ and a symplectic manifold $Y$ with
convex (resp. concave) boundary to obtain a closed symplecitc manifold, then
we call $D$ a {\bf capping} (resp. {\bf filling}) divisor.
In both cases, we call $D$ a {\bf compactifying} divisor of $Y$.

\end{defn}

\subsection{Motivation}

We provide some motivation from two typical families of examples in algebraic geometry together with some general symplectic compactification phenomena.

Suppose $Y$ is a smooth affine algebraic variety over $\mathbb{C}$.
Then $Y$ can be compactified by a  divisor $D$ to a projective variety $X$.
By Hironaka's resolution of singularities theorem,  we could assume that $X$ is smooth and $D$ is a simple normal crossing divisor.
In this case, $Y$ is a Stein manifold and $D$ has a concave neighborhood induced by a plurisubharmonic function on $Y$ (\cite{ElGr91}).
Moreover, $Y$ is symplectomorphic to the completion of a suitably chosen Stein domain $\overline{Y}\subset Y$ (See e.g. \cite{McL12}).
Therefore, compactifying $Y$ by $D$ in the algebro geometric  situation is analogous to gluing $\overline{Y}$ with a concave neighborhood of $D$ along their contact boundaries \cite{Et98}.

On the other hand, suppose we have a compact complex surface with an isolated normal singularity.
We can resolve the isolated normal singularity and obtain a pair $(W,D)$, where $W$ is a smooth compact complex surface and $D$ is a simple normal crossing resolution divisor.
In this case, we can define a K\"ahler form near $D$ such that $D$ has a convex neighborhood $P(D)$.
If the K\"ahler form can be extended to $W$, then the K\"ahler compactification of $W-D$ by $D$ is analogous to gluing the symplectic manifold $W-Int(P(D))$ with $P(D)$ along their contact boundaries.

From the symplectic point of view, there are both flexibility and constraints for capping a symplectic 4 manifold $Y$ with convex boundary. 
For flexibility, there are infinitely many ways to embed $Y$ in closed symplectic 4-manifolds (Theorem 1.3 of \cite{EtHo02}).
This still holds even when $Y$ has only weak convex boundary (See \cite{El04} and \cite{Et04}).
For constraints, it is well-known that (e.g. \cite{Hu13}) $Y$ does not have any exact capping.
From these perspectives, divisor cappings might provide a suitable capping model to study (See also \cite{Ga03} and \cite{Ga03c}).

On the other hand,  divisor fillings have been  studied by several authors.  For instance, 
it is known that  they are the maximal fillings for the canonical contact structures on Lens spaces (See \cite{Li08} and \cite{BhOz14}).

In this setting, the following questions are natural:
Suppose $D$ is a symplectic divisor.

(i) When is $D$ also a compactifying divisor?

(ii) What symplectic manifolds can be compactified by $D$?

\subsection{A Flowchart}\label{A Flowchart}

Regarding the first question, observe that a divisor is a capping (resp. filling) divisor if it is concave (resp. convex), and embeddable in the following sense:

\begin{defn}
If a symplectic divisor $D$ admits a symplectic embedding into a closed symplectic manifold $W$, then we call $D$ an {\bf embeddable} divisor. 
\end{defn}

We recall some results from the literature for the filling side.
It is proved in \cite{GaSt09} that when the graph of a symplectic divisor is negative definite, it can always be perturbed to be a convex divisor.
Moreover, a convex divisor is always embeddable, by \cite{EtHo02}, hence a filling divisor.

However, a concave divisor is not necessarily embeddable.
An obstruction is provided by \cite{Mc90} (See Theorem \ref{McDuff}).

Our first main result:

\begin{thm}\label{MAIN}
  Let $D \subset (W,\omega_0)$ be a symplectic divisor.
If the intersection form of $D$ is not negative definite and $\omega_0$ restricted to the boundary  of $D$ is exact, then $\omega_0$ can be deformed through a family of symplectic forms $\omega_t$ on $W$ keeping $D$  symplectic   and such that $(D,\omega_1)$ is a concave divisor.

In particular, if $D$ is also an embeddable divisor, then it is a capping divisor after a deformation.
\end{thm}

It is convenient to  associate  an augmented graph $(\Gamma,a)$  to a symplectic divisor $(D,\omega)$, 
where $\Gamma$ is the graph of $D$ and $a$ is the area vector for the embedded symplectic surfaces (See Section \ref{Preliminary} for details).
The intersection form of $\Gamma$  is denoted by $Q_{\Gamma}$.

\begin{defn}
 Suppose $(\Gamma,a)$ is an augmented graph with $k$ vertices.
Then, we say that $(\Gamma,a)$ satisfies the positive (resp. negative) {\bf GS criterion} if there exists
 $z \in (0,\infty)^k$ (resp $(-\infty,0]^k$) such that $Q_{\Gamma}z=a$.

A symplectic divisor is said to satisfy the  positive (resp. negative) GS criterion if its associated augmented graph does.
\end{defn}

One important ingredient for the proof of Theorem \ref{MAIN} is the following result.
\begin{prop}\label{MAIN2}
  Let $(D,\omega)$ be a symplectic divisor with $\omega$-orthogonal intersections.
Then, $(D,\omega)$ has a concave (resp. convex) neighborhood inside any regular neighborhood of $D$ if $(D,\omega)$ satisfies the positive (resp. negative) GS criterion. 
\end{prop}

The construction  is essentially due to Gay and Stipsicz in \cite{GaSt09}, which we call the GS construction.
We remark that GS criteria can be verified easily.
They are conditions on wrapping numbers in disguise.
Therefore, by a recent result of Mark McLean \cite{McL14}, Proposition \ref{MAIN2} can be generalized to higher dimensions with GS criteria being replaced accordingly.
Moreover, using techniques in \cite{McL14}, we  establish the necessity of the GS criterion and answer the uniqueness  question in \cite{GaSt09}.

\begin{thm}\label{obstruction-GS}
 Let $D \subset (W,\omega)$ be an $\omega$-orthogonal symplectic divisor.
If $(D,\omega)$ does not satisfy the positive (resp. negative) GS criterion.
Then, there is a neighborhood $N$ of $D$ such that any plumbing $P(D) \subset N$ of $D$ is not a concave (resp. convex) neighborhood.
\end{thm}

\begin{thm}\label{uniqueness-GS}
 Let $(D,\omega_i)$ be $\omega_i$-orthogonal  symplectic divisors for $i=0,1$ such that both
 satisfy the positive (resp. negative) GS criterion.
 Then the  concave (resp. convex) structures on the boundary of $(D,\omega_0)$ and $(D,\omega_1)$ via the GS construction are contactomorphic.
 
 In particular, when $\omega_0=\omega_1$, the contact structure constructed via GS construction is independent of choices, up to contactomorphism.
\end{thm}

Summarizing Theorem \ref{MAIN} and Proposition \ref{MAIN2}, we have

\begin{corr}\label{QHS}
 Let $(D,\omega)$ be a   symplectic divisor with $\omega$ exact on the boundary of $D$.
Then $D$ is either a concave divisor or a convex divisor, possibly after a symplectic deformation. 
\end{corr}

More precise information is illustrated by the following  schematic flowchart.

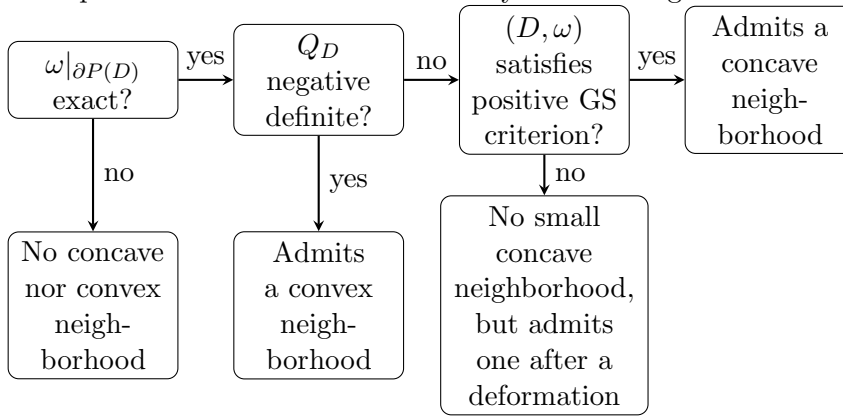
\begin{tikzpicture}[node distance=3cm]

\node (exact) [startstop] {$\omega|_{\partial P(D)}$ exact?};
\node (not exact) [startstop2, below of=exact] {No concave nor convex neighborhood};
\node (definite) [startstop, right of=exact] {$Q_D$ negative definite?};
\node (convex) [startstop2, below of=definite] {Admits a convex neighborhood};
\node (GS criterion) [startstop, right of=definite] {$(D,\omega)$ satisfies positive GS criterion?};
\node (concave) [startstop2, right of=GS criterion] {Admits a concave neighborhood};
\node (deformation) [startstop3, below of=GS criterion] {No small concave neighborhood, but admits one after a deformation};

\draw [arrow] (exact) -- node[right]{no}(not exact);
\draw [arrow] (exact) -- node[above]{yes}(definite);
\draw [arrow] (definite) -- node[right]{yes}(convex);
\draw [arrow] (definite) -- node[above]{no}(GS criterion);
\draw [arrow] (GS criterion) -- node[above]{yes}(concave);
\draw [arrow] (GS criterion) -- node[right]{no}(deformation);

\end{tikzpicture}

For a general divisor $(D,\omega)$, which is not necessarily $\omega$-orthogonal,
the corresponding results for Proposition \ref{MAIN2} and Theorem \ref{uniqueness-GS} are still valid 
(See Proposition \ref{McLean0} and Proposition \ref{McLean}), by McLean's construction.
The generalization of Theorem \ref{obstruction-GS} is a bit subtle.
For an embeddable divisor $(D,\omega)$, we obtain in Theorem \ref{obstruction-closed case} that if it does not satisfy the positive GS criterion
then there is a neighborhood $N$ of $D$ such that any plumbing $P(D) \subset N$ is not a concave neighborhood.

\subsection{Divisors with Finite Boundary $\pi_1$}

 Using Theorem \ref{MAIN}, Theorem \ref{obstruction-closed case}, Proposition \ref{McLean0} and Proposition \ref{McLean},
 we classify, in our second main results Theorem \ref{main classification theorem} and Theorem \ref{complete classification}, 
 capping divisors (not necessarily $\omega$-orthogonal)
 with finite boundary fundamental group.
This allows us to answer the second question if $D$ is a capping divisor with finite boundary fundamental group.
As a consequence, we show that only finitely many minimal symplectic manifolds can be compactified by $D$, up to diffeomorphism. 
More details are described in Section \ref{Finiteness}.

Moreover, we also investigate special kinds of symplectic filling.
In Section \ref{Non-Conjugate Phenomena}, we study pairs of symplectic divisors that compactify each other.

\begin{defn}\label{Conjugate Definition_Divisor}
For symplectic divisors $D_1$ and $D_2$, we say that they are {\bf conjugate} to each other if there exists plumbings $P(D_1)$ and $P(D_2)$ for $D_1$ and $D_2$, respectively, such that
$D_1$ is a capping divisor of $P(D_2)$ and $D_2$ is a filling divisor of $P(D_1)$. 
\end{defn}

On the other hand, it is also interesting to investigate  the category of symplectic manifolds having  symplectic divisorial compactifications. 
Affine surfaces are certainly in this category. 
In this regard,  symplectic cohomology could play an important role.
Growth rate of symplectic cohomology has been used in \cite{Se08} and \cite{McL12} 
to distinguish a family of cotangent bundles from affine varieties.
The proof actually applies to any Liouville domain which admits a divisor cap, so certain boundedness  on the growth rate is  necessary  for 
a Liouville domain to be in this category.
Finally, 
for  symplectic manifolds in this category we would like to  define invariants in  terms of the divisorial compactifications 
(See \cite{LiZh11} for a related invariant).

The remaining of this article is organized as follows.
In Section~\ref{Preliminary}, we give the proof of Proposition \ref{MAIN2}, Theorem \ref{obstruction-GS} and Theorem \ref{uniqueness-GS}.
Section \ref{Operation on Divisors} is mainly devoted to the proof of Theorem \ref{MAIN}.
We give the statement and proof of the classification of compactifying divisors with finite boundary fundamental group in Section~\ref{Classification of Symplectic Divisors Having Finite Boundary Fundamental Group}.

\subsection*{Acknowledgements}
The authors would like to thank Mark McLean for many helpful discussions, in particular for explaining the canonical contact structure in \cite{McL14}.
They would also like to thank David Gay, Ko Honda, Burak Ozbagci, Andras Stipsicz, Weiwei Wu and Weiyi Zhang for their interest in this work. 
They are also grateful to Laura Starkston for stimulating discussions.

\section{Contact Structures on the Boundary}\label{Preliminary}

Essential topological information of a symplectic divisor can be encoded by its {\it graph}.
The graph is a weighted finite graph with vertices representing the surfaces and each edge joining two vertices representing an intersection between the two surfaces corresponding to the two vertices.
Moreover, each vertex is weighted by its genus (a non-negative integer) and its self-intersection number (an integer).

If each vertex is also weighted by its symplectic area (a positive real number), then we call it an {\it augmented graph}.
Sometimes, the genera (and the symplectic area) are not explicitly stated.
For simplicity, we would like to assume the symplectic divisors are connected.

In what follow, we call a finite graph weighted by its self-intersection number and its genus (resp. and its area) with no edge coming from a vertex back to itself a graph (resp. an augmented graph).
For a graph (resp. an augmented graph) $\Gamma$ (resp. $(\Gamma,a)$), we use $Q_{\Gamma}$ to denote the intersection matrix for $\Gamma$ (resp. and $a$ to denote the area weights for $\Gamma$).
We denote the determinant of $Q_{\Gamma}$ as $\delta_{\Gamma}$.
Moreover, $v_1,\dots,v_k$ are used to denote the vertices of $\Gamma$ and $s_i$, $g_i$ and $a_i$ are self-intersection, genus and area of $v_i$, respectively.

Notice that, $\omega$ being exact on the boundary of a plumbing is equivalent to $[\omega]$ being able to be lifted to a relative cohomological class.
Using Lefschetz duality, this is in turn equivalent to $[\omega]$ being able to be expressed as a linear combination $\sum\limits_{i=1}^k z_i[C_i]$, where $z_i \in \mathbb{R}$ and $D=C_1 \cup \dots \cup C_k$.
As a result, $\omega$ is exact on the boundary of a plumbing if and only if there exist a solution $z$ for the equation $Q_{\Gamma}z=a$ (See Subsection \ref{wrapping numbers} for a more detailed discussion).

We also remark that the germ  of a symplectic divisor $(D,\omega)$ with $\omega$-orthogonal intersections is uniquely determined by its augmented graph $(\Gamma,a)$ (See \cite{McR05} and Theorem 3.1 of \cite{GaSt09}) and a symplectic divisor can always be made $\omega$-orthogonal after a perturbation (See \cite{Go95}).

\begin{eg} \label{2-1} The graph 
\begin{displaymath}
\xymatrix{
        \bullet^{2}_{v_1} \ar@{-}[r] & \bullet^{1}_{v_2}\\
	}
\end{displaymath} 
where both vertices are of genus zero, 
represents a symplectic divisor of two spheres with self-intersection $2$ and $1$ and intersecting  positively transversally at a point.
\end{eg}

\begin{defn}\label{realizable definition}
A graph $\Gamma$ is called {\bf realizable} (resp. {\bf strongly realizable}) if there is an embeddable (resp. compactifying) symplectic divisor $D$ such that its graph is the same as $\Gamma$.
In this case, $D$ is called a realization (resp. strongly realization) of $\Gamma$.
\end{defn}

Similar to Definition \ref{realizable definition}, we can define realizability and strongly realizability for an augmented graph.
If the area weights attached to $\Gamma$ is too arbitrary, it is possible that $(\Gamma,a)$ is not strongly realizable but $\Gamma$ is strongly realizable. 

\subsection{Existence}

In this subsection, Proposition \ref{MAIN2} is given via two different approaches, namely, GS construction and McLean's construction.

\subsubsection{Existence via the GS construction}

\begin{defn}\cite{GaSt09}
 $(X,\omega,D,f,V)$ is said to be an {\bf orthogonal neighborhood 5-tuple} if $(X,\omega)$ is a symplectic 4-manifold  with
$D$ being a collection of closed symplectic surfaces in $X$ intersecting $\omega$-orthogonally such that
$f:X \to [0,\infty)$ is a smooth function with no critical value in $(0,\infty)$ and with $f^{-1}(0)=D$, and
$V$ is a Liouville vector field on $X-D$.

Moreover, if $df(V)>0$ (resp $<0$), then $(X,\omega,D,f,V)$ is called a convex (resp concave) neighborhood 5-tuple. 
\end{defn}

\begin{figure}\label{fig GS construction}
\begin{tikzpicture}[scale=1.5]
    \draw [<->,thick] (0,4.5) node (yaxis) [above] {$y$}
        |- (6,0) node (xaxis) [right] {$x$};
    \draw (1.5,0.9) coordinate (a_1) -- (4.5,0.9) coordinate (a_2);
    \draw (1.5,0.9) coordinate (b_1) -- (1.5,3.9) coordinate (b_2);
    \draw (2.4,1.8) coordinate (d_1)-- (5.4,1.8) coordinate (d_2);   
    \draw (2.4,1.8) coordinate (e_1)-- (2.4,3) coordinate (e_2);
    \draw (2.5,3.3) node [left] {$R_{e_{\alpha \beta},v_{\alpha}} $};
    \draw (2.4,3) coordinate (j_1)-- (2.4,4.3) coordinate (j_2);

    \draw (1.5,3.9) coordinate (f_1)-- (2.4,4.3) coordinate (f_2); 
    \draw (1.5,2.4) coordinate (g_1)-- (2.4,2.8) coordinate (g_2);
    \draw (3,0.9) coordinate (h_1)-- (3.9,1.8) coordinate (h_2);
    \draw (3.8,1.3) node [right] {$R_{e_{\alpha \beta},v_{\beta}} $};
    \draw (4.5,0.9) coordinate (i_1)-- (5.4,1.8) coordinate (i_2); 
    \coordinate (c) at (intersection of a_1--a_2 and b_1--b_2);
    \draw[dashed] (yaxis |- c) node[left] {$z_{\beta}'$}
        -| (xaxis -| c) node[below] {$z_{\alpha}'$};
    \draw[dashed] (1.5,2.4) -- (0,2.4) node[left] {$z_{\beta}'+\epsilon$};
    \draw[dashed] (1.5,3.9) -- (0,3.9) node[left] {$z_{\beta}'+2\epsilon$};
    \draw[dashed] (3,0.9) -- (3,0) node[below] {$z_{\alpha}'+\epsilon$};
    \draw[dashed] (4.5,0.9) -- (4.5,0) node[below] {$z_{\alpha}'+2\epsilon$};        
        
\end{tikzpicture}

Figure \ref{fig GS construction}
\end{figure}

In \cite{GaSt09}, Gay and Stipsicz constructed a convex orthogonal neighborhood 5-tuple $(X,\omega,D,f,V)$ when the augmented graph $(\Gamma,a)$ of $D$ satisfies the negative GS criterion. 
We first review their construction and an immediate consequence will be Proposition \ref{MAIN2}.  

Let $z$ be a vector solving $Q_{\Gamma}z=a$ with $z \in (-\infty,0]^k$.
Then $z'=(z_1',\dots,z_n')^T=\frac{-1}{2\pi}z$ has all entries being non-negative.
We remark that the $z'$ we use corresponds to the $z$ in \cite{GaSt09}.

For each vertex $v$ and each edge $e$ meeting the chosen $v$, we set $s_{v,e}$ to be an integer.
These integers $s_{v,e}$ are chosen such that $\sum\limits_{\text{e meeting v}} s_{v,e}=s_v$ for all $v$, where $s_v$ is the self-intersection number of the vertex $v$.
Also, set $x_{v,e}=s_{v,e}z_v'+z_{v'}'$, where $v'$ is the other vertex of the edge $e$.

For each edge $e_{\alpha \beta}$ of $\Gamma$ joining vertices $v_\alpha$ and $v_\beta$, we construct a local model $N_{e_{\alpha \beta}}$ as follows.
Let $\mu:\mathbb{S}^2 \times \mathbb{S}^2 \to [z_{\alpha}',z_{\alpha}'+1] \times [z_{\beta}',z_{\beta}'+1]$ be the moment map of $\mathbb{S}^2 \times \mathbb{S}^2$ onto its image.
We use $p_1$ for coordinate in $[z_{\alpha}',z_{\alpha}'+1]$, $p_2$ for coordinate in $[z_{\beta}',z_{\beta}'+1]$ and $q_i \in \mathbb{R}/2\pi$ be the corresponding fibre coordinates so $\theta=p_1dq_1+p_2dq_2$ gives a primitive of the symplectic form $dp_1 \wedge dq_1 + dp_2 \wedge dq_2$ on the preimage of the interior of the moment image.

Fix a small $\epsilon >0$ and let $D_1=\mu^{-1}(\{ z_{\alpha}'\} \times [z_{\beta}', z_{\beta}' + 2\epsilon])$ be a symplectic disc.
Let also $D_2=\mu^{-1}([z_{\alpha}', z_{\alpha}' + 2\epsilon] \times \{ z_{\beta}' \})$ be another symplectic disc meeting $D_1$ $\omega$-orthogonal at the point $\mu^{-1}(\{z_{\alpha}'\} \times \{z_{\beta}' \})$.

Our local model $N_{e_{\alpha \beta}}$ is going to be the preimage under $\mu$ of a region containing $\{ z_{\alpha}'\} \times [z_{\beta}', z_{\beta}' + 2\epsilon] \cup [z_{\alpha}', z_{\alpha}' + 2\epsilon] \times \{ z_{\beta}' \}$.

A sufficiently small $\delta$ will be chosen.
For this $\delta$, let $R_{e_{\alpha \beta},v_{\alpha}}$  be the closed parallelogram with vertices $(z_{\alpha}',z_{\beta}'+\epsilon), (z_{\alpha}',z_{\beta}'+2\epsilon), (z_{\alpha}'+\delta,z_{\beta}'+2\epsilon-s_{v_{\alpha},e_{\alpha \beta}} \delta), (z_{\alpha}'+\delta,z_{\beta}'+\epsilon-s_{v_{\alpha},e_{\alpha \beta}} \delta)$.
Also, $R_{e_{\alpha \beta},v_{\beta}}$ is defined similarly as the closed parallelogram with vertices $(z_{\alpha}'+\epsilon,z_{\beta}'), (z_{\alpha}'+2 \epsilon,z_{\beta}'), (z_{\alpha}'+2\epsilon-s_{v_{\beta},e_{\alpha \beta}} \delta,z_{\beta}'+\delta), (z_{\alpha}'+\epsilon-s_{v_{\beta},e_{\alpha \beta}} \delta,z_{\beta}'+\delta)$.
We extend the right vertical edge of $R_{e_{\alpha \beta},v_{\alpha}}$ downward and extend the top horizontal edge of $R_{e_{\alpha \beta},v_{\beta}}$ to the left until they meet at the point $(z_{\alpha}'+\delta,z_{\beta}'+\delta)$.
Then, the top edge of $R_{e_{\alpha \beta},v_{\alpha}}$, the right edge of $R_{e_{\alpha \beta},v_{\beta}}$, the extension of right edge of $R_{e_{\alpha \beta},v_{\alpha}}$, the extension of top edge of $R_{e_{\alpha \beta},v_{\beta}}$, $\{ z_{\alpha}'\} \times [z_{\beta}', z_{\beta}' + 2\epsilon]$ and $[z_{\alpha}', z_{\alpha}' + 2\epsilon] \times \{ z_{\beta}' \}$
enclose a region.
After rounding the corner symmetrically at $(z_{\alpha}'+\delta,z_{\beta}'+\delta)$, we call this closed region $R$.
Now, we set $N_{e_{\alpha \beta}}$ to be the preimage of $R$ under $\mu$.
See Figure \ref{fig GS construction}.

On the other hand, for each vertex $v_{\alpha}$, we also need to construct a local model $N_{v_{\alpha}}$.
Let $g_{\alpha}$ be the genus of $v_{\alpha}$.
We can form a genus $g_{\alpha}$ compact Riemann surface $\Sigma_{v_{\alpha}}$ such that the boundary components  one to one correspond to the edges meeting $v_{\alpha}$.  
We denote the boundary component corresponding to $e_{\alpha \beta}$ by $\partial_{e_{\alpha \beta}} \Sigma_{v_{\alpha}}$.
There exists a symplectic form $\bar{\omega}_{v_{\alpha}}$ and a Liouville vector field $\bar{X}_{v_{\alpha}}$ on $\Sigma_{v_{\alpha}}$ such that 
when we give the local coordinates $(t,\vartheta_1) \in (x_{v_{\alpha}, e_{\alpha \beta}}-2\epsilon,x_{v_{\alpha}, e_{\alpha \beta}}-\epsilon] \times \mathbb{R}/2\pi \mathbb{Z}$ to the neighborhood of the boundary component $\partial_{e_{\alpha \beta}} \Sigma_{v_{\alpha}}$, we have that 
$\bar{\omega}_{v_{\alpha}}=dt \wedge d\vartheta_1$ and $\bar{X}_{v_{\alpha}}=t \partial_t$.
Now, we form the local model $N_{v_{\alpha}}=\Sigma_{v_{\alpha}} \times \mathbb{D}^2_{\sqrt{2\delta}}$ with product symplectic form $\omega_{v_{\alpha}}=\bar{\omega}_{v_{\alpha}}+rdr \wedge d\vartheta_2$ and Liouville vector field $X_{v_{\alpha}}=\bar{X}_{v_{\alpha}}+(\frac{r}{2}+\frac{z_{v_{\alpha}}'}{r})\partial_r$, where $(r,\vartheta_2)$ is the standard polar coordinates on $\mathbb{D}^2_{\sqrt{2\delta}}$.

Finally, the GS construction is done by gluing these local models appropriately.
To be more precise, the preimage of $R_{e_{\alpha \beta},v_{\alpha}}$ of $N_{e_{\alpha \beta}}$ is glued via a symplectomorphism preserving the Liouville vector field to $[x_{v_{\alpha}, e_{\alpha \beta}}-2\epsilon,x_{v_{\alpha}, e_{\alpha \beta}}-\epsilon] \times \mathbb{R}/2\pi \mathbb{Z} \times \mathbb{D}^2_{\sqrt{2\delta}}$ of $N_{v_{\alpha}}$ and other matching pieces are glued similarly.
When $\delta >0$ is chosen  sufficiently small, this glued manifold give our desired convex orthogonal neighborhood 5-tuple with the symplectic divisor having graph $\Gamma$.

We remark  the whole construction works exactly the same if all entries of $z'$ are negative.
In this case, all entries of $z$ are positive and we get the desired concave orthogonal neighborhood 5-tuple if $(\Gamma,a)$ satisfies the positive GS criterion.
Now, if we have an $\omega'$-orthogonal divisor $(D',\omega')$ with augmented graph $(\Gamma,a)$, which 
is the same as that of the concave orthogonal neighborhood 5-tuple $(X,\omega,D,f,V)$,
then there exist neighborhood $N'$ of $D'$ symplectomorphic to a neighborhood of $D$ and sending $D'$ to $D$ (See \cite{McR05} and \cite{GaSt09}).
Therefore, a concave neighborhood of $D$ in $N$ give rise to a concave neighborhood of $D'$ in $N'$.
This finishes the proof of Proposition \ref{MAIN2}.

\subsubsection{Existence in Higher Dimensions via Wrapping Numbers}\label{wrapping numbers}
To understand the geometrical meaning of the GS criteria, we  recall wrapping numbers  from \cite{McL14} and \cite{McL12}.
Then, another construction for Proposition \ref{MAIN2} is given.

Let $(D,P(D),\omega)$ be a plumbing of a symplectic divisor.
If $\omega$ is not exact on the boundary of $D$, then there is no Liouville flow $X$ near $\partial P(D)$ such that $\alpha=i_X \omega$ and $d\alpha=\omega$.
Therefore, $D$ does not have concave nor convex neighborhood.

When $\omega$ is exact on the boundary, let $\alpha$ be a $1$-form on $P(D)-D$ such that $d \alpha=\omega$.
Let $\alpha_c$ be a $1$-form on $P(D)$ such that it is $0$ near $D$ and it equals $\alpha$ near $\partial P(D)$.
Note that $[\omega-d\alpha_c] \in H^2(P(D), \partial P(D);\mathbb{R})$.
Let its Lefschetz dual be $-\sum\limits_{i=1}^k \lambda_i [C_i] \in H_2(P(D); \mathbb{R})$.
We call $\lambda_i$ the wrapping number of $\alpha$ around $C_i$.

Also, there is another equivalent interpretation of wrapping numbers.
If we symplectically embed a small disc to $P(D)$ meeting $C_i$ positively transversally at the origin of the disc, then the pull-back of $\alpha$ equals $ \frac{r^2}{2}d\vartheta + \frac{\lambda_i}{2\pi} d\vartheta +df$, where $(r,\vartheta)$ is the polar coordinates of the disc and $f$ is some function defined on the punctured disc. (See the paragraph before Lemma 5.17 of \cite{McL12}).

From this point of view, we can see that the $z_i$'s in the GS criteria are minus of the wrapping numbers $-\lambda_i$'s for a lift of the symplectic class $[\omega] \in H^2(P(D);\mathbb{R})$ to $H^2(P(D),\partial P(D);\mathbb{R})$.
In particular, $Q$ being non-degenerate is equivalent to lifting of symplectic class being unique, which is in turn equivalent to the connecting homomorphism $H^1(\partial P(D); \mathbb{R}) \to H^2(P(D),\partial P(D);\mathbb{R})$ is zero.
When $Q$ is degenerate and for a fixed $\omega$, the equation $Qz=a$ having no solution for $z$ is equivalent to $\omega|_{\partial P(D)}$ being not exact.
Similarly, when $Qz=a$ has a solution for $z$, then the solution is unique up to the kernel of $Q$, which corresponds to the unique lift of $\omega$ up to the image of the connecting homomorphism $H^1(\partial P(D); \mathbb{R}) \to H^2(P(D),\partial P(D);\mathbb{R})$.

To summarize, we have

\begin{lemma}
 Let $(D,\omega)$ be a symplectic divisor.
 Then, lifts $[\omega-d\alpha_c] \in H^2(P(D),\partial P(D);\mathbb{R})$ of the symplectic class $[\omega]$ are in one-to-one correspondence to the solution $z$ of $Q_Dz=a$ via the minus of Lefschetz dual $PD([\omega-d\alpha_c])= -\sum\limits_{i=1}^k \lambda_i [C_i]$ and $z_i=-\lambda_i$.
\end{lemma}

Proposition  \ref{MAIN2} can be  generalized to arbitrary  dimension if we apply the constructions in the recent paper of McLean \cite{McL14}.
We first recall an appropriate definition of a symplectic divisor in higher dimension (See \cite{McL14} or \cite{McL12}).

Let $(W^{2n},\omega)$ be a symplectic manifold with or without boundary.
Let $C_1, \dots, C_k$ be real codimension $2$ symplectic submanifolds of $W$ that intersect $\partial W$ trivially (if any).
Assume all intersections among $C_i$ are transversal and positive, where positive is defined in the following sense.

(i) For each $I \subset \{1, \dots,k \}$, $C_I=\cap_{i \in I}C_i$ is a symplectic submanifold.

(ii) For each $I, J \subset \{1, \dots, k\}$ with $C_{I \cup J} \neq \emptyset$, we let $N_1$ be the symplectic normal bundle of $C_{I \cup J}$ in $C_I$ and $N_2$ be the symplectic normal bundle of $C_{I \cup J}$ in $C_J$.
Then, it is required that the orientation of $N_1 \oplus N_2 \oplus TC_{I \cup J}$ is compatible with the orientation of $TW|_{C_{I \cup J}}$.

We remark that the condition (ii) above guarantees that no three distinct $C_i$ intersect at a common point when $W$ is four dimensions.
Therefore, this higher dimension definition coincides with the one we use in four dimension. 
To make our paper more consistent, in higher dimension, we call $D= C_1 \cup \dots \cup C_k$ a symplectic divisor 
if $D$ is moreover connected and the orientation of each $C_i$ is induced from $\omega^{n-1}|_{C_i}$.

Now, for each $i$, let $N_i$ be a neighborhood of $C_i$ such that we have a smooth projection $p_i: N_i \to C_i$ with a connection rotating the disc fibers.
Hence, for each $i$, we have a well-defined radial coordinate $r_i$ with respect to the fibration $p_i$ such that $C_i$ corresponds to $r_i=0$.

Let $\bar{\rho}: [0,\delta) \to [0,1]$ be a smooth function such that $\bar{\rho}(x)=x^2$ near $x=0$ and $\bar{\rho}(x)=1$ when $x$ is close to $\delta$.
Moreover, we require $\bar{\rho}'(x) \ge 0$.

A smooth function $f: W-D \to \mathbb{R}$ is called compatible with $D$ if $f=\sum\limits_{i=1}^k \log(\bar{\rho}(r_i))+\bar{\tau}$ for some smooth $\bar{\tau}: W \to \mathbb{R}$ and choice of $\bar{\rho}(r_i)$ as above.

Here is the analogue of Proposition \ref{MAIN2} in arbitrary even dimension. 

\begin{prop}[cf. Propositon 4.1 of \cite{McL14}] \label{McLean0}
 Suppose $f: W^{2n}-D \to \mathbb{R}$ is compatible with $D$ and $D$ is a symplectic divisor with respect to $\omega$.
Suppose $\theta \in \Omega^1(W^{2n}-D)$ is a primitive of $\omega$ on $W^{2n}-D$ such that it has positive (resp. negative) wrapping numbers for all $i=1, \dots, k$.
Then, there exist $g: W^{2n}-D \to \mathbb{R}$ such that
$df(X_{\theta+dg}) > 0$ (resp. $df(-X_{\theta+dg}) > 0$) near $D$, where $X_{\theta+dg}$ is the dual of $\theta+dg$ with respect to $\omega$.

In particular, $D$ is a convex (resp. concave) divisor.
\end{prop}

This is essentially contained in  Propositon 4.1 of \cite{McL14}--the  only new statement is the last sentence.  And 
 Proposition  4.1 in \cite{McL14} is stated only for the case in which wrapping numbers are all positive,
however, the proof there goes through without additional difficulty for the other case.
We give here the most technical lemma  adapted to the case of negative wrapping numbers and ambient manifold being dimension four for the sake of completeness. 

We remark that the $\omega$-orthogonal intersection condition is not required in his construction.

\begin{lemma}[cf. Lemma 4.5 of \cite{McL14}]
 Given $D=D_1 \cup D_2 \subset (U,\omega)$, where $D_1$ and $D_2$ are symplectic $2$-discs intersecting each other positively and transversally at a point $p$.
Suppose $\theta \in \Omega^1(U-D)$ is a primitive of $\omega$ on $U-D$ such that it has negative wrapping numbers with respect to both $D_1$ and $D_2$.
Then there exists $g$ such that for all smooth functions $f:U-D \to \mathbb{R}$ compatible wtih $D$, we have that $df(-X_{\theta+dg}) > c_f \|\theta+dg\| \| df\|$ near $D$,
where $c_f>0$ is a constant depending on $f$.

Also $c_1 \| db\| < \|  \theta+dg \| < c_2 \| db\|$ near $D$ for some smooth function $b$ compatible with $D$, where $c_1$ and $c_2$ are some constants.
\end{lemma}

\begin{proof}
 By possibly shrinking $U$, we give a symplectic coordinate system at the intersection point $p$ such that $D_1=\{x_1=y_1=0\}$ and $0$ corresponds to $p$.
Let $\pi_1$ be the projection to the $x_2,y_2$ coordinates.
Write $x_1=r \cos \vartheta$ and $y_1=r \sin \vartheta$ and let $\tau=\frac{r^2}{2}$.
Let $U_1=U-D_1$ and $\tilde{U_1}'$ be the universal cover of $U_1$ with covering map $\alpha$.
Give $\tilde{U_1}'$ the coordinates $(\tilde{x_1},\tilde{y_1},\tilde{x_2},\tilde{y_2})$ coming from pulling back the coordinates of $(\tau,\vartheta,x_2,y_2)$ by the covering map.
Then, the pulled back symplectic form on $\tilde{U_1}'$ is given by $d\tilde{x_1} \wedge d\tilde{y_1} + d\tilde{x_2} \wedge d\tilde{y_2}$. 
Hence, we can enlarge $\tilde{U_1}'$ across $\{ \alpha^*\tau=\tilde{x_1}=0 \}$ to $\tilde{U_1}$ by identifying $\tilde{U_1}'$ as an open subset of $\mathbb{R}^4$ with standard symplectic form.

Let $L_{\vartheta_0}=\{(\tau,\vartheta_0,x_2,y_2) \in \tilde{U_1}| \tau,x_2,y_2 \in \mathbb{R}\}$, which is a $3$-manifold depending on the choice of $\vartheta_0$.
Let $T$ be the tangent space of $D_2$ at $0$ and identify it as a $2$ dimensional linear subspace in $(x_1,y_1,x_2,y_2)$ coordinates. 
Then, $l_{\vartheta_0}=\alpha(L_{\vartheta_0} \cap \tilde{U_1}') \cap T$ is an open ray starting from $0$ in $U$ because $D_1$ and $D_2$ are assumed to be transversal.
If we pull back the tangent space of $l_{\vartheta_0}$ to the $(\tilde{x_1},\tilde{y_1},\tilde{x_2},\tilde{y_2})$ coordinates in $\tilde{U_1}'$, it is spanned by a vector of the form $(1,0,a_{\vartheta_0},b_{\vartheta_0})$ for some $a_{\vartheta_0},b_{\vartheta_0}$.
We identify this vector as a vector at $(0,\vartheta_0,0,0)$ and call it $v_{\vartheta_0}$.
Notice that the $\omega$-dual of $v_{\vartheta_0}$ is $d\tilde{y_1}-b_{\vartheta_0}d\tilde{x_2}+a_{\vartheta_0}d\tilde{y_2}$, for all $\vartheta_0 \in [0,2\pi]$.

Let $X_1$ be a vector field on $\tilde{U_1}$ such that $X_1=\frac{\lambda_1}{2\pi}v_{\vartheta_0}$ at $(\tilde{x_1},\tilde{y_1},\tilde{x_2},\tilde{y_2})=(0,\vartheta_0,0,0)$ for all $\vartheta_0 \in [0,2\pi]$, where $\lambda_1$ is the wrapping number of $\theta$ with respect to $D_1$.
We also require the $\omega$-dual of $X_1$ to be  a closed form on $\tilde{U_1}$.
This can be done because the $\omega$-dual of $X_1$ restricted to $\{\tilde{x_1}=\tilde{x_2}=\tilde{y_2}=0\}$ is closed.
Furthermore, we can also assume $X_1$ is invariant under the $2\pi \mathbb{Z}$ action on $\tilde{y_1}$ coordinate.
Note that $d\tilde{x_1}(X_1)=\frac{\lambda_1}{2\pi} < 0$ at $(0,\vartheta_0,0,0)$ for all $\vartheta_0$ so we have $d\tilde{x_1}(X_1) < 0$ near $\{\tilde{x_1}=\tilde{x_2}=\tilde{y_2}=0\}$.

Let the $\omega$ dual of $X_1$ be $\tilde{q_1}$, which is exact as it is closed in $\tilde{U_1}$.
Now, $\tilde{q_1}$ can be descended to a closed 1-form $q_1$ in $U_1$ under $\alpha$ with wrapping numbers $\lambda_1$ and $0$ with respect to $D_1$ and $D_2$, respectively.
We can construct another closed 1-form $q_2$ in $U_2$ in the same way as $q_1$ with $D_1$ and $D_2$ swapped around.
Notice that $q_1+q_2$ is a well-defined closed 1-form in $U-D$ with same wrapping numbers as that of $\theta$.
Let $\theta'=\theta_1+q_1+q_2$ be such that $d(\theta')=\omega$ and $\theta_1$ has bounded norm.
Since $\theta'$ has the same wrapping numbers as that of $\theta$, we can find a function $g: U-D \to \mathbb{R}$ such that $\theta'=\theta+dg=\theta_1+q_1+q_2$.

We want to show that $df(-X_{\theta+dg}) > c_f \|\theta+dg\| \| df\|$ near $D$.
It suffices to show that $df(-X_{q_1+q_2}) > c_f \|q_1+q_2\| \| df\|$ near $D$ as $\| \theta_1 \|$ is bounded.
Since $f=\sum\limits_{i=1}^n \log(\rho(r_i))+\bar{\tau}$ for some smooth $\bar{\tau}: M \to \mathbb{R}$, it suffices to show that $\sum\limits_{i=1}^2 (d\log(x_i'^2+y_i'^2)) (-X_{q_1+q_2}) > c_f \|q_1+q_2\| \| \sum_{i=1}^2 (d\log(x_i'^2+y_i'^2))\|$, where $(x_1',y_1',x_2',y_2')$ are smooth coordinates adapted to the fibrations used to define compatibility.

To do this, we pick a sequence of points $p_k \in U-D$ converging to $0$.
Then $\frac{X_{q_1}}{\| q_1 \|}$ at $p_k$ converges (after passing to a subsequence) to a vector transversal to $D_1$ but tangential to $D_2$.
The analogous statement is true for $\frac{X_{q_2}}{\| q_2 \|}$.
Hence we have $\sum\limits_{i=1}^2 (d\log(x_i'^2+y_i'^2)) (-X_{q_1+q_2}) > c_f \sum\limits_{i=1}^2 \|q_i\| \| (d\log(x_i'^2+y_i'^2))\|$ and thus get the desired estimate (See \cite{McL14} for details).

On the other hand, $c_1 \| db\| < \|  \theta+dg \| < c_2 \| db\|$ near $D$ for some smooth function $b$ compatible with $D$ is easy to achieve by taking $b=C\sum\limits_{i=1}^2 (d\log(x_i'^2+y_i'^2))$ near $D$.
\end{proof}

Careful readers will find that when constructing a convex neighborhood, the  GS construction works when wrapping numbers are all non-negative while McLean's constructions work only when wrapping numbers are all positive.
We end this subsection with a lemma saying that the GS construction is not really more powerful than McLean's construction in dimension four.

\begin{lemma}\label{non-negative wrapping numbers}
   Let $(D^{2n-2},\omega)$ be a symplectic divisor with $n>1$.
Suppose $\omega$ is exact on the boundary with $\alpha$ being a primitive on $P(D)-D$.
If the wrapping numbers of $\alpha$ are all non-negative, then all are positive.
\end{lemma}

\begin{proof}
Suppose the wrapping numbers $\lambda_i$ of $\alpha$ are all non-negative and $\lambda_1=0$.
Then, $\alpha$ can be extend over $C_1-\cup_{1 \in I, |I| \ge 2}C_I$, where we recall $C_I$ with $1 \in I$ are the symplectic submanifold of $C_1$ induced from intersection with other $C_i$.
Therefore, 
$$\int_{C_1} \omega^{n-1} = \int_{P(\cup_{1 \in I, |I| \ge 2} C_I)} \omega^{n-1} - \int_{\partial P(\cup_{1 \in I, |I| \ge 2} C_I)} \alpha \wedge \omega^{n-2},$$
where $P(\cup_{1 \in I, |I| \ge 2} C_I)$ is a regular neighborhood of $\cup_{1 \in I, |I| \ge 2} C_I$ in $C_1$.
We claim that $\int_{P(\cup_{1 \in I, |I| \ge 2} C_I)} \omega^{n-1} - \int_{\partial P(\cup_{1 \in I, |I| \ge 2} C_I)} \alpha \wedge \omega^{n-2} \le 0$ so we will arrive at a contradiction.

We first assume that if $1 \in I$, then $C_I = \emptyset$ except $C_1$ and $C_{ \{1,2 \}}$.
As a submanifold of $C_1$, $P(\cup_{1 \in I, |I| \ge 2} C_I)=P(C_{\{1,2 \}})$ can be symplectically identified with a closed $2$-disc bundle over $C_{ \{1,2 \}}$.
For each fibre, $\alpha|_{\text{fibre}}= \frac{r^2}{2}d\vartheta + \frac{\lambda_2}{2\pi} d\vartheta +df$, where $(r,\vartheta)$ is the polar coordinates of the disc and $f$ is a smooth function defined on the punctured disc.
Without loss of generality, we can assume $P(C_{\{1,2\}})$ is taken such that symplectic connection rotates the fibre and we have a well defined one form $\frac{\lambda_2}{2\pi} d\vartheta$ on $P(C_{\{1,2\}})-C_{\{1,2\}}$.
Then, $\alpha- \frac{\lambda_2}{2\pi} d\vartheta -df$ can be defined over $P(C_{\{1,2\}})$ for some $f$ defined on $P(C_{\{1,2\}})-C_{\{1,2\}}$ and 

\begin{eqnarray*}
 \int_{\partial P(C_{\{1,2\}})} (\alpha -\frac{\lambda_2}{2\pi} d\vartheta)  \wedge \omega^{n-2}
 &=&\int_{\partial P(C_{\{1,2\}})} (\alpha -\frac{\lambda_2}{2\pi} d\vartheta-df)  \wedge \omega^{n-2}\\
 &=&\int_{P(C_{\{1,2\}})} rdr \wedge d\vartheta \wedge \omega^{n-2}\\
 &=&\int_{P(C_{\{1,2\}})} \omega^{n-1}.
\end{eqnarray*}

Therefore, 
\begin{eqnarray*}
\int_{P(C_{\{1,2\}})} \omega^{n-1}- \int_{\partial P(C_{\{1,2\}})} \alpha \wedge \omega^{n-2}
&=& -\int_{\partial P(C_{\{1,2\}})} \frac{\lambda_2}{2\pi} d\vartheta  \wedge \omega^{n-2} \\
&=& -\lambda_2 \int_{C_{\{1,2\}}} \omega^{n-2} \le 0
\end{eqnarray*}

It is not hard to see that this argument can be generalized to more than two $C_I$ being non-empty, where $1 \in I$.
This completes the proof.
\end{proof}

\subsection{Obstruction}
In this subsection we prove Theorem \ref{obstruction-GS} and Theorem \ref{obstruction-closed case}.
We first prove Theorem \ref{obstruction-GS}, in which $(D,\omega)$ is assumed to be $\omega$-orthogonal. 
Then the proof for Theorem \ref{obstruction-closed case}, which is similar, is sketched.

\subsubsection{Energy Lower Bound}

Given an $\omega$-orthogonal symplectic divisor $D =C_1 \cup \dots \cup C_k$ in a 4-manifold $(W,\omega)$, 
for each $i$, let $N_i$ be a neighborhood of $C_i$ together with  a symplectic open disk fibration $p_i: N_i \to C_i$ 
 such that
the symplectic connection induced by $\omega$-orthogonal subspace of the fibers rotates the symplectic disc fibres.
Hence, for each $i$, we have a well-defined radial coordinate $r_i$ with respect to the fibration $p_i$ such that $C_i$ corresponds to $r_i=0$.
Also, $N_i$ are chosen such that the disk fibers are symplectomorphic to the standard open symplectic disk with radius $\epsilon_i$.
We also assume $\min_{i=1}^k {r_i}=r_1$ (or simply $r_1=r_2=\dots=r_k$).
Moreover, we require  
$p_{ij}: N_i \cap N_j \to C_{ij}$ to be a  symplectic $\mathbb{D}^2 \times \mathbb{D}^2$ fibration 
such  that $p_i|_{N_i \cap N_j}$ is the projection to the first factor
and $p_j|_{N_i \cap N_j}$ is the projection to the  second factor.
Such choice of $p_i$ and $N_i$ exist (See Lemma 5.14 of \cite{McL12}).

\begin{lemma}
 Let $(D=C_1 \cup \dots \cup C_k ,\omega)$ be a symplectic divisor with $p_i$ and $N_i$ as above.
 There exist an $\omega$-compatible almost complex structure $J_N$ on $N=\cup_{i=1}^k N_i$ such that $C_i$ are $J_N$-holomorphic,
 the projections $p_i$ are $J_N$-holomorphic and the fibers are $J_N$-holomorphic.
\end{lemma}

\begin{proof}
 Using $p_{ij}$, we can define a product complex structure on $\mathbb{D}^2 \times \mathbb{D}^2 = N_i \cap N_j$.
Since $p_{ij}$ are compatible with $p_i$ and $p_j$, we can extend this almost complex structure such that $J|_{C_i}$ and $J|_{C_j}$ are complex structures,
$(p_l)_*J=J|_{C_l}$ 
and $J(r_l\partial_{r_l})=\partial_{\vartheta_l}$ for $l=i,j$, where $(r_i,\vartheta_i)$ and $(r_j,\vartheta_j)$ are polar coordinates 
of the disk fiber for $p_i$ and $p_j$, respectively.

Although  $\vartheta_i$ and $\vartheta_j$ are not well-defined if the disk bundle has non-trivial Euler class, $\partial_{\vartheta_l}$ are well-defined for $k=i,j$.
Since the almost complex structure $J$ is `product-like', $J$ is compatible with the symplectic form $\omega$.
We call this desired almost complex structure $J_N$.
\end{proof}

Now, we consider a partial compactification of $N=\cup_{i=1}^k N_i$ in the following sense.
Consider a local symplectic trivialization of the symplectic disk bundle induced by $p_1$, $B_1 \times \mathbb{D}^2$,
where $B_1 \subset C_1$ is symplectomorphic to the standard symplectic closed disk with radius $\tau$.
We assume that $4\epsilon_1 < \tau$.
We recall that $\mathbb{D}^2$ is equipped with a standard symplectic form with radius $\epsilon_1$.
Choose a symplectic embedding of $\mathbb{D}^2_{\epsilon_1}$ to $S^2_{\epsilon}$ with $\epsilon$ slightly large than $\epsilon_1$, where $S^2_{\epsilon}$ is a 
symplectic sphere of area $\pi \epsilon^2$.
We glue $\cup_{i=1}^k N_i$ with $B_1 \times S^2_{\epsilon}$ along $B_1 \times \mathbb{D}^2_{\epsilon_1}$ with the identification above.
This glued manifold is called $\bar{N}$ and the compatible almost complex structure constructed above can be extended to $\bar{N}$, which we denote as $J_{\bar{N}}$.
We further require that $\{ q \} \times S^2_{\epsilon}$ is $J_{\bar{N}}$-holomorphic for every $q \in B_1$. 

We want to get an energy uniform lower bound for $J$-holomorphic curves representing certain fixed homology class, for those $J$ that are equal to $J_{\bar{N}}$ away from a neighborhood of the divisor $D$.
Let $N^{\delta} = \cup_{i=1}^k \{ r_i \le \delta \} \subset \bar{N}$, where $r_i$ are the radial coordinates for the disk fibration $p_i$.

\begin{lemma}\label{energy lower bound}
 Let $\delta_{\min} >0$ be small and $\delta_{\max}>0$ be slightly less than $\epsilon_1$.
 Let $q_{\infty} \in B_1 \times S^2_{\epsilon}$ be a point in $\bar{N}-N$ and the first coordinate of which is the center of $B_1$.
 Let $J$ be an $\omega$-compatible almost complex structure such that $J=J_{\bar{N}}$ on $\bar{N}-N^{\delta_{\min}}$.
 If $u:\mathbb{C}P^1 \to \bar{N}$ is a non-constant $J$ holomorphic curve passing through $q_{\infty}$, then 
 either $u^*\omega([\mathbb{C}P^1]) > 1.9\pi (\delta_{\max}^2 - \delta_{\max}\delta_{\min}) $ or the image of $u$ stays inside 
 $N^{\delta_{\max}} \cup B_1^{\frac{\tau}{2}} \times S^2_{\epsilon}$, where $B_1^{\frac{\tau}{2}}$ is a closed sub-disk of $B_1$
 with the same center but radius $\frac{\tau}{2}$.
\end{lemma}

\begin{proof}
 Let us assume $u^*\omega([\mathbb{C}P^1]) \le 1.9\pi (\delta_{\max}^2 - \delta_{\max}\delta_{\min}) $.
Otherwise, we have nothing to prove.
Also, we can assume $u$ intersect $\partial N^{\delta_{\min}}$  and $\partial N^{\delta_{\max}}$ transversally, by slightly adjusting $\delta_{\min}$ and
$\delta_{\max}$.
Passing to the underlying curve if necessary, we can also assume $u$ is somewhere injective.

Consider the portion of $u$ inside $Int(B_1^{\frac{\tau}{2}}) \times S^2_{\epsilon}-N^{\delta_{\min}}$.
Let $\bar{p_1}$ be the projection to the first factor.
We have $\bar{p_1} \circ u|_{u^{-1}(Int(B_1^{\frac{\tau}{2}}) \times S^2_{\epsilon}-N^{\delta_{\min}})}$ is a holomorphic map because $J=J_{\bar{N}}$ in
$Int(B_1^{\frac{\tau}{2}}) \times S^2_{\epsilon}-N^{\delta_{\min}}$ and $J_{\bar{N}}$ splits as a product.
This map is also proper.
Therefore, the map is either a surjection or a constant map.
If it is a surjection, then $u^*\omega([\mathbb{C}P^1]) > \pi (\frac{\tau}{2})^2 >1.9\pi (\delta_{\max}^2 - \delta_{\max}\delta_{\min})$.
Contradiction.
If it is a constant map, then the image of $u|_{u^{-1}(Int(B_1^{\frac{\tau}{2}}) \times S^2_{\epsilon}-N^{\delta_{\min}})}$ is the fiber.
Hence, we get $u^*\omega([\mathbb{C}P^1]) \ge \pi (\epsilon^2 - \delta_{\min}^2) > \pi (\delta_{\max}^2 - \delta_{\min}^2)$.

If the image of $u$ does not stay inside $N^{\delta_{\max}} \cup B_1^{\frac{\tau}{2}} \times S^2_{\epsilon}$,
then there is a point $q_*$ outside this region, lying inside the image of $u$ and $N^{\epsilon_1}-N^{\delta_{\max}}$.
We can assume $q_*$ is an injectivity point of $u$.
In particular, it also means that $u^{-1}(\bar{N}-N^{\delta_{\min}})$ is disconnected.
Consider the connected component $\Sigma$ of $u^{-1}(\bar{N}-N^{\delta_{\min}})$, which contains the preimage of $q_*$ under $u$.

Using one of the projections $p_i$, depending on the position of $q_*$, we can identify a neighborhood of $q_*$ as $Int(\mathbb{D}^2_{\delta_{\max}-\delta_{\min}}) \times (Int(\mathbb{D}^2_{\epsilon_1}) -\mathbb{D}^2_{\delta_{\min}})$,
where $\mathbb{D}^2_{\delta_{\min}}$ has the same center as $\mathbb{D}^2_{\epsilon_1}$ and they are closed disks with radii $\delta_{\min}$ and $\epsilon_1$, respectively.
We call this neighborhood $N_{q_*}$.
Also, we still have $J=J_{\bar{N}}$ and $J_{\bar{N}}$ splits as a product in $N_{q_*}$.
Similar as before, by projection to the factors, we see that 
$\int_{\Sigma \cap u^{-1}(N_{q_*})} u^*\omega \ge \min \{ \pi(\delta_{\max}-\delta_{\min})^2, \pi \epsilon_1^2-\pi \delta_{\min}^2 \}$.
Therefore, we have
$$u^*\omega([\mathbb{C}P^1]) \ge
\int_{u^{-1}(Int(B_1^{\frac{\tau}{2}}) \times S^2_{\epsilon}-N^{\delta_{\min}})} u^*\omega+\int_{\Sigma \cap u^{-1}(N_{q_*})} u^*\omega > 1.9\pi (\delta_{\max}^2 - \delta_{\max}\delta_{\min})$$
Contradiction.
\end{proof}

\subsubsection{Theorem \ref{obstruction-GS} and Theorem \ref{obstruction-closed case}}

We recall the terminology {\it GW triple} used in \cite{McL14}.
For a symplectic manifold $(W,\omega)$ (possibly non-compact), a homology class $[A] \in H_2(W;\mathbb{Z})$ and a family of compatible almost complex structures $\mathcal{J}$ such that

(1) $\mathcal{J}$ is non-empty and path connected.

(2) there is a relative compact open subset $U$ of $W$ such that for any $J \in \mathcal{J}$, any compact genus $0$ nodal $J$-holomorphic curve representing the class $[A]$ lies inside $U$.

(3) $c_1(TW)([A])+n-3=0$.

$GW_0(W,[A],\mathcal{J})$ is called a GW triple.
The key property of a GW triple is the following.

\begin{prop}[cf \cite{McL14} and the references there-in]  \label{GW}
Suppose $GW_0(W,[A],\mathcal{J})$ is a GW triple.
Then, the GW invariants $GW_0(W,[A],J_0)$ and $GW_0(W,[A],J_1)$ are the same for any $J_0,J_1 \in \mathcal{J}$.
In particular, if $GW_0(W,[A],J_0) \neq 0$, then for any $J \in \mathcal{J}$, there is a nodal closed genus $0$ $J$-holomorphic curve representing the class $[A]$.
\end{prop}
   
One more technique that we need to use is usually called neck-stretching (See \cite{BEHWZ03} and the references there-in).
Given a contact hypersurface $Y \subset W$ separating $W$ with Liouville flow $X$ defined near $Y$.
We call the two components of $W-Y$ as $W^-$ and $W^+$, where $W^-$ is the one containing $D$.
Then, $Y$ has a tubular neighborhood of the form $(-\delta,\delta)\times Y$ induced by $X$, which can be identified as part of the symplectization of $Y$.
By this identification, we can talk about what it means for an almost complex structure to be translation invariant and cylindrical in this neighborhood.
If one choose a sequence of almost complex structures $J_i$ that "stretch the neck" along $Y$ and a sequence of closed $J_i$-holomorphic curve $u_i$ with the same domain such that there is a uniform energy bound, then $u_i$ will have a subseguence 'converge' to a $J_{\infty}$-holomorphic building.
The fact that we need to use is the following.

\begin{prop}[cf \cite{BEHWZ03}, \cite{McL14} and the references there-in] \label{SFT}
 Suppose we have a sequence of $\omega$-compatible almost complex structure $J_i$ and a sequence of nodal closed genus $0$ $J_i$-holomorphic maps $u_i$ to $W$ representing the same homology class in $W$
 such that the image of $u_i$ stays inside a fixed relative compact open subset of $W$.
 Assume $J_i$ stretch the neck along a separating contact hypersurface $Y \subset W$ with respect to a Liouville flow $X$ defined near $Y$.
 Assume that the image of $u_i$ has non-empty intersection with $W^-$ and $W^+$, respectively, for all $i$. 
 Then, there are proper genus $0$ $J_{\infty}$-holomorphic maps (domains are not compact) $u_{\infty}^-:\Sigma^- \to W^-$ and $u_{\infty}^+:\Sigma^+ \to W^+$ such that $u_{\infty}^-$ and $u_{\infty}^+$ are asymptotic to Reeb orbits on $Y$ with respect to the contact form $\iota_X \omega$.
\end{prop}

In our notations, $u_{\infty}^-$ and $u_{\infty}^+$ are certain irreducible components of the top/bottom buildings but not necessary the whole top/bottom buildings.
Also, $u_{\infty}^-$ does not necessarily refer to the bottom building because we do not declare the direction of the Liouville flow near $Y$.
We are finally ready to prove Theorem \ref{obstruction-GS}.
The following which we are going to prove implies Theorem \ref{obstruction-GS}.

\begin{thm}\label{obstruction}
 Let $D \subset (W,\omega)$ be an $\omega$-orthogonal symplectic divisor with area vector $a=(\omega[C_1], \dots, \omega[C_k])$.
Let $z=(z_1,\dots,z_k)$ be a solution of $Q_Dz=a$.
If one of the $z_i$ is non-positive (resp. positive),
there is a small neighborhood $N^{\delta_{\min}} \subset W$ of $D$ such that there is no plumbing $(P(D),\omega|_{P(D)}) \subset (N^{\delta_{\min}},\omega|_{N^{\delta_{\min}}})$
of $D$ being a capping (resp. filling) of its boundary $(\partial P(D),\alpha)$ with $\alpha$ being the contact form, where $\alpha$ is any primitive of $\omega$ defined near $\partial P(D)$ with wrapping numbers $-z$. 
\end{thm}

\begin{figure}\label{fig Liouville flow}
\begin{center}
\begin{tikzpicture}[scale=1.5]
    \draw [<->,thick] (0,2) node (yaxis) [above] {$y$}
        |- (2,0) node (xaxis) [right] {$x$};
    \draw [<->,thick] (0,-2) |- (-2,0);

    \draw (1,0) .. controls (2,1) and (-0.3,0.2) .. (-1.3,-0.1);
    \draw (-1.3,-0.1) .. controls (-1.5,-0.3) and (-0.5,-2) .. (1,0);
    
    \draw (-0.5,-1.2) node {$(u_{\infty}^-)^{-1}(\partial_{\eta} P(D))$};
    
    \draw [->] (0.5,0.5) -- (0.4,0.3);
    \draw [->] (-1,0.2) -- (-0.8,-0.2);
    \draw [->] (0.5,-0.8) -- (0.4,-0.5);
    \draw [->] (0.05,0.05) -- (0.15,0.15);
    \draw [->] (0.05,-0.05)-- (0.15,-0.15);
    \draw [->] (-0.05,-0.05) -- (-0.15,-0.15);
    \draw [->] (-0.05,0.05) -- (-0.15,0.15);
\end{tikzpicture}
\end{center}
Figure \ref{fig Liouville flow}

The direction of the arrows indicates the direction of the Liouville flow $X_{\Sigma^-}$ and $q_0$ is identified with the origin.
This is a schematic picture and we do not claim that $(u_{\infty}^-)^{-1}(\partial_{\eta} P(D))$ is connected.
\end{figure}

\begin{proof}
 We first prove the case that one of the $z_i$ is non-positive.
Without loss of generality, assume $z_1 \le 0$.
We use the notation in Lemma \ref{energy lower bound}.
In particular, we have symplectic disk fibration $p_i:N_i \to C_i$, the partial compactification $\bar{N}$ and its $\omega$-compatible almost complex structure $J_{\bar{N}}$.
We also have $N^{\delta} = \cup_{i=1}^k \{ r_i \le \delta \} \subset \bar{N}$ and so on.
We want to prove the statement with $N^{\delta_{\min}}$ being a small neighborhood of $D$, where $\delta_{\min}$ is so small such that $1.9\pi (\delta_{\max}^2 - \delta_{\max}\delta_{\min}) > \pi \epsilon^2$.
We recall that we have $\epsilon$ slightly larger than $\epsilon_1$ and $\epsilon_1$ is slightly larger than $\delta_{\max}$.
We also recall that we have $B_1 \times S^2_{\epsilon} \subset \bar{N}$ and $\{q\} \times S^2_{\epsilon}$ has symplecitc area $\pi \epsilon^2$ for any $q \in B_1$. 

Suppose the contrary, assume $(P(D),\omega|_{P(D)}) \subset (N^{\delta_{\min}},\omega|_{N^{\delta_{\min}}})$ caps its boundary $(\partial P(D),\alpha)$.
Let $X$ be the corresponding Liouville flow near $\partial P(D)$.
We do a small symplectic blow-up centered at $q_{\infty}$ and this blow-up is so small that it is done in $\bar{N}-N$.
We call this blown-up manifold $(\bar{N}',\omega_{\bar{N}'})$ and pick a $\omega_{\bar{N}'}$-compatible almost complex structure $J_{\bar{N}'}$ such that the blow-down map is $(J_{\bar{N}'},J_{\bar{N}})$-holomorphic, the exceptional divisor is $J_{\bar{N}'}$-holomorphic  and  $J_{\bar{N}'}=J_{\bar{N}}$ in $N$.
Let the exceptional divisor be $E$ and the proper transform of the sphere fiber in $B_1 \times S^2_{\epsilon}$ containing $q_{\infty}$ be $A$.
We have $GW_0(\bar{N}',[A],J_{\bar{N}'})=1$ by automatic transversality or argue as in the end of Step $4$ of the proof of Theorem 6.1 in \cite{McL14} for higher dimensions.
Since blow-up decreases the area, $\omega_{\bar{N}'}([A])<\omega_{\bar{N}}(Bl_*[A])=\pi \epsilon^2$, where $Bl$ is the blow-down map.
This gives us the energy upper bound.
By the same argument as in Lemma \ref{energy lower bound}, we have that for any $\omega_{\bar{N}'}$-compatible almost complex structure $J$ such that $J=J_{\bar{N}'}$ on $\bar{N}'-N^{\delta_{\min}}$, any (nodal) $J$-holomorphic curve representing the class $[A]$ stays inside a fixed relative compact open subset of $\bar{N}'$.
As a result, we have a GW triple $GW_0(\bar{N}',[A],\mathcal{J})$, where $\mathcal{J}$ is the family of compatible almost complex structure that equals $J_{\bar{N}'}$ on $\bar{N}'-N^{\delta_{\min}}$.
By Proposition \ref{GW}, we have a nodal closed genus $0$ $J$-holomorphic curve for any $J \in \mathcal{J}$.

Now, since $P(D) \subset N^{\delta_{\min}}$, we can choose a sequence $J_i$ in $\mathcal{J}$ such that it stretches the neck along $(\partial P(D),\alpha)$ and $J_i=J_{\bar{N}'}$ very close to $D$.
We have a corresponding sequence of nodal closed genus $0$ $J_i$-holomorphic curve $u_i$ to $\bar{N}'$.
By Proposition \ref{SFT}, we have a proper genus $0$ $J_{\infty}$-holomorphic maps $u_{\infty}^-:\Sigma^- \to Int(P(D))$ such that $u_{\infty}^-$ is asymptotic to Reeb orbits on $\partial P(D)$ with respect to the contact form $\alpha$.
By the direction of the flow, $u_{\infty}^-$ corresponds to the top building.
In general, the top building can be reducible.
In our case, since $[A]\cdot [C_1]=1$ and $ [A] \cdot [C_i]=0$ for $i=2,\dots,k$, if the top building is reducible, there is some irreducible component lying inside $Int(P(D))-D$, by positivity of intersection and $D$ being $J_{\infty}$-holomorphic.
Since $\omega_{\bar{N}'}$ is exact on $Int(P(D))-D$, any irreducible component lying inside $Int(P(D))-D$ most have non-compact domain and converge asymptotically to Reeb orbits on 
$Y$.
By the direction of the Reeb flow, we get a contradiction by Stoke's theorem.
(cf. Proposition 8.1 of \cite{McL14} or Step 3 of proof of Theorem 6.1 in \cite{McL14} or Lemma 7.2 of \cite{AbSe10})
Therefore, we conclude that there is only one irreducible component which is exactly $u_{\infty}^-$ and the image of $u_{\infty}^-$ intersect $C_1$ transversally once.
Let $q_0 \in \Sigma^-$ be the point that maps to the intersection.
Let also $\mathbb{D}^2_{q_0}$ be a Darboux disk around $q_0$ and $i=u_{\infty}^-|_{\mathbb{D}^2_{q_0}}$.
Now, we want to draw contradiction using the existence of $u_{\infty}^-$.

Since $\omega$ is exact on $\partial P(D)$, it is exact in $N^{\delta_{\min}}-D$.
Extend $\alpha$ to be a primitive of $\omega$ in $N^{\delta_{\min}}-D$ and we still denote it as $\alpha$.
By assumption, $[\omega-d \alpha_c]$ is Lefschetz dual to $\sum\limits_{i=1}^k z_i[C_i]$.
In particular, $i^* \alpha$ has wrapping number $-z_1$ around $q_0$ on $\mathbb{D}^2_{q_0}-q_0$.
In other words, $[i^*\alpha-\frac{r^2}{2}d\theta]$ is cohomologous to $\frac{-z_1}{2\pi} d\theta$ in $H^1(\mathbb{D}^2_{q_0}-q_0,\mathbb{R})$, where $(r,\theta)$ are the polar coordinates.
We have $$ i^*\alpha=\frac{r^2}{2}d\theta+\frac{-z_1}{2\pi} d\theta+df$$ for some function $f$ on $\mathbb{D}^2_{q_0}-q_0$.
When we choose the extension of $\alpha$ to $N^{\delta_{\min}}-D$, we can choose in a way that $i^*\alpha=\frac{r^2}{2}d\theta+\frac{-z_1}{2\pi} d\theta$ because the image of $i$ is away from $\partial P(D)$.

Notice that the $(u_{\infty}^-)^*\omega$ dual of $(u_{\infty}^-)^*\alpha$ defines a Liouville vector field $X_{\Sigma^-}$ on $\Sigma^- -q_0$ away from critical points of $u_{\infty}^-$ (if any).
This Liouville flow equals to the component of $X$ in $T(u_{\infty}^-(\Sigma^-))$ near $\partial P(D)$, when we write down the decomposition of $X$ into  $T(u_{\infty}^-(\Sigma^-))$-component and its $\omega_{\bar{N}'}$-orthogonal complement component.
Here, $T(u_{\infty}^-(\Sigma^-))$ denotes the tangent bundle of the image of $u_{\infty}^-$, which is well-defined near $\partial P(D)$.
Since $u_{\infty}^-$ is $J_{\infty}$-holomorphic and it asymptotic converges to Reeb orbits of $(\partial P(D),\alpha)$, $X=-J_{\infty}R_{\alpha}$ has non-zero $T(u_{\infty}^-(\Sigma^-))$-component near $\partial P(D)$.
Moreover, since $X$ points inward with respect to $P(D)$, so is $X_{\Sigma^-}$ on $\Sigma^-$ near infinity.
In particular, if we take $\partial_{\eta}P(D)$ to be the flow of $\partial P(D)$ with respect to $X$ for a sufficiently small time, then $X_{\Sigma^-}$ is pointing inward along $(u_{\infty}^-)^{-1}(\partial_{\eta} P(D))$.

On the other hand, $i^*\alpha=\frac{r^2}{2}d\theta+\frac{-z_1}{2\pi} d\theta$.
Therefore, the Liouville vector field $X_{\Sigma^-}$ near $q_0$ equals $(\frac{r}{2}+\frac{-z_1}{2\pi r}) \partial_r$ and hence points outward with respect to $\mathbb{D}^2_{q_0}$ (This is where we use $z_1 \le 0$).  
As a result, we get a compact codimension $0$ submanifold $\Sigma^-_0$ with boundary in $\Sigma^--q_0$ that has Liouville flow pointing inward along the boundaries and $(u_{\infty}^-|_{\Sigma^-_0})^*\omega$ has a globally defined primitive $(u_{\infty}^-|_{\Sigma^-_0})^*\alpha$.
It gives a contradiction by Stoks's theorem and $\int_{\Sigma^-_0} (u_{\infty}^-)^*\omega \ge 0$.
See Figure \ref{fig Liouville flow}.

For the other case, we assume $z_1$ is positive.
In this case, the argument is basically the same but we need to use $u_{\infty}^+$ instead of $u_{\infty}^-$.
We have $u_{\infty}^+$ intersecting the exceptional divisor $E$ transversally exactly once.
By blowing down, we get a corresponding pseudo-holomorphic map $Bl \circ u_{\infty}^+ $.
Let the point on $\Sigma^+$ that maps to the exceptional divisor be $q_0$ as before.
Then, the $(u_{\infty}^+)^*\omega$ dual of $(u_{\infty}^+)^*\alpha$ defines a Liouville vector field on $\Sigma^+ -q_0$ away from critical points of $u_{\infty}^+$ (if any).
Similar as before, we get a contradiction by Stoke's theorem and the fact that $u_{\infty}^+$ restricted to any subdomian has non-negative energy.
This completes the proof.

\end{proof}

We remark that the same proof can be generalized to higher dimensional $\omega$-orthogonal divisors.
The only thing that need to be changed is to use Monotonicity lemma to get the energy lower bound in Lemma \ref{energy lower bound} instead of using surjectivity.
We leave it to interested readers.

On the other hand, this is not obvious to the authors that how one can remove the $\omega$-orthogonal assumption in Theorem \ref{obstruction}.
Therefore, Theorem \ref{obstruction} is not enough for our application.
To deal with this issue, we have the following Theorem \ref{obstruction-closed case}.

\begin{thm}\label{obstruction-closed case}
Let $(D,\omega)$ be a symplectic divisor in a closed symplectic manifold $(W,\omega)$.
There exists a neighborhood $N$ of $D$ such that there is no concave neighborhood $P(D)$ inside $N$ with the Liouville form $\alpha$ on $\partial P(D)$ having a non-negative wrapping number among its wrapping numbers.   
\end{thm}

The proof is in the same vein of that of Theorem \ref{obstruction}.
However, we need $W$ to be closed to help us to run the argument this time.

\begin{figure}\label{fig obstruction closed}
\begin{tikzpicture}[scale=1.5]
    \draw [<->,thick] (-3.5,0) -- (3.5,0) node (xaxis) [right] {$C_1$};
    \draw  (-2,0) node [below] {$4\tau$} -- (-2,2);
    \draw  (2,0) node [below] {$4\tau$} -- (2,2);
    \draw [-,line width=1.5pt] (-2,0.2) -- (-2,2);
    \draw [-,line width=1.5pt] (2,0.2) -- (2,2);
    \draw [-,line width=1.5pt] (-2,2) -- (2,2);
    \draw [-,line width=1.5pt] (-2,0.2) -- (2,0.2);

    \draw (-1.5,1.3) node [left] {$N_1$};
    
    \draw [-,line width=2pt](-1.5,0.8) -- (-1.5,1.7);
    \draw [-,line width=2pt](1.5,0.8) -- (1.5,1.7);
    \draw [-,line width=2pt](-1.5,1.7) -- (1.5,1.7);
    \draw [-,line width=2pt](-1.5,0.8) -- (-0.25,0.8);
    \draw [-,line width=2pt](0.25,0.8) -- (1.5,0.8);
    \draw (-0.25,0.8) -- (-0.25,0.8);
    
    \draw [-,line width=2pt](-1,1) -- (-1,1.5);
    \draw [-,line width=2pt](1,1) -- (1,1.5);
    \draw [-,line width=2pt](-1,1.5) -- (1,1.5);
    \draw (-1,1) -- (1,1);
    \draw [-,line width=2pt] (-1,1) -- (-0.5,1);
    \draw [-,line width=2pt] (0.5,1) -- (1,1);
    \draw (-1,1.3) node [left] {$N_2$};
    
    \draw [-,line width=2pt](-0.5,1) -- (-0.5,1.2);
    \draw [-,line width=2pt](0.5,1) -- (0.5,1.2);
    \draw (-0.5,1.2) -- (0.5,1.2);
    
    \draw [-,line width=2pt](-0.25,0.8) -- (-0.25,1.2);
    \draw [-,line width=2pt](0.25,0.8) -- (0.25,1.2);
    \draw [-,line width=2pt] (-0.5,1.2) -- (-0.25,1.2);
    \draw [-,line width=2pt] (0.25,1.2) -- (0.5,1.2);
    \draw (0,1.35) node {$q_{\infty}$};
    \draw (0,1.2) node {$\ast$};

    \draw [dashed] (-1.5,0.8) --(-1.5,0) node [below] {$3\tau$};
    \draw [dashed] (1.5,0.8) --(1.5,0) node [below] {$3\tau$};
    \draw [dashed] (-1,1) --(-1,0) node [below] {$2\tau$};
    \draw [dashed] (1,1) --(1,0) node [below] {$2\tau$};
    \draw [dashed] (-0.5,1) --(-0.5,0) node [below] {$\tau$};
    \draw [dashed] (0.5,1) --(0.5,0) node [below] {$\tau$};
    \draw [dashed] (-0.25,0.8) --(-0.25,0) node [below] {$\frac{\tau}{2}$};
    \draw [dashed] (0.25,0.8) --(0.25,0) node [below] {$\frac{\tau}{2}$};

    \draw [dashed] (2,0.2) --(2.5,0.2) node [right] {$\epsilon_5$};
    \draw [dashed] (1.5,0.8) --(2.5,0.8) node [right] {$\epsilon_4$};
    \draw [dashed] (1,1) --(2.5,1) node [right] {$\epsilon_3$};
    \draw [dashed] (0.5,1.2) --(2.5,1.2) node [right] {$\epsilon_2$};
    \draw [dashed] (1,1.5) --(2.5,1.5) node [right] {$\epsilon_1$};
    \draw [dashed] (1,1.7) --(2.5,1.7) node [right] {$\epsilon_0$};
    \draw [dashed] (2,2) --(2.5,2) node [right] {$\epsilon$};

\end{tikzpicture}

Figure \ref{fig obstruction closed}:
The bottom line represents $C_1$.
The region enclosed by thickest solid lines represents $N_2$.
The region enclosed by less thick solid lines represents $N_1$
For any $J=J_{\overline{W}}$ on $N_1$,
every closed genus $0$ $J$-holomorphic curve $u$ passing through $q_{\infty}$ stays away from $N_2$ if energy of $u$ is sufficiently small.
\end{figure}

\begin{proof}
Same as before, we start with an energy estimate.
Suppose $\alpha$ is a primitive of $\omega$ defined near $D$ but not defined on $D$.
Let the wrapping numbers of $\alpha$ be $-z_i$.
Suppose $z_1$ is non-positive.

Identify a neighborhood $M_1$ of $C_1$ with a symplectic disk bundle over $C_1$ with symplectic connection rotating the fibers.
Assume the fibers are symplecticomorphic to standard symplectic disk of radius $\epsilon$.
Pick a point $p$ on $C_1$ such that there is a Darboux disk of radius $4\tau$ and $\tau > c\epsilon$ for some $c>0$ to be determined (we can achieve this by choosing a small $\epsilon$ in advance).
Identify the neighborhood of $p$ as a product of closed disks $\mathbb{D}^2_{4\tau} \times \mathbb{D}^2_{\epsilon}$.
Choose $\epsilon_i$ for $i=0,1,\dots,5$ to be determined such that $\epsilon > \epsilon_0 > \epsilon_1 > \dots > \epsilon_5 > 0$.
Now, cut out the closed region $\mathbb{D}^2_{2\tau} \times (\mathbb{D}^2_{\epsilon_1}-Int(\mathbb{D}^2_{\epsilon_3}))$ from $W$ and call it $W_0$.
We partially compactify $W_0$ to be $\overline{W}$ by gluing $W_0$ and $Int(\mathbb{D}^2_{\tau}) \times S^2_{\epsilon_2}$ along $Int(\mathbb{D}^2_{\tau}) \times Int(\mathbb{D}^2_{\epsilon_3})$ by identifying $Int(\mathbb{D}^2_{\epsilon_3})$ with a choose of symplectic embedding to $S^2_{\epsilon_2}$, where $S^2_{\epsilon_2}$ is a symplectic sphere of symplectic area $\pi \epsilon_2^2$.
We define $N_3 \subset N_2 \subset N_1 \subset N \subset \overline{W}$ to be the following subset of $\overline{W}$.
Notice that $N$, $N_1$ and $N_2$ are closed (but not compact) and $N_3$ is open.

$$N=(\mathbb{D}^2_{4\tau} \times \mathbb{D}^2_{\epsilon}-\mathbb{D}^2_{2\tau} \times (\mathbb{D}^2_{\epsilon_1}-Int(\mathbb{D}^2_{\epsilon_3})))\cup Int(\mathbb{D}^2_{\tau}) \times S^2_{\epsilon_2} $$

$$N_1=(\mathbb{D}^2_{4\tau} \times (\mathbb{D}^2_{\epsilon}-Int(\mathbb{D}^2_{\epsilon_5}))-\mathbb{D}^2_{2\tau} \times (\mathbb{D}^2_{\epsilon_1}-Int(\mathbb{D}^2_{\epsilon_3})))\cup Int(\mathbb{D}^2_{\tau}) \times S^2_{\epsilon_2}$$

\begin{eqnarray*}
N_2&=&((\mathbb{D}^2_{3\tau} \times (\mathbb{D}^2_{\epsilon_0}-Int(\mathbb{D}^2_{\epsilon_4}))-\mathbb{D}^2_{2\tau} \times (\mathbb{D}^2_{\epsilon_1}-Int(\mathbb{D}^2_{\epsilon_3}))-Int(\mathbb{D}^2_{\frac{\tau}{2}}) \times Int(\mathbb{D}^2_{\epsilon_3}))    \\
&&\cup (Int(\mathbb{D}^2_{\tau})-Int(\mathbb{D}^2_{\frac{\tau}{2}})) \times S^2_{\epsilon_2}
\end{eqnarray*}

$$N_3=Int(\mathbb{D}^2_{\tau}) \times S^2_{\epsilon_2}-Int(\mathbb{D}^2_{\tau}) \times \mathbb{D}^2_{\epsilon_3}$$
Then, $\overline{W}$ is our desired manifold to run the argument above (See Figure \ref{fig obstruction closed}).

Let $J_{\overline{W}}$ be an $\omega$-compatible almost complex structure on $\overline{W}-N_3$ such that $J_{\overline{W}}$ is split as a product in $N-N_3$.
Then, extend $J_{\overline{W}}$ naturally over $N_3$ such that it is 'product-like' making the $S^2_{\epsilon_2}$ sphere fibers $J_{\overline{W}}$-holomorphic.
We still call this $J_{\overline{W}}$.
Let $q_{\infty} \in N_3$ be a point in $N$ such that it lies in the $S^2_{\epsilon_2}$ sphere fiber at $p$.
Similar as before, for any $\omega$-compatible almost complex structure $J$ such that $J=J_{\overline{W}}$ on $N_1$ 
and any closed genus $0$ $J$-holomorphic curve $u$ to $\overline{W}$ passing through $q_{\infty}$, we must have the image of $u$ stays inside $\overline{W}-N_2$ or 
the energy of $u$, $\int_{\mathbb{C}P^1}u^*\omega$, greater than a lower bound depending on $\epsilon_5$ (once $c$ and $\epsilon_i$ for $i=0,\dots,4$ are determined).
It should be convincing that one can choose a choice of $c$ and $\epsilon_i$ such that
any $J$-holomorphic curve representing the class $[S^2_{\epsilon_2}]$, the spherical fiber class at $p$, and passing through $q_{\infty}$ has to stay inside $\overline{W}-N_2$.
Since $W$ is closed, $\overline{W}-N_2$ is a relative compact open subset which we can use to define the GW triple below.

We claim that $D$ does not have a concave neighborhood $P(D)$ lying inside $W-N_1$ with the Liouville contact form $\alpha'$ defined near $\partial P(D)$ having the same wrapping numbers as that of $\alpha$.
Suppose on the contrary, there were such a $P(D)$.
Then, by a $C^0$ perturbation near the intersection points of $C_i$ in $D$, we can assume that $D$ is $\omega$-orthogonal and it still lies inside $P(D)$.
We do a sufficiently small blow-up at $q_{\infty}$ as before and let the proper transform of the $S^2_{\epsilon_2}$ fiber containing $q_{\infty}$ be $A$.
We have a GW triple $GW_0(\overline{W},[A],\mathcal{J})$, where $\mathcal{J}$ is the family of $\omega$-compatible almost complex structure $J$ such that $J=J_{\overline{W}}$ on $N_1$.
Since $D$ is now $\omega$-orthogonal, we can find $J\in \mathcal{J}$ such that $D$ is $J$-holomorphic (notice that, there exists symplectic divisor with no almost complex structure making all irreducible components pseudo-holomorphic simultaneously).
Then, we find a sequence $J_i \in \mathcal{J}$ making $D$ $J_i$-holomorphic for all $i$ and stretch the neck along $\partial P(D)$ as before to draw contradiction.

\end{proof}

\subsection{Uniqueness}
In this subsection we 
show  that any contact structure obtained from the GS construction is contactomorphic to one from the McLean's construction.
Then  Theorem \ref{uniqueness-GS} follows from the uniqueness of McLean's construction.

In fact, we are going to prove the following   more precise version  of Theorem \ref{uniqueness-GS}.

\begin{prop}\label{canonical contact structure_proof}
 Suppose $D=\cup_{i=1}^k C_i$ is a symplectic divisor with each intersection point being $\omega$-orthogonal such that the augmented graph $(\Gamma,a)$ satisfies the positive (resp. negative) GS criterion.
Then, the contact structures induced by the positive (resp. negative) GS criterion are contactomorphic, independent of choices made in the construction 
and independent of $a$ as long as $(\Gamma,a)$ satisfies positive GS criterion.

Moreover, if $D$  arises from resolving an isolated normal surface singularity, then the contact structure induced by the negative GS criterion is contactomorphic to the contact structure induced by the complex structure. 

On the other hand, if $D$ is the support of an effective ample line bundle, then the contact structure induced by the positive GS criterion is contactomorphic to that induced by a positive hermitian metric on the ample line bundle.
\end{prop}

\subsubsection{Uniqueness of McLean's construction}

We recall the uniqueness part of McLean's construction, which can be regarded as a more complete version of Proposition \ref{McLean0}.

\begin{prop}\label{McLean}[cf. Corollary 4.3 and Lemma 4.12 of \cite{McL14}]
 Suppose $f_0,f_1: W-D \to \mathbb{R}$ are compatible with $D$ and $D$ is a symplectic divisor with respect to both $\omega_0$ and $\omega_1$ having positive transversal intersections.
Suppose $\theta_j \in \Omega^1(W-D)$ is a primitive of $\omega_j$ on $W-D$ such that it has positive (resp. negative) wrapping numbers for all $i=1, \dots, k$ and for both $j=0,1$.
Suppose, for both $j=0,1$, there exist $g_j: W-D \to \mathbb{R}$ such that
$df(X^j_{\theta_j+dg_j}) > 0$ (resp. $df(-X^j_{\theta_j+dg_j}) > 0$) near $D$, where $X^j_{\theta_j+dg_j}$ is the dual of $\theta_j+dg_j$ with respect to $\omega_j$.
Then, for sufficiently negative $l$, we have that $(f_0^{-1}(l),\theta_0+dg_0|_{f_0^{-1}(l)})$ is contactomorphic to $(f_1^{-1}(l),\theta_1+dg_1|_{f_1^{-1}(l)})$.

Moreover, when $(W,D,\omega)$ arises from resolving a normal isolated surface singularity, then the link with contact structure induced from complex line of tangency is contactomorphic to this canonical contact structure.
\end{prop}

The first thing to note is that the choice of $g_j$ for $j=0,1$ always exist (cf. Proposition \ref{McLean0} above, Proposition 4.1 and Proposition 4.2 in \cite{McL14}).
Moreover, by the definition of compatible function, it also always exist.
In other words, Proposition \ref{McLean} implies that in dimension four, for any symplectic form $\omega_0$ and $\omega_1$ making $D$ a divisor such that they have primitives
$\theta_0$ and $\theta_1$ on $W-D$ with positive (resp. negative) wrapping numbers, the contact structures constructed by McLean's construction with respect to $\theta_0$
and $\theta_1$ are contactomorphic.

Proposition \ref{McLean} is literally not exactly the same as Corollary 4.3 and Lemma 4.12 in \cite{McL14}
so we want to make clear why it is still valid after we have made the changes.
We remark that if $\theta_0 $ and $\theta_1$ have positive wrapping numbers, then $\theta_t=(1-t)\theta_0+t\theta_1$ has positive wrapping numbers for all $t$ and $f_t=(1-t)f_0+tf_1$ is compatible with $D$ for all $t$.
As a symplectic divisor, we always assume $C_i$ have positive orientations with respect to the symplectic form for all $i$.
In other words, both $\omega_0|_{C_i}$ and $\omega_1|_{C_i}$ are positive and hence $D$ is a symplectic divisor with respect to $d\theta_t$ for all $t$. 
Therefore, we get a deformation of $\omega_t$ and the first half of Proposition \ref{McLean} with $\theta_0 $ and $\theta_1$ having positive wrapping numbers follows from Corollary 4.3 of \cite{McL14}.

The analogous statement for the first half of Proposition \ref{McLean} with $\theta_0 $ and $\theta_1$ having negative wrapping numbers follows similarly as in the case where $\theta_0 $ and $\theta_1$ have positive wrapping numbers.

On the other hand, Lemma 4.12 of \cite{McL14} requires that the resolution is obtained from blowing up.
Although there exist a resolution such that it is not obtained from blowing up in complex dimension three or higher, every resolution for an isolated normal surface singularity can be obtained by blowing up the unique minimal model, where the minimal model is obtained from blowing up the singularity.
Therefore, the second half of Proposition \ref{McLean} follows.

\subsubsection{Proof of Theorem \ref{uniqueness-GS}}

To prove Proposition \ref{canonical contact structure_proof} using Proposition \ref{McLean}, the remaining task is to construct an appropriate disc fibration having a connection rotating fibers for the local models in the GS-construction.
Then, the constructions of $\theta$, $f$ and $g$ will be automatic.
We give the fibration in the following Lemma.

\begin{lemma}\label{canonical contact structure_tech}
Let $z_1'$ and $z_2'$ be two positive numbers.
Let $\mu: \mathbb{S}^2 \times \mathbb{S}^2 \to [z_1',z_1'+1] \times [z_2',z_2'+1]$ be the moment map of $\mathbb{S}^2 \times \mathbb{S}^2$ onto its image.

Fix a small $\epsilon >0$ and let $D_1=\mu^{-1}(\{ z_1'\} \times [z_2', z_2' + 2\epsilon])$ be a symplectic disc.
Fix a number $s \in \mathbb{R}$ first and then let $\delta>0$ be sufficiently small.
Let $Q$ be the closed polygon with vertices $(z_1',z_2'), (z_1'+\delta,z_2'), (z_1'+\delta,z_2'+2\epsilon-s\delta), (z_1',z_2'+2\epsilon)$.
Using the $(p_1,q_1,p_2,q_2)$ coordinates described in the GS-construction above, we define a map $\pi: \mu^{-1}(Q) \to D_1$ by sending
$(p_1,q_1,p_2,q_2)$ to $(z_1',*,p_2+\frac{(2\epsilon-t(p_1,p_2)-\rho(t(p_1,p_2)))p_1}{\delta}, q_2)$, where $\rho: [0, 2\epsilon] \to [0,2\epsilon-s\delta]$ is a smooth strictly monotonic decreasing function with $\rho(0)=2\epsilon-s\delta$ and $\rho(2\epsilon)=0$ such that $\rho'(t)=-1$ for $t\in [0, \epsilon]$ and near $t=2\epsilon$.
This can be done as $\delta$ is sufficiently small. 
Moreover, $t(p_1,p_2)$ is the unique $t$ solving $p_2-(z_2'+2\epsilon-t)=\frac{(\rho(t)-(2\epsilon-t))(p_1-z_1')}{\delta}$ and $*$ means that there is no $q_1$ coordinate above $(z_1',x)$ for any $x$ so $q_1$ coordinate is not relevant.

Then, we have that $\pi$ gives a symplectic fibration with each fibre symplectomorphic to $(\mathbb{D}^2_{\sqrt{2\delta}},\omega_{std})$ and the symplectic connection of $\pi$ has structural group lies inside $U(1)$. 
Moroever, fibres are symplectic orthogonal to the base.

\end{lemma}

\begin{figure}\label{fig disk fibration}
\begin{tikzpicture}[scale=1.5]
    \draw [<->,thick] (0,4.5) node (yaxis) [above] {$y$}
        |- (6,0) node (xaxis) [right] {$x$};
    \draw (1.5,0.9) coordinate (a_1) -- (4.5,0.9) coordinate (a_2);
    \draw (1.5,0.9) coordinate (b_1) -- (1.5,3.9) coordinate (b_2);
    \draw (2.4,1.8) coordinate (d_1)-- (5.4,1.8) coordinate (d_2);   
    \draw (2.4,0.9) coordinate (e_1)-- (2.4,3) coordinate (e_2);
    \draw (2.2,2.5) node [left] {$Q$};
    \draw (2.4,3) coordinate (j_1)-- (2.4,4.3) coordinate (j_2);

    \draw (1.5,3.9) coordinate (f_1)-- (2.4,4.3) coordinate (f_2); 
    \draw (1.5,2.4) coordinate (g_1)-- (2.4,2.8) coordinate (g_2);
    \draw [<-,thick] (1.5,3.6) -- (2.4,4); 
    \draw [<-,thick] (1.5,3.4) -- (2.4,3.8);
    \draw [<-,thick] (1.5,3.2) -- (2.4,3.6); 
    \draw [<-,thick] (1.5,3.0) -- (2.4,3.4); 
    \draw [<-,thick] (1.5,2.8) -- (2.4,3.2);     
    \draw [<-,thick] (1.5,2.6) -- (2.4,3);    

    \draw [<-,thick] (1.5,1.1) -- (2.4,1.1); 
    \draw [<-,thick] (1.5,1.3) -- (2.4,1.3);    
    \draw [<-,thick] (1.5,1.5) -- (2.4,1.6);     
    \draw [<-,thick] (1.5,1.7) -- (2.4,1.9);
    \draw [<-,thick] (1.5,1.9) -- (2.4,2.2); 
    \draw [<-,thick] (1.5,2.1) -- (2.4,2.5);

    \draw (3,0.9) coordinate (h_1)-- (3.9,1.8) coordinate (h_2);
    \draw (4.5,0.9) coordinate (i_1)-- (5.4,1.8) coordinate (i_2); 
    \coordinate (c) at (intersection of a_1--a_2 and b_1--b_2);
    \draw[dashed] (yaxis |- c) node[left] {$z_{2}'$}
        -| (xaxis -| c) node[below] {$z_{1}'$};
    \draw[dashed] (1.5,2.4) -- (0,2.4) node[left] {$z_{2}'+\epsilon$};
    \draw[dashed] (1.5,3.9) -- (0,3.9) node[left] {$z_{2}'+2\epsilon$};
    \draw[dashed] (3,0.9) -- (3,0) node[below] {$z_{1}'+\epsilon$};
    \draw[dashed] (4.5,0.9) -- (4.5,0) node[below] {$z_{1}'+2\epsilon$};        
        
\end{tikzpicture}

Figure \ref{fig disk fibration}.

The arrows give the schematic picture for the projection $\pi$.
\end{figure}

\begin{proof}
 First, we want to explain what $t(p_1,p_2)$ means geometrically.
$\rho(2 \epsilon -t)$ is an oriented diffeomorphism from $[0, 2\epsilon] \to [0,2\epsilon-s\delta]$ so it can be viewed as a diffeomorphism from the left edge of $Q$ to the right edge of $Q$.
$p_2-(z_2'+2\epsilon-t)=\frac{(\rho(t)-(2\epsilon-t))(p_1-z_1')}{\delta}$, which we call $L_t$, is the equation of line joining the point $(z_1',z_2'+2\epsilon-t)$ and $(z_1'+\delta,z_2'+\rho(t))$.
Therefore, for a point $(p_1,p_2)$, $t(p_1,p_2)$ is such that $L_{t(p_1,p_2)}$ contains the point $(p_1,p_2)$.
Moreover, $p_2+\frac{(2\epsilon-t(p_1,p_2)-\rho(t(p_1,p_2)))p_1}{\delta}$ is the $p_2$-coordinate of the intersection between line $L_{t(p_1,p_2)}$ and the left edge of $Q$, $\{ p_1=z_1' \}$.
See Figure \ref{fig disk fibration}.

To prove the Lemma, we pick $\kappa$ close to $2 \epsilon$ from below such that $\rho'(t)=-1$ for all $t \in [\kappa, 2\epsilon]$.
Let $\Delta$ be $\pi^{-1}(\mu^{-1}(\{ z_1' \} \times [z_2'+2\epsilon-\kappa,z_2'+2\epsilon]))$
We give a smooth trivialization of $\pi|_{\Delta}$ as follows.

Let $\Phi: [0,\kappa] \times \mathbb{R}/2\pi \mathbb{Z} \times \mathbb{D}^2_{\sqrt{2\delta}} \to \Delta$ be given by sending $(t,\vartheta_1,\tau,\vartheta_2)$ to
$(p_1,q_1,p_2,q_2)=(z_1'+\tau,-s \vartheta_1+\vartheta_2,(z_2'+2\epsilon-t)+\frac{(\rho(t)-(2\epsilon-t))\tau}{\delta},-\vartheta_1)$, where $t,\vartheta_1$ are the coordinates of $[0,\kappa]$ and $\mathbb{R}/2\pi$, respectively, and $(\tau=\frac{r^2}{2},\vartheta_2)$ is such that $(r,\vartheta_2)$ is the standard polar coordinates of $\mathbb{D}^2_{\sqrt{2\delta}}$.
In particular, $\tau \in [0,\delta]$.
Note that, $\Phi$ is well-defined and it is a diffeomorphism.

Let $\pi_{\Phi}:[0,\kappa] \times \mathbb{R}/2\pi \mathbb{Z} \times \mathbb{D}^2_{\sqrt{2\delta}} \to [0,\kappa] \times \mathbb{R}/2\pi \mathbb{Z} \times \mathbb{D}^2_{\sqrt{2\delta}}$ be the projection to the first two factors.
Then, we have $\pi \circ \Phi= \Phi \circ \pi_{\Phi}$.
Notice that, when $\tau=0$, the $\vartheta_2$-coordinate degenerates and it corresponds to $p_1=z_1'$ and the $q_1$-coordinate degenerates.

To investigate this fibration under the trivialization, we have
\begin{align*}
 \Phi^{*}\omega & = \Phi^{*}(dp_1 \wedge dq_1+dp_2 \wedge dq_2) \\
                & = d\tau \wedge (-sd\vartheta_1+d\vartheta_2) \\
                & \quad + (-dt+\frac{\tau}{\delta}dt+\frac{\rho'(t)\tau}{\delta}dt+ \frac{\rho(t)-(2\epsilon-t)}{\delta}d\tau) \wedge (-d\vartheta_1)\\
                & = (1-\frac{\tau}{\delta}-\frac{\rho'(t)\tau}{\delta})dt \wedge d\vartheta_1 + d\tau \wedge d\vartheta_2 + (\frac{2\epsilon-t-\rho(t)}{\delta}-s) d\tau \wedge d\vartheta_1
\end{align*}

For a fibre, we have $t$ and $\vartheta_1$ being contant so the the symplectic form restricted on the fibre is $ d\tau \wedge d\vartheta_2$, which is the standard one.
Hence, each fibre is symplectomorphic to $(\mathbb{D}^2_{\sqrt{2\delta}},\omega_{std})$.
When $\tau=0$, the symplectic form equals $dt \wedge d\vartheta_1$ so the base is symplectic and fibres are symplectic orthogonal to the base.
Moreover, the vector space that is symplectic orthogonal to the fibre at a point is spanned by $\partial_{t}$ and $\partial_{\vartheta_1}-(\frac{2\epsilon-t-\rho(t)}{\delta}-s)\partial_{\vartheta_2}$ so the symplectic connection has structural group inside $U(1)$.

Finally, we remark that $\rho(0)=2\epsilon-s\delta$ and $\rho'(t)=-1$ when $t$ is close to $0$, hence $\Phi^{*}\omega=dt \wedge d\vartheta_1 + d\tau \wedge d\vartheta_2$.
Therefore, when $t$ is close to $0$, the trivialization $\Phi$ actually coincide with the gluing symplectomorphism in the GS construction from preimage of $R_{e_{\alpha \beta},v_{\alpha}}$ to $[x_{v_{\alpha}, e_{\alpha \beta}}-2\epsilon,x_{v_{\alpha}, e_{\alpha \beta}}-\epsilon] \times \mathbb{R}/2\pi \mathbb{Z} \times \mathbb{D}^2_{\sqrt{2\delta}}$ , up to a translation in $t$-coordinate.
On the other hand, $\pi|_{\mu^{-1}(Q)-\Delta}$ is clearly a symplectic fibration with all the desired properties described in the Lemma as it corresponds to the trivial projection by sending $(p_1,q_1,p_2,q_2)$ to $(z_1', *, p_2,q_2)$.
This finishes the proof of this Lemma.

\end{proof}

We remark that the disc fibration above gives a fibration on the local models $N_{e_{\alpha \beta}}$ and it is compatible with the trivial fibration on the local models $N_{v_{\alpha}}$ so they give a well-defined fibration after gluing all the local models $N_{v_{\alpha}}$ and $N_{e_{\alpha \beta}}$.
Now, we are ready to prove Proposition \ref{canonical contact structure_proof}.

\begin{proof}[Proof of  Proposition \ref{canonical contact structure_proof}]
 Let $(D=\cup_{i=1}^k C_i,W, \omega)$ be a symplectic plumbing.
First, we assume the intersection form of $D$ is negative definite (or equivalently, the augmented graph satisfies the negative GS criterion). 
By \cite{GaSt09}, $D$ satisfies the negative GS criterion.
Therfore, by possibly shrinking $W$, we can assume $W$ is a symplectic plumbing constructed from the negative GS criterion.
A byproduct of the construction is the existence of a primitive of $\omega$ on $W-D$, $\theta$, given by contracting $\omega$ by the Liouville vector field.
From the construction, in the $N_{v_{\alpha}}$ local model, we have $\theta=\iota_{\bar{X}_{v_{\alpha}}+(\frac{r}{2}+\frac{z_{v_{\alpha}}'}{r})\partial_r} (\bar{\omega}_{v_{\alpha}}+rdr \wedge d\vartheta_2)=\iota_{\bar{X}_{v_{\alpha}}}\bar{\omega}_{v_{\alpha}}+(\frac{r^2}{2}+z_{v_{\alpha}}')d\vartheta_2$.
When we restrict it to a fibre, we can see that the wrapping numbers of $\theta$ with respect to $C_{v_{\alpha}}$ is $2 \pi z_{v_{\alpha}}'$, which is positive.
Here, $C_{v_{\alpha}}$ is the smooth symplectic submanifold corresponding to the vertex $v_{\alpha}$.

Note that we have $\lambda=-z=2\pi z'$ in our convention above.
By tracing back the negative GS construction, we see that Lemma \ref{canonical contact structure_tech} provides a desired symplectic fibrations we needed to apply Proposition \ref{McLean}.
In particular, this symplectic fibrations give us well-defined $r_i$-coordinates near the divisor.
As a result, one can set $f=\sum\limits_{i=1}^k \log(\rho(r_i))$ and $g=0$ and get that $df(X_\theta) > 0$ near $D$ and $(f^{-1}(l),\theta|_{f^{-1}(l)})$ is precisely the contact manifold obtained from the negative GS criterion.
In particular, $(f^{-1}(l),\theta|_{f^{-1}(l)})$ is the canonical one with respect to $(W,D)$.

If we have made another set of choices in the construction, we get that $(\bar{f}^{-1}(l),\theta|_{\bar{f}^{-1}(l)})$ is the canonical one with respect to $(\bar{W},\bar{D})$.
Then, since $(W,D)$ is diffeomorphic to $(\bar{W},\bar{D})$, we can pull back the compatible function and the $1$-form on $(\bar{W},\bar{D})$ to $(W,D)$.
By Proposition \ref{McLean}, the two contact manifolds are contactomorphic.
Same argument works to show that this contact structure is independent of symplectic area $a$ as long as $(\Gamma,a)$ still satisfies GS criteria.
Also, when $D$ is arising from isolated normal surface singularity, contact structure of its link is contactomorphic that induced by GS-criterion, by Proposition \ref{McLean}, again.
This finishes the case when $D$ is negative definite.

Now, we assume that $(D,\omega)$ satisfies the positive GS criterion and $D$ is $\omega$-orthogonal.
By the same reasoning as bove, the contact structure induced by the positive GS criterion is independent of choices.

Suppose $D$ is also the support of an effective ample line bundle.
Pick a hermitian metric $\| . \|$ and a section $s$ with zero being $\sum\limits_{i=1}^k z_i C_i$, where $z_i > 0$.
Let $f=-\log \| s\|$, $\theta= -d^cf$ and $\omega= d\theta$, where $d^cf=df \circ J_{std}$.
Then, $\theta$ induces a contact structure on the boundary of plumbing of $(D,\omega)$ with negative wrapping numbers (See Lemma 5.19 of \cite{McL12}).
Moreover, $f$ is compatible with $D$ and $df(-X_{\theta}) > 0$ near $D$ (See Lemma 2.1 of \cite{McL12} or Lemma 4.12 of \cite{McL14}).
Hence the contact structure induced by $\theta$ is contactomorphic to the canonical one by Proposition \ref{McLean}, which is contactomorphic to the one induced by the positive GS criterion.

\end{proof}

\subsection{Examples of Concave Divisors}\label{Comments on the Induced Contact Structure}

In this subsection, we are going to see five illuminating examples.
The first one is the simplest kind of symplectic divisor.
The second one illustrates Theorem \ref{obstruction-GS} is no longer valid if the plumbing chosen is not close to the divisor.
In particular, there is a concave divisor which admits a convex neighborhood but it is not a convex divisor.
The third one is a frequently used example when studying Lefschetz fibration.
The forth one is a concave divisor with non-fillable contact structure on the boundary.
The last one shows that the constructed contact structure on the boundary is not necessarily contactomorphic to the standard one that one might expect if the divisor is concave.

\begin{eg}\label{single vertex}
A symplectic surface with self-intersection $n$ admits a concave (resp convex) boundary when $n>0$ (resp $n<0$).
When $n=0$, a symplectic form cannot make both the surface symplectic and the restriction to boundary be exact so it has no convex or concave neighborhood.
In fact, more is true, by a result of Eliasberg \cite{El90}, $\mathbb{S}^1 \times \mathbb{S}^2$ cannot be a convex boundary of any symplecyic form on $\mathbb{D}^2 \times \mathbb{S}^2$.
In contrast, althought a symplectic torus with self-intersection zero has no concave nor convex neighborhood, a Lagrangian torus has self-intersection zero and has a convex neigborhood.
\end{eg}

\begin{eg}(\cite{Mc91})
In \cite{Mc91}, McDuff constructed a symplectic form on $(S\Sigma_g \times [0,1],\omega)$ such that it has disconnected convex boundary, where $S\Sigma_g$ is a circle boundle of a genus $g$ surface and $g>1$.
The contact structure near $S\Sigma_g \times \{0\}$ is contactomorphic to the concave boudary near a self-intersection $2g-2$ symplectic genus g surface.
The contact structure near $S\Sigma_g \times \{1\}$ is contactomorphic to the convex boundary near a Lagrangian genus g surface.
If one glues a symplectic closed disc bundle $P(D)$ over a symplectic genus g surface $D$ with $(S\Sigma_g \times [0,1],\omega)$ along $S\Sigma_g \times \{0\}$.
One gets a plumbing of the surface with convex boundary.
This suggests that a symplectic genus $g$ ($g > 1$) surface can have both concave and convex neighborhood, depending on the symplectic form and the neighborhood.
Notice that $D$ is trivially $\omega$-orthogonal.
It illustrates that the assumption on $P(D)$ being sufficiently close to $D$ in Theorem \ref{obstruction-GS} cannot be dropped.
Moreover, by Theorem \ref{obstruction-GS}, $D$ is a concave divisor but not a convex divisor although it admits convex neighborhood.
\end{eg}

\begin{eg}\label{section fiber}
Suppose there is a symplectic Lefschetz fibration $(X,\omega)$ over $\mathbb{CP}^1$ with generic fibre $F$ and a symplectic section $S$ of self-intersection $-n$ ($n \ge 0$).
Let $D=F \cup S$, then the augmented graph of $D$ always satisfies the positive GS criterion regardless the area weights of the surfaces.
Suppose also that $S$ is perturbed to be $\omega$-orthogonal to $F$.
Then, Proposition \ref{MAIN2} (or, Proposition \ref{McLean0} if one does not want to perturb) shows that $D$ is a concave divisor.
In other words, the complement of a concave neighborhood of $D$ is a convex filling of its boundary.

This fits well to the well-known fact that the complement of a regular neighborhood of $D$ is a Stein domain.
Moreover, this construction has been successfully used to find exotic Stein fillings \cite{AkEtMaSm08}. 
\end{eg}

\begin{lemma}\label{b_2^+}
 Let $(\Gamma,a)$ be an augmented graph satisfying the positive GS criterion and $D$ be a realization.
Suppose there are two genera zero vertices with self-intersection $s_1,s_2$ such that either

(i) they are adjacent to each other and $s_1 >s_2\ge 1$, or

(ii) they are not adjacent to each other with $s_1 \ge 1$ and $s_2 \ge 0$.

Then, $D$ is a concave divisor but not a capping divisor.
\end{lemma}

\begin{proof}
Suppose on the contrary, the boundary has a convex fillings $Y$.
Then, we can glue $D$ with $Y$ to obtain a closed symplectic 4 manifold $W$.
By McDuff's theorem \cite{Mc90} (see also Theorem \ref{McDuff} below), $W$ is rational or ruled and hence have $b_2^+ =1$.
For (i), the two spheres generates a positive two dimensional subspace of $H_2(W)$ with respect to the intersection form.
Thus, we get a contradiction.
For (ii), it suffices to consider tha case $s_1=1$ and $s_2=0$.
By the Theorem in \cite{Mc90}, one can assume the sphere with self-intersection $1$ represent the hyperplane class $H$, with respect to an orthonormal basis $\{H,E_1,\dots,E_n \}$ for $H_2(W)$.
The two spheres being disjoint implies the one with self-intersection $0$ has homology class being a linear combination of exceptional classes.
Since the sphere is symplectic, the linear combination is non-trivial.
Thus, we get a contradiction.
\end{proof}

\begin{eg}
 Let $\Gamma$ be the graph in Example \ref{2-1}.
\[ Q_{\Gamma}= \left( \begin{array}{cc}
2 & 1 \\
1 & 1 \end{array} \right),\]
Then the boundary fundamental group of $\Gamma$ is the free group generated by $e_1$ and $e_2$ modulo the relations $e_1e_2^1=e_2^1e_1$, $1=e_1^2e_2$ and $1=e_1e_2$ (See Lemma \ref{representation} below).
Therefore, the boundary of the plumbing according to $\Gamma$ has trivial fundamental group and hence diffeomorphic to a sphere.
It is easily see that the corresponding augmented graph $(\Gamma,a)$ satisfies the positive GS criterion if and only if the area weights satisfy $a_1 < a_2 < 2a_1$, where $a_i$ is the area weight of $v_i$.
In other words, if $a_1 < a_2 < 2a_1$, by Proposition \ref{MAIN2} and Lemma \ref{b_2^+}, we get an overtwisted contact structure on $S^3$ ($S^3$ has only one tight contact structure which is fillable).
\end{eg}

\begin{eg}\label{non-standard example}

There is a capping divisor $D$ with graph as in the following Figure, by Theorem \ref{main classification theorem}.
However, $D$ is not conjugate to any other divisor (See Definition \ref{Conjugate Definition_Divisor}), by Lemma \ref{first chern class}.
$\partial P(D)$ is diffeomorphic to the boundary of the plumbing of the resolution of a tetrahedral singularity and the latter one which is equipped with a standard contact structure has a conjugate.
Therefore, the fillable contact strucutre on $\partial P(D)$ is not the standard one.
Applying the method in \cite{Li08},\cite{BhOn12} and \cite{St13}, one can obtain a finiteness result on the number of minimal symplectic manifolds that can be compactified by $D$, up to diffeomorphism (See Proposition \ref{bounds}).

\begin{fig}
$$    \xymatrix{
        \bullet^{-3} \ar@{-}[r]& \bullet^{-2} \ar@{-}[r] \ar@{-}[d]& \bullet^{-2} \ar@{-}[r] & \bullet^{1} \\
			                              &\bullet^{-2} \\
	}
$$ 
\end{fig}
 
\end{eg}

\section{Operations on Divisors}\label{Operation on Divisors}
In this section we first apply the inflation operation to establish Theorem \ref{MAIN}.
Then, we explain in details the resulting flowchart and illustrate how to reduce the classification problem into
a problem of graphs. 
Finally, we introduce blow up and the {\bf dual} blow up operations, which are essential to the next section. 

\subsection{Theorem \ref{MAIN}}

The proof of Theorem \ref{MAIN} involes two inputs.
The first one is a linear algebraic lemma, which is simple but important.
The second one which is called inflation lemma allows us to deform the symplectic form to our desired one 
so as to apply Propositon \ref{MAIN2}.

\subsubsection{A key lemma}

The following  linear algebraic Lemma related to the positive GS criterion will be crucial.
\begin{lemma}\label{trichotomy}
 Let $Q$ be a k by k symmetric matrix with off-diagonal entries being all non-negative.
Assume that there exist $a \in (0,\infty)^k$ such that there exist $z \in \mathbb{R}^k$ with $Qz=a$.
Suppose also that $Q$ is not negative definite.
Then, there exists $z \in (0,\infty)^k$ such that $Qz \in (0,\infty)^k$.
\end{lemma}

\begin{proof}
When $k=1$, it is trivial.
Suppose the statement is true for (k-1) by (k-1) matrix and now we consider a k by k matrix $Q$.
Let $q_{i,j}$ be the $(i,j)^{th}$-entry of $Q$.
First observe that if $q_{i,i} \ge 0$, for all $i=1,\dots,k$, then the statement is true.
The reason is that if each row has a positive entry, then $z=(1,\dots,1)$ works.
If there exist a row with all $0$, then there is no $a \in (0,\infty)^k$ such that there exist $z \in \mathbb{R}^k$ with $Qz=a$.

Therefore, we might assume $q_{k,k}<0$.
Let $l_j=-\frac{q_{k,j}}{q_{k,k}} \ge 0$, for $j <k$, and let $B$ be the lower triangular matrix given by 
\[ b_{i,j} = \left \{ 
  \begin{array}{l l}
    \delta_{i,j} & \quad \text{if $i \neq k$ or $(i,j)=(k,k)$}\\
    l_j & \quad \text{if $i = k$ and $(i,j) \neq (k,k)$}
  \end{array} \right.\]
Let $M=B^TQB$.
Then, 
\[ m_{i,j} = \left \{ 
  \begin{array}{l l}
    q_{i,j}-\frac{q_{i,k}q_{k,j}}{q_{k,k}} & \quad \text{if $(i,j) \neq (k,k)$}\\
    q_{k,k} & \quad \text{if $(i,j) = (k,k)$}
  \end{array} \right.\]
In particular, $m_{i,k}=m_{k,j}=0$, for all $i$ and $j$ less than $k$.
We can write $M$ as a direct sum of a k-1 by k-1 matrix $M'$ with the 1 by 1 matrix $q_{k,k}$ in the obvious way.
Notice that the off diagonal entries of $M'$ are all non-negative.

Let $a=(a_1, \dots,a_k)^T$ and $z=(z_1,\dots,z_k)^T$ such that $Qz=a$.
Let also $\overline{z}=(\overline{z}_1,\dots,\overline{z}_k)^T=B^{-1}z$ and $\overline{a}=(\overline{a}_1, \dots,\overline{a}_k)^T=B^Ta$.
Then, $Qz=a$ is equivalent to $M\overline{z}=\overline{a}$.
Here, $\overline{z}_i=z_i$, for $i <k$, and $\overline{z}_k=z_k-\sum\limits_{i=1}^{k-1}l_iz_i$.
On the other hand,  $\overline{a}_i=a_i+l_ia_k$, for all $i < k$, and  $\overline{a}_k=a_k$.

By assumption, there exist $a \in (0,\infty)^k$ such that there exist $z \in \mathbb{R}^k$ with $Qz=a$.
So we have $(\overline{a}_1,\dots, \overline{a}_{k-1})^T \in (0,\infty)^k$ and $M' (z_1,\dots,z_{k-1})^T=(\overline{a}_1,\dots, \overline{a}_{k-1})^T$.
Apply induction hypothesis, we can find $y \in (0,\infty)^{k-1}$ such that $M'y \in (0,\infty)^{k-1}$.
Pick $y_k >0$ such that $q_{k,k}(y_k-\sum\limits_{i=1}^{k-1}l_iy_i) >0$ but sufficient close to zero.
Then, let $\overline{z}=(y_1,\dots,y_{k-1},y_k-\sum\limits_{i=1}^{k-1}l_iy_i)^T$ and tracing it back.
We have $Q(y_1,\dots,y_k)^T \in (0,\infty)^k$.
\end{proof}

Regarding the negative GS criterion, we remark that one can show the following. (It is mentioned in \cite{GaMa11} with additional assumption but the additional assumption can be removed.)
Suppose $Q$ is a symmetric matrix with off-diagonal entries being non-negative.
Then, the following statements are equivalent.

(a) For any $a \in (0,\infty)^n$, there exist $z \in (-\infty,0)^n$ satisfying $Qz=a$.

(a2) For any $a \in (0,\infty)^n$, there exist $z \in (-\infty,0]^n$ satisfying $Qz=a$.

(b) There exist $a \in (0,\infty)^n$ such that there exist $z \in (-\infty,0)^n$ satisfying $Qz=a$.

(b2) There exist $a \in (0,\infty)^n$ such that there exist $z \in (-\infty,0]^n$ satisfying $Qz=a$.

(c) $Q$ is negative definite.

The implication from (a) to (b), (a2) to (b2), (a) to (a2), (b) to (b2) are trivial.
(c) implying (a2) is Lemma 3.3 of \cite{GaSt09} and a moment thought will justify (c)+(a2) implying (a), which is hiddenly used in \cite{GaSt09}.
(b) implying (c) is similar to the proof of Lemma \ref{trichotomy}.
To be more precise, one again use induction on the size of $Q$ and change the basis using $B$.  
Therefore, an augmented graph $(\Gamma,a)$ satisfies the negative GS criterion if and only if $Q_{\Gamma}$ is negative definite.
In particular, when a graph $\Gamma$ is negative definite, the negative GS criterion is always satisfied, independent of the area weights.

\subsubsection{Inflation}
Now, it comes the second input.

\begin{lemma}\label{inflation}(Inflation, See \cite{LaMc96} and \cite{LiUs06})
 Let $C$ be a smooth symplectic surface inside $(W,\omega)$.
If $[C]^2 \ge 0$, then there exists a family of symplectic form $\omega_t$ on $W$ such that $[\omega_t]=[\omega]+tPD(C)$ for all $t \ge 0$.
If $[C]^2 < 0$, then there exists a family of symplectic form $\omega_t$ on $W$ such that $[\omega_t]=[\omega]+tPD(C)$ for all $0 \le t < -\frac{\omega[C]}{[C]^2}$.
Also, $C$ is symplectic with respect to $\omega_t$ for all $t$ in the range above.
Moreover, if there is another smooth symplectic surface $C'$ intersect $C$ positively and $\omega$-orthogonally, then $C'$ is also symplectic with respect to $\omega_t$ for all $t$ in the range above. 
Here, $PD(C)$ denotes the Poincare dual of $C$.
\end{lemma}

When $[C]^2 < 0$, one can see that $([\omega]+tPD(C))[C] > 0$ if and only if $t < -\frac{\omega[C]}{[C]^2}$.
Therefore, the upper bound of $t$ in this case comes directly from $\omega_t[C] > 0$.
We remarked that one can actually do inflation for a larger $t$ but one cannot hope for $C$ being symplectic anymore when $t$ goes beyond $-\frac{\omega[C]}{[C]^2}$.

\subsubsection{Proof}
\begin{proof} [Proof of Theorem \ref{MAIN}]

First of all, we can isotope $D$ to $D'$ such that every intersection of $D'$ is $\omega_0$-orthogonal, using Theorem 2.3 of \cite{Go95}.
Since every intersection of $D$ is transversal and no three of $C_i$ intersect at a common point, such an isotopy can be extended to an ambient isotopy.
Now, instead of isotoping $D$, we can deform $\omega_0$ through the pull back of $\omega_0$ along the isotopy.
As a result, we can assume $D$ is $\omega_0$-orthogonal.

Now, we want to construct a family of realizations $D_t$ of $\Gamma$, by deforming the symplectic form, such that the augmented graph of $D_1$ satisfies the positive GS criterion.

Let $D=D_0=C_1 \cup \dots \cup C_k$ and let also the area weights of $D_0$ with respect to $\omega_0$ be $a$.
Since $\omega$ is exact on $\partial P(D)$, there exists $z$ such that $Q_{\Gamma}z=a$.
Also, by assumption and Lemma \ref{trichotomy}, there exists $\overline{z} \in (0, \infty)^k$ such that $Q_{\Gamma}\overline{z}=\overline{a} \in (0, \infty)^k$.
Let $z^t=z+t(\overline{z}-z)$ and $a^t=a+t(\overline{a}-a)=Q_{\Gamma}z^t \in (0, \infty)^k$.
We want to construct a realization $D_1$ of $\Gamma$ with area weights $a^1$.
If this can be done, then the augmented graph of $D_1$ will satisfy the positive GS criterion. 

Observe that, it suffices to find a family of symplectic forms $\omega_t$ such that $[\omega_t]=[\omega_0]+t \sum\limits_i (\overline{z_i}-z_i)PD([C_i])$ 
and a corresponding family of $\omega_t$-symplectic divisor $D_t=C_1 \cup \dots \cup C_k$.
The reason is that $C_i$ has symplectic area equal the $i^{th}$ entry of $a^t$ under the symplectic form $[\omega_t]=[\omega_0]+t \sum\limits_i (\overline{z_i}-z_i)PD([C_i])$.
However, we need to modify this natural choice of family a little bit.
Without loss of generality, we can assume $\overline{z_i}>z_i$ for all $1 \le i \le k$. 
We can choose a piecewise linear path $p^t$ arbitrarily close to $z^t$ such that each piece is parallel to a coordinate axis and moving in the positive axis direction.
Since satisfying the positive GS criterion is an open condition, we can choose $p^t$ such that $Q_{\Gamma}p^t \in (0, \infty)^k$.
The fact that $p^t$ is chosen such that $Q_{\Gamma}p^t$ is entrywise greater then zero allows us to do inflation along $p^t$ to get out desired family of $\omega_t$ and $D_t$, by Lemma \ref{inflation}.
Therefore, we arrive at a symplectic form $\omega_1$ such that the augmented graph of $(D,\omega_1)$, denoted by $(\Gamma,a)$, satisfies the positive GS criterion.  
We finish the proof by applying Propositon \ref{MAIN2}.

\end{proof}

\begin{rmk}\label{symplecitc cone}
 The proof of Theorem \ref{MAIN} implies that for any $a \in (0,\infty)^k \cap Q_D (0,\infty)^k$,
 there is a symplectic deformation making the augmented graph of $(D,\omega_1)$ to be $(\Gamma,a)$.
\end{rmk}

\begin{proof}[Proof of Corollary \ref{QHS}]
First suppose $D$ is not negative definite.
By Theorem \ref{MAIN}, $\omega$ being exact on the boundary implies $D$ is a concave divisor after a symplectic deformation.
If $D$ is negative definite, then $\omega$ is necessarily exact on the boundary with unique lift of $[\omega]$ to a relative second cohomology class.
Moroever, the discussion after the proof of Lemma \ref{trichotomy} implies that $D$ satisifes negative GS criterion and hence $D$ is a convex divisor.
\end{proof}

\subsection{Flowchart and Reducing Classification Problem to Graph}

We offer a detailed explanation of  the flowchart.

Given a divisor $(D,\omega)$ (not necessarily $\omega$-orthogonal, see Proposition \ref{McLean0}),
the first obstruction of whether $D$ admits a concave or convex neighborhood comes
from $\omega$ being not exact on the boundary of $D$.
In this case, $[\omega]$ cannot be lifted to a relative second cohomology class and $Q_{D}z=a$ has no solution for $z$.

If $\omega$ is exact on the boundary, we look at the solutions $z$ for the equation $Q_{D}z=a$.
When $Q_D$ is negative definite (in this case $\omega$ is necessarily exact on the boundary), there is a 
unique solution for $z$ and all the entries for this solution is negative.
Therefore, $(D,\omega)$ satisfies the negative GS criterion and $D$ is convex (Proposition \ref{MAIN2} or Proposition \ref{McLean0}).

If $\omega$ is exact on the boundary but $Q_D$ is not negative definite, the situation becomes a bit more complicated.
There might be more than one solution for $z$ (when $Q_D$ is degenerate).
If we are lucky that there is one solution $z$ with all entries being positive, then $D$ is concave (Proposition \ref{MAIN2} or Proposition \ref{McLean0}).

However, it is possible that all the solutions $z$ have at least one entry being non-positive.
In this case, if $D$ is $\omega$-orthogonal or $D$ lies inside a closed symplectic manifold, there is a small neigborhood $N$ of $D$ such that $D$ has no convex nor conave neighborhood inside $N$ (Theorem \ref{obstruction-GS} and Theorem \ref{obstruction-closed case}). 
However, we can choose an area vector $\bar{a}$ such that there is a solution $\bar{z}$ for $Q_D\bar{z}=\bar{a}$
with all entries of $\bar{z}$ being positive (Lemma \ref{trichotomy}).
Geometrically, we can do inflation (Lemma \ref{inflation}) to deform the symplectic form such that $(D,\bar{\omega})$
has area vector $\bar{a}$.
Then, $(D,\bar{\omega})$ is concave (Proposition \ref{MAIN2} or Proposition \ref{McLean0}).
This is exactly the proof of Theorem \ref{MAIN}.

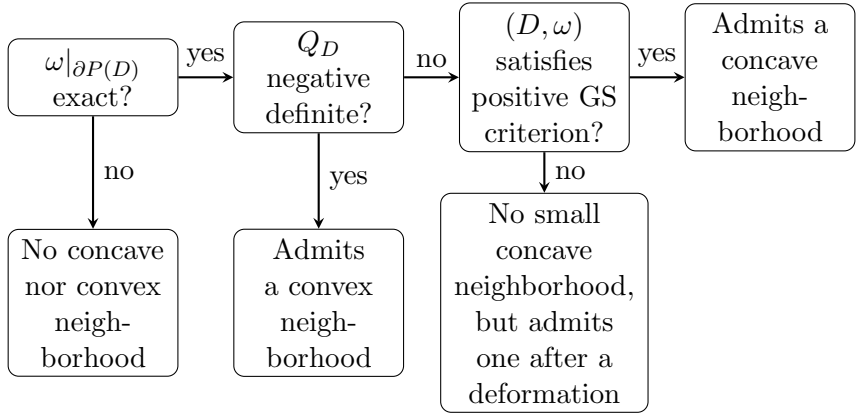
\begin{tikzpicture}[node distance=3cm]

\node (exact) [startstop] {$\omega|_{\partial P(D)}$ exact?};
\node (not exact) [startstop2, below of=exact] {No concave nor convex neighborhood};
\node (definite) [startstop, right of=exact] {$Q_D$ negative definite?};
\node (convex) [startstop2, below of=definite] {Admits a convex neighborhood};
\node (GS criterion) [startstop, right of=definite] {$(D,\omega)$ satisfies positive GS criterion?};
\node (concave) [startstop2, right of=GS criterion] {Admits a concave neighborhood};
\node (deformation) [startstop3, below of=GS criterion] {No small concave neighborhood, but admits one after a deformation};

\draw [arrow] (exact) -- node[right]{no}(not exact);
\draw [arrow] (exact) -- node[above]{yes}(definite);
\draw [arrow] (definite) -- node[right]{yes}(convex);
\draw [arrow] (definite) -- node[above]{no}(GS criterion);
\draw [arrow] (GS criterion) -- node[above]{yes}(concave);
\draw [arrow] (GS criterion) -- node[right]{no}(deformation);

\end{tikzpicture}

In Section \ref{Classification of Symplectic Divisors Having Finite Boundary Fundamental Group}, we investigate capping (i.e. embeddable and concave) divisors with boudary fundamental group.
Before doing this, we want to see how we use Theorem \ref{MAIN}, Proposition \ref{McLean0}, Theorem \ref{obstruction-closed case} and Proposition \ref{McLean} to reduce the problem
to the realizability of its graph.

\begin{prop}\label{reduce to graph}
 Suppose $\Gamma$ is realizable.
 If the graph of $(D,\omega)$ is $\Gamma$, then $(D,\omega)$ is a capping divisor if and only if $(D,\omega)$ satisfies the positive GS criterion.
\end{prop}

\begin{proof}
 Let $(\overline{D},\overline{\omega}) \subset \overline{W}$ be a realization of $\Gamma$ (See Definition \ref{realizable definition}).
 In other words, $\overline{W}$ is a closed symplectic manifold and $\Gamma$ is the graph of $\overline{D}$.
 
 We first assume $(D,\omega)$ satisfies the positive GS criterion.
 By Theorem \ref{MAIN} (See Remark \ref{symplecitc cone}), we can find an $\overline{\omega}'$-orthogonal 
 capping divisor $(\overline{D},\overline{\omega}')$ by doing symplectic deformation in $\overline{W}$
 as long as its augmented graph $(\Gamma,\overline{a})$ satisfies the positive GS criterion (this is equivalent to $\overline{a} \in (0,\infty)^k \cap Q_{\Gamma} (0,\infty)^k$).
 Therefore, we can choose $\overline{a}$ such that $\overline{a}=a$, where $a$ is the area vector of $(D,\omega)$.
 Hence the augmented graphs of $(D,\omega)$  and $(\overline{D},\overline{\omega}')$ are 
 the same and satisfy the positive GS criterion.
 By Proposition \ref{McLean0}, $(D,\omega)$ is a concave divisor.
 Moreover, Proposition \ref{McLean} implies that the contact structures constructed on $P(D)$ and $P(\overline{D})$ are contactomorphic.
 Therefore, we can cut $Int(P(\overline{D}))$ from $\overline{W}$ and glue it with $P(D)$ along the boundary to get a closed symplectic manifold $W$.
 As a result, $(D,\omega)$ is a capping divisor.
 
 For the other direction, $(D,\omega)$ is a capping divisor.
 In particular, $(D,\omega)$ can be embedded into a closed symplectic manifold.
 By Theorem \ref{obstruction-closed case}, there is a neighborhood $N$ of $D$ such that there is no concave neighborhood $P(D)$ of $D$ lying inside $N$ if $(D,\omega)$ does not satisfies positive GS criterion.
 Therefore, $(D,\omega)$ is not a concave divisor.
 Contradiction.
\end{proof}

Having this, we are going to focus our study on graphs and Section \ref{Classification of Symplectic Divisors Having Finite Boundary Fundamental Group}
is solely the classification of realizable graphs with finite boundary fundamental group.

\subsection{Blow Up and Dual Blow Up}
The symplectic blow up and blow down operations have obvious analogues  in the category of graphs and  augmented graphs. We will
describe these operations for augmented graphs. Both of these operations will play an important role in the classification of capping divisors with
finite boundary $\pi$
in the next section. 

Let $(\Gamma,a)$ be an augmented graph.
In the following figure, $(\tilde{\Gamma},\tilde{a})$ is obtained by blowing up $(\Gamma,a)$ at $v_1$ and $(\tilde{\tilde{\Gamma}},\tilde{\tilde{a}})$ is obtained by blowing up $(\tilde{\Gamma},\tilde{a})$ at an edge between $v_1$ and $v_0$.
If the weight for the first blow up is $a_0$, then the area of $v_1$ and $v_0$ in $(\tilde{\Gamma},\tilde{a})$ is $a_1-a_0$ and $a_0$, respectively.
If the weight for the second blow up is $a_{-1}$, then the area of $v_1$, $v_{-1}$ and $v_0$ in $(\tilde{\tilde{\Gamma}},\tilde{\tilde{a}})$ is $a_1-a_0-a_{-1}$, $a_{-1}$ and $a_0-a_{-1}$, respectively.

\begin{fig}\label{blow up figure}
 \begin{displaymath}
    \xymatrix@R=1pc @C=1pc{
     \dots \ar@{-}[r]  & \bullet_{v_1}^{-y} \ar@{-}[r]  & \dots        &\dots \ar@{-}[r] &  \bullet_{v_1}^{-1-y} \ar@{-}[r] \ar@{-}[d] & \dots &          \dots \ar@{-}[r] & \bullet_{v_1}^{-2-y} \ar@{-}[r] \ar@{-}[d] & \dots \\
	           &                              &              &		   & \bullet_{v_0}^{-1}  	                   &       &         &                 \bullet_{v_{-1}}^{-1} \ar@{-}[d]  & \\            
(a)~(\Gamma, a)         &  	                          &              & (b)~(\tilde{\Gamma},\tilde{a})    &                    	                   &        &         (c)~(\tilde{\tilde{\Gamma}},\tilde{\tilde{a}})          &    \bullet_{v_0}^{-2} &
}
\end{displaymath} 
\end{fig}

\begin{defn}
 If one graph $\Gamma$ can be obtained from another graph $\Gamma'$  through blow ups and blow downs, then we call $\Gamma$ and $\Gamma'$  {\it equivalent}.
 
 A graph is called minimal if no blow down can be performed. 
\end{defn}

\begin{eg}\label{0-0>1}
$\xymatrix{ \bullet^{0} \ar@{-}[r] & \bullet^{0}\\ } $ is equivalent to $\xymatrix{ \bullet^{1}} $:

$\xymatrix@R=1pc @C=2.5pc{ \bullet^{0} \ar@{-}[r] & \bullet^{0} & \ar[r] & & \bullet^{-1} \ar@{-}[r] & \bullet^{-1} \ar@{-}[r] & \bullet^{-1}\\ 
\ar[r] &  & \bullet^{-1} \ar@{-}[r] & \bullet^{0} & \ar[r] & & \bullet^{1} \\} $
\end{eg}

We remark that both realizablility and strongly realizability (See Definition \ref{realizable definition}) are stable under blow ups and blow downs for graphs.
However, there is no obvious reason for (strong) realizability to be stable under blow ups for augmented graphs (See Lemma \ref{stability of criterion_blow up} below).

\begin{lemma}\label{stability of criterion_blow up}
 Suppose $(\Gamma,a)$ is an augmented graph and $(\tilde{\Gamma},\tilde{a})$ be obtained from a single blow up of $(\Gamma,a)$ with weight $a_0$.
If $(\Gamma,a)$ satisfies the negative GS criterion, then so does $(\tilde{\Gamma},\tilde{a})$.
If $(\Gamma,a)$ satisfies the positive GS criterion, then so does $(\tilde{\Gamma},\tilde{a})$ for $a_0$ being sufficiently small.
\end{lemma}

\begin{proof}
Suppose that $Q_{\Gamma}z=a$ for a vector $z=(z_1, \dots, z_k)^{T}$.
We need to consider two cases.
First, if $\tilde{\Gamma}$ is obtained from blowing up at the vertex of $v_1$ of $\Gamma$,
then 
$\tilde{z}=(\tilde{z_0},\tilde{z_1},\dots,\tilde{z_k})^T=(z_1-a_0,z_1,z_2, \dots, z_k)^{T}$ 
satisfies $Q_{\tilde{\Gamma}}\tilde{z}=\tilde{a}$.
Secondly, if $(\tilde{\Gamma},\tilde{a})$ is obtained from blowing up an edge between $v_1$ and $v_2$,
then
$\tilde{z}=(\tilde{z_0},\tilde{z_1},\dots,\tilde{z_k})^T=(z_1+z_2-a_0,z_1,z_2, \dots, z_k)^{T}$
satisfies $Q_{\tilde{\Gamma}}\tilde{z}=\tilde{a}$.
In any of the above two cases, the Lemma follows.

\end{proof}

This illustrates the difference between convex and concave boundary.
After blowing up a concave neighborhood 5-tuple that is obtained from the positive GS criterion, 
we might no longer  be able to apply the GS criterion to get a concave neighborhood 5-tuple if the weight of blow up is too large.
However, if we blow up a convex neighborhood 5-tuple that is obtained from the negative GS criterion, we can still get a convex neighborhood 5-tuple by the criterion again.

For a graph $\Gamma$ and a vertex $v_1$ of $\Gamma$, we use $\Gamma^{(v_1)}$ to denote the graph that is obtained by adding two genera zero and self-intersection number zero vertices to the vertex $v_1$ of $\Gamma$ as illustrated in the following Figure.
It is clear that $\Gamma^{(v_1)}$ is equivalent to attaching a single vertex of genus $0$ and self-intersection $1$ to $v_1$ and adding the self-intersection of $v_1$ by $1$ as in the following Figure (See Example \ref{0-0>1}).
We denote it by $\overline{\Gamma^{(v_1)}}$ and call it the {\bf dual} blow up of $\Gamma$ at $v_1$.

\begin{fig}\label{dual blow up example}

\begin{displaymath}
    \xymatrix @R=1pc @C=1pc {
     \dots \ar@{-}[r] & \bullet_{v_1}^{-y} \ar@{-}[r]     & \dots        &\dots \ar@{-}[r] & \bullet_{v_1}^{-y} \ar@{-}[r] \ar@{-}[d] & \dots &          \dots \ar@{-}[r] & \bullet_{v_1}^{1-y} \ar@{-}[r] \ar@{-}[d] & \dots \\
			 &                              &              &		    & \bullet_{v_0}^{0}  	  \ar@{-}[d]      &        &                            &    \bullet_{v_0}^{1}      \\
	(a)~\Gamma       &  	                        &              & (b)~\Gamma^{(v_1)}     & \bullet_{v_{-1}}^{0} 	                  &        &      (c)~\overline{\Gamma^{(v_1)}}    &    
}
\end{displaymath}  

By comparing $\overline{\Gamma^{(v_1)}}$ and the blown-up graph of $\Gamma$ at $v_1$ (See Figure \ref{blow up figure}(b)), we can regard the dual blow up as a dual operation of blow up.

\end{fig}

We remark that  in \cite{Ne81} the  dual blow up operation is also called blow up, and 
it  is mentioned in \cite{Ne81} that the blow up and dual blow up operations do not change the oriented diffeomorphism type of the boundary of the plumbing.

\section{Capping Divisors with Finite Boundary $\pi_1$}\label{Classification of Symplectic Divisors Having Finite Boundary Fundamental Group}
In this section, we classify capping divisors with finite boundary fundamental group.
For completion and illustration, the classification of filling divisors with finite boundary fundamental group is given in subsection \ref{Filling classification}. 
Different from the proof of filling divisors, the study for capping divisors requires essential symplectic input.
Then, we illustrate the (strong) realizability of the graphs in type (P) and thus finish the proof of Theorem \ref{main classification theorem}.
Moreover, we sketch the proof of finiteness of fillings result (Proposition \ref{bounds}) and study a conjugate phenomenon in subsection \ref{Fillings}.

\subsection{Statement of Classification}

We use $<n,\lambda>$ to denote the following linear graph, where $\lambda$ and $n$ are both positive integers and $\lambda < n$,
\begin{displaymath}
    \xymatrix{
        \bullet^{-d_1} \ar@{-}[r] & \bullet^{-d_2} \ar@{-}[r] & \dots \ar@{-}[r] & \bullet^{-d_k}\\
	}
\end{displaymath}
where each vertex has genus zero and $d_i \ge 2$ are the minus of 
the self-intersection numbers such that 
$$\frac{n}{\lambda}=d_1-\frac{1}{d_2-\frac{1}{\dots-\frac{1}{d_k}}}. $$
In what follows, we use $[d_1,\dots,d_k]$ to denotes the continuous fraction 
so the condition above is just $\frac{n}{\lambda}=[d_1,\dots,d_k]$.

Moreover, we use $<y;n_1,\lambda_1;n_2,\lambda_2;n_3,\lambda_3>$ to denote the following graph with exactly one branching point.
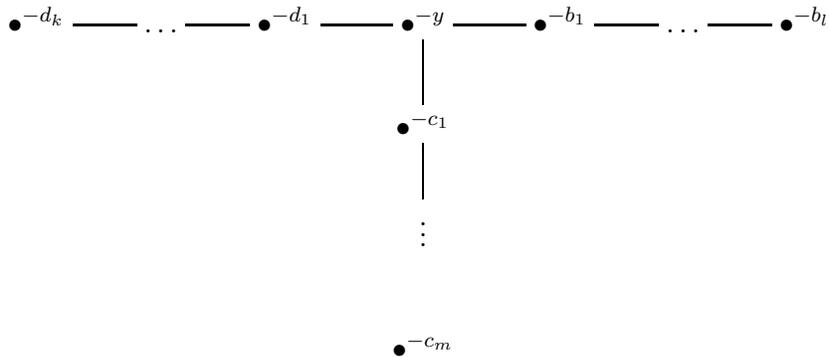
\begin{displaymath}
    \xymatrix{
        \bullet^{-d_k} \ar@{-}[r] & \dots \ar@{-}[r] & \bullet^{-d_1} \ar@{-}[r]& \bullet^{-y} \ar@{-}[r] \ar@{-}[d]& \bullet^{-b_1} \ar@{-}[r] & \dots \ar@{-}[r] & \bullet^{-b_l}\\
				  &			&		&\bullet^{-c_1} \ar@{-}[d] \\
				  &			&		&\vdots \\
				  &			&		&\bullet^{-c_m}\\
	}
\end{displaymath}
where all vertices have genera zero and we require the self intersection numbers satisfies $\frac{n_1}{\lambda_1}=[d_1,\dots,d_k] $, $\frac{n_2}{\lambda_2}=[b_1,\dots,b_l] $ and $\frac{n_3}{\lambda_3}=[c_1,\dots,c_m]$.
We call the vertex with self-intersection $-y$ to be the central vertex.

\begin{defn}\label{Five Types}
We define eight special types of graphs as follows.

Type(N1): empty graph,

Type(N2): linear graph $<n,\lambda>$, for $0 < \lambda < n$, $(n,\lambda)=1$,

Type(N3): one branching point  graph $<y;2,1;n_2,\lambda_2;n_3,\lambda_3>$, where $(n_2, n_3)$ is one of the pairs $(3, 3)$, $(3, 4)$, $(3, 5)$, or $(2, n)$, for some  $n \ge 2$ and $0 < \lambda_i < n_i$, $(n_i,\lambda_i)=1$, and $y \ge 2$,

Type(P1): linear graph $\xymatrix{
        \bullet^{0} \ar@{-}[r] & \bullet^{0}\\
	} $,

Type(P2):  (linear) dual blown up graph $\overline{\Gamma^{(v)}}$ where $\Gamma=<n,n-\lambda>$ is of type (N2) and $v$ is the left-end vertex,

Type(P3): one branching point graph $<3-y;2,1;n_2,n_2-\lambda_2;n_3,n_3-\lambda_3>$, where  $<y;2,1;n_2,\lambda_2;n_3,\lambda_3>$ is of type (N3),

Type(P4):  (one branch point)  dual blown up graph $\overline{\Gamma^{(v)}}$ where $\Gamma$ is of type (N2) and $v$ is not an end vertex, 

Type(P5):  (one or two branch points)  dual blown up graphs $\overline{\Gamma^{(v)}}$ where $\Gamma$ is of type (N3) and $v$ is any vertex in $T$.

\end{defn}

\begin{rmk}
  Type (N) graphs have  $b_2^+=0$ and  type (P) graphs have  $b_2^+=1$.
\end{rmk}

These graphs are going to be our focus for the remaining of the paper.
We remark that the set of type (N2) graphs is the same as the set of linear graphs with all self-intersection less than $-1$.
Therefore, the set of type (P2) graphs is the same as the set of dual blow up at the right-end vertex of (N2) graphs, by symmetry.

\begin{thm}\label{main classification theorem}
Let $\Gamma$ be a graph with finite boundary fundamental group.
If $Q_{\Gamma}$ is not negative definite, then $\Gamma$ is realizable if and only if $\Gamma$ satisfies one of the following conditions.

(A) $\Gamma$ is equivalent to a graph in types (P1), (P2),  (P3), (P4), or

(B) $\Gamma$ is equivalent to a graph $T^{(v)}$ in type (P5) such that

(B)(i) $y \neq 2$, or

(B)(ii) $y=2$ and $v$ is a vertex labeled by a subscript $Y$ in a graph from Figure \ref{Tetrahedral} to Figure \ref{Dihedral},
where $-y$ is the self-intersection of the central vertex of $T$ for (B)(i) and (B)(ii).

In particular, if $\Gamma$ is realizable, then we have $b_2^+(Q_{\Gamma})=1$ ,  $\delta_{\Gamma} \neq 0 $ and $\Gamma$ is strongly realizable. 
\end{thm}

\begin{prop}\label{main classification theorem_convex}
Suppose $\Gamma$ is a graph with finite boundary fundamental group.
If $Q_{\Gamma}$ is negative definite, which means that $b_2^+(Q_{\Gamma})=0$ here,  then $\Gamma$ is equivalent to a graph in type (N).
Moreover, any type (N) graph  can be realized as a resolution graph of an isolated quotient singularity.
\end{prop}

Notice that a graph with finite boundary fundamental group must have non-degenerate intersection form (See Lemma \ref{order}) except the empty graph.
Therefore, $b_2^+(Q_{\Gamma})=0$ is equivalent to $Q_{\Gamma}$ being negative definite.
To be consistent, the empty graph is considered to be negative definite in this paper.

Using Theorem \ref{main classification theorem} and Proposition \ref{main classification theorem_convex}, 
we give the classification of filling divisors and capping divisors with finite boundary fundamental group.

\begin{thm}\label{complete classification}
 Let $(D,\omega)$ be a divisor (not necessarily $\omega$-orthogonal) with finite boundary fundamental group.
 Then $(D,\omega)$ is a capping divisor if and only if it satisifies the positive GS
 criterion and its graph is equivalent to a realizable graph in type (P).
 On the other hand, $(D,\omega)$ is a filling divisor if and only if its graph is equivalent to a graph in type (N).
\end{thm}

\begin{proof}
 The statement for capping divisor follows directly from Proposition \ref{reduce to graph} and Theorem \ref{main classification theorem}.
 
 On the other hand, if $(D,\omega)$ is of type $N$, it is negative definite and hence satisfies the 
 negative GS criterion.
 By Proposition \ref{McLean}, $(D,\omega)$ is convex and we can close it up (by \cite{EtHo02}), thus is a filling divisor.
 
 If $(D,\omega)$ is a filling divisor but not in type (N), then $(D,\omega)$ is in type (P) and has $b_2^+(D)=1$, by Theorem \ref{main classification theorem} and Proposition \ref{main classification theorem_convex}.
 We can close it up to a closed symplectic manifold $(W,\omega)$, which must be rational since it contains the divisor $D$
 (for the reason why $W$ is rational if the graph of $D$ is equivalent to one in type (P3), see \cite{BhOn12}, for the other, see Theorem \ref{McDuff}).
 Therefore, $b_2^+(W)=1$.
 However, $\omega|_{W-P(D)}$ descends to a relative class in the cap, $W-P(D)$, and thus has positive square
 (i.e $[\omega|_{W-P(D)} -d\alpha_c]^2 >0$ for any choice of primitive $\alpha$ of $\omega$ defined near $P(D)$).
 Therefore, $b_2^+(W-P(D)) \ge 1$.
 However, $b_2^+(W)=b_2^+(P(D))+b_2^+(W-P(D))$ as $\partial P(D)$ is a rational homology sphere.
 Contradiction.
\end{proof}

\subsection{Filling Divisors with Finite Boundary $\pi_1$}\label{Filling classification}
We prove Proposition \ref{main classification theorem_convex} in this subsection.

\subsubsection{Topological Restrictions}
We first recall some topological constraints for a configuration to have  finite boundary fundamental group.

\begin{defn} \label{graph theoretic definition}
 Suppose we have a graph $\Gamma$.   The boundary fundamental group of $\Gamma$,  denoted  by $\pi_1(\Gamma)$,  is the fundamental group of the boundary of the plumbing of the configuration represented by $\Gamma$.
 We call $\Gamma$ spherical, cyclic, finite cyclic  if $\pi_1(\Gamma)$ is trivial, cyclic, finite cyclic.

A branch point (or branch vertex) of a graph is a vertex with at least three branches.
A branch at a vertex $v$ also refers to  the sub-graph $\Gamma$ obtained by deleting $v$ and all other branches linking to $v$.

A simple branch $\gamma$ is a branch that is linear.

An extremal branch point is a branch point with only one non-simple branch.

Finally, for a connected sub-graph $\gamma$, $\delta_{\gamma}$ denotes the determinant of the intersection form of $\gamma$.

\end{defn}

\begin{lemma}\label{tree}
 Let $T$ be a graph with finite  $\pi_1(T)$.
Then, all of its vertices have genera zero and $T$ is a finite tree.
\end{lemma}

Therefore, from now on, all vertices are assumed to have genera zero and the number above a vertex is the self-intersection number of the vertex.
Here, we give the concrete representation for boundary fundamental group.

\begin{lemma} \label{representation}(\cite{Hi66})
 Let $T$ be a finite tree such that the genera of all vertices are zero.
Label the vertices as $v_i$ for $i=1, \dots, n$ and let $q_{ij}=[v_i][v_j] \in \mathbb{Z}$ be the $(i,j)^{th}$-entry of the intersection form of $T$.
Then, $\pi_1(T)$ is isomorphic to the free group generated by $e_1,\dots,e_n$ modulo the relations 
$$e_ie_j^{q_{ij}}=e_j^{q_{ij}}e_i,\quad  \hbox{for any } i, j$$ and
$$1=\prod_{1 \le j\le n}e_j^{q_{ij}}, \quad  \hbox{for any } i.$$
\end{lemma}

\begin{lemma} \label{order}   (\cite{Hi66})
 Let $T$ be a finite tree such that the genera of all vertices are zero.
Then, the order of the abelianization of $\pi_1(T)$ is finite if and only if $\delta_T \neq 0$.
In this case, $\delta_T$ equals the order of the abelianization of $\pi_1(T)$. 
\end{lemma}

\begin{eg}\label{non-spherical}
For a type (N2) linear graph $T$, $\delta_T \neq 1$ and hence $\pi_1(T)$ is nontrivial by Lemma \ref{order}. 
\end{eg}

\begin{lemma} \label{key}
 Let $T$ be a minimal tree and $v$ a vertex in $T$.

(i) If $\pi_1(T)$ is cyclic, then there are at most two non-spherical branches at $v$.

(ii) If $\pi_1(T)$ is finite, then there are at most three non-spherical branches at $v$.
Moreover, if there are three non-spherical branches, then they are all finite cyclic.
\end{lemma}

\begin{proof}
 The proof is based on the representation in Lemma \ref{representation} and the group theoretical result in the following lemma. See Lemma 3.1 and 3.2 of~\cite{Sh85}
\end{proof}

\begin{lemma}\label{quotient}
 Let $G_1, \dots, G_n$ be non-trivial groups and let $t_i \in G_i$ be an arbitrary element. Then,

(i) for $n \ge 4$, $G_1 * \dots *G_n/(\prod_{i=1}^{n}t_i=id)$ is infinite.

(ii) for $n \ge 3$, $G_1 * \dots *G_n/(\prod_{i=1}^{n}t_i=id)$ is non-trivial and non-cyclic.

(iii) $G_1*G_2*G_3/(\prod_{i=1}^{3}t_i=id)$ is finite if and only if $G_i$ are all cyclic groups generated by $t_i$ with $(G_1,G_2,G_3)$ isomorphic to one of the following unordered triples 

$(\mathbb{Z}/2\mathbb{Z},\mathbb{Z}/2\mathbb{Z},\mathbb{Z}/k\mathbb{Z})$, ($k \ge 2$), 

$(\mathbb{Z}/2\mathbb{Z},\mathbb{Z}/3\mathbb{Z},\mathbb{Z}/3\mathbb{Z}) $, 
$(\mathbb{Z}/2\mathbb{Z},\mathbb{Z}/3\mathbb{Z},\mathbb{Z}/4\mathbb{Z})$,
or $(\mathbb{Z}/2\mathbb{Z},\mathbb{Z}/3\mathbb{Z},\mathbb{Z}/5\mathbb{Z})$.
\end{lemma}

\begin{lemma}\label{lens space analysis}
 Suppose $T$ is of the form
$\xymatrix@R=1pc @C=1pc{
        T_1 \ar@{.}[r] & \bullet_{v_0} \ar@{-}[r] \ar@{.}[d]& \dots \ar@{-}[r]  & \bullet_{v_r} \ar@{.}[d] \ar@{.}[r] & T_3 \\
		      &		T_2 		&				&	T_4
} 
$
with $r \ge 1$, where $T_i$ are non-spherical branches (not necessarily simple) such that $\delta_{T_1}\delta_{T_2} \neq 0$ and $\delta_{T_3}\delta_{T_4} \neq 0$.
Suppose also that the boundary of the plumbing of  $\xymatrix@R=1pc @C=1pc{
        T_1 \ar@{.}[r] & \bullet_{v_0} \ar@{.}[r] & T_2 \\}
$
 and $\xymatrix@R=1pc @C=1pc{
        T_3 \ar@{.}[r] & \bullet_{v_r} \ar@{.}[r] & T_4 \\}
$ are diffeomorphic to a lens space or $\mathbb{S}^2 \times \mathbb{S}^1$.
Then, $\pi_1(T)$ contains  $\mathbb{Z} \oplus \mathbb{Z} $ as a subgroup.
\end{lemma}

\begin{proof}
 See Lemma 3.3 and 3.4 of~\cite{Sh85}
\end{proof}

It is time to mention the following

{\bf Fact}:   type (N) graphs have finite boundary fundamental group.
To be more precise, 
 (N1) graph is spherical, (N2) graphs are finite cyclic, and (N3) graphs are finite and non-cyclic.

This is  well known in algebraic geometry:  graphs in type (N2) correspond to resolution graphs of cyclic quotient singularities and the graphs in type (N3) correspond to resolution graphs of dihedral, tetrahedral, octahedral and icosahedral singularities (cf. \cite{Br68} Satz $2.11$). 
One can also prove this fact directly by  Lemma \ref{representation} and Lemma \ref{quotient}. 
It is easy for (N2) graphs. And for an  (N3) graph $T$, $\pi_1(T)$  is finite as it  can be realized as a finite extension of a  finite group (basically by Lemma \ref{quotient}), and  it is non-cyclic because it has a  non-cyclic quotient.

\subsubsection{Proof of Proposition \ref{main classification theorem_convex}}\label{classification}
In this subsection we are going to make use of the constraints above to prove Proposition \ref{main classification theorem_convex}. 

\begin{lemma}\label {simple branch}
 Let $T$ be a negative definite, minimal tree with no branch point.  Then $T$ is of type (N1) or (N2).  In particular, $T$ is finite cyclic. 
\end{lemma}

\begin{proof}
A  connected genus zero tree has no branch point, so  it is linear.
Linearity and minimality ensure no $-1$  vertices,  while being negative definite ensures that each vertex has self-intersection less than 0.
So $T$ is  
the empty graph (N1),  or  a  linear graph with all vertices having self-intersection less than $-1$, which  is  a type (N2) graph.
\end{proof}

\begin{lemma}\label{T shape}
 Let $T$ be a minimal tree with exactly one branch point $v$.
Suppose all the self-intersection of vertices in the branches are negative (satisfied if $T$ is negative definite).
Then, $\pi_1(T)$ is finite if and only if $T$ is a (N3) or (P3) graph. 
In particular, $\pi_1(T)$ is not cyclic if $\pi_1(T)$ is finite.
\end{lemma}

\begin{proof}
 See Theorem 4.3 of \cite{Sh85}. The proof is purely topological.
\end{proof}

Now, we deal with  the  case that there are more than $1$ branch points.

\begin{lemma}\label{Tech}
 Let $T$ be a negative definite, minimal tree with   $k \ge 2$  branch points.
Then, $\pi_1(T)$ is non-cyclic and infinite.
\end{lemma}

\begin{proof}
The proof follows  \cite{Sh85} closely.
It is convenient to make two observations first.

{\bf Observation 1}:  For any branch point,  by the minimality assumption,  all self-intersections of vertices in simple branches are less than $-1$.
Therefore  every simple branch of $T$ is not spherical from Example \ref{non-spherical}.

{\bf Observation 2}: 

\begin{lemma}\label{observation}
 Let $T$ be a negative definite, minimal tree with   $k \ge 2$  branch points.
Suppose $\pi_1(T)$ is finite or cyclic.

Let $v$ be a branch point and $\Gamma$ a branch at $v$. 
Suppose  there are at least two non-spherical branches at $v$ other than $\Gamma$ (it is satisfied if $v$ is an extremal branch point and $\Gamma$ is the non-simple branch), then $\Gamma$ is finite cyclic and is either

(a)a negative definite minimal tree with $k-1$ branch points, OR

(b)a negative definite minimal tree with $k-2$ branch points, OR 

(c)not minimal.
 
In case (c), there exists a branch point of $T$, $v_2$, which is linked to $v$, with exactly three branches and the self-intersection of $v_2$ is $-1$. 
\end{lemma}

\begin{proof}

Suppose first that $\pi_1(T)$ is cyclic.  Since there are at least two non-spherical branches at $v$ other than $\Gamma$, $\Gamma$ is spherical by  Lemma \ref{key}(i) applied to  $v$.
If $\pi_1(T)$ is finite, then $\Gamma$ is finite cyclic by  Lemma \ref{key}(ii) applied to $v$.
Therefore in either case $\pi_1(T)$ is finite cyclic.

Moreover, if $\Gamma$ is minimal, it is either in (a) or (b).
If $\Gamma$ is not minimal, it is in (c).
In this case, the only possible $-1$ vertex that can be blown down is the vertex in $\Gamma$ linked to $v$, which we call it $v_2$.
Since $T$ is minimal, $v_2$ has self-intersection $-1$ means that it is a branch point of $T$ but it can be blown down in $\Gamma$ means that it is not a branch point of $\Gamma$.
Therefore, the result follows.  
An extremal branch point satisfies the assumption because there are at least two simple branches at $v$,  which are non-spherical by Observation 1.
\end{proof}

We are going to first establish the claim of Lemma \ref{Tech} for the cases $k=2$ and $k=3$, then prove by contradiction using induction on $k$.
Label the vertices of $T$ as $v_1, \dots, v_m$ with the corresponding self-intersection $s_1, \dots, s_m$.

First suppose $k=2$ and $v_1$, $v_2$ are the two branch points of $T$ with $\pi_1(T)$ cyclic or finite.
If one of $v_1$, $v_2$ has three simple branches, say $v_2$, we denote the non-simple branch at $v_1$ by $\gamma$.
Apply Lemma \ref{observation} to $v_1$, $\gamma$ is in (a).
However, negative definite minimal tree with exactly one branch point is not cyclic (See Lemma \ref{T shape}).
Contradiction.
Thus, both $v_1$ and $v_2$ have only two simple branches.
Let the two simple branches at $v_1$  be $T_1$ and $T_2$ and those at $v_2$ be $T_3$ and $T_4$.
Then, the assumptions of Lemma \ref{lens space analysis} for $v_0=v_1$ and $v_r=v_2$ are satisfied.
Thus, $\pi_1(T)$ contains $\mathbb{Z} \oplus \mathbb{Z}$, contradiction.

For $k=3$, let $v_2$, $v_1$ and $v_3$ be the three branch points of $T$ and suppose $T$ is finite or cyclic.
We have two of the three branch points are extremal, say $v_2$ and $v_3$.
Let $\Gamma$ be the non-simple branch at $v_2$ and $\Gamma'$ be the non-simple branch at $v_3$.
Apply Lemma \ref{observation} at $v_2$, we have $\Gamma$ is not minimal because we have shown that negative definite minimal trees with exactly one or two branch points are not finite cyclic.
Thus, we must have $v_1$ is linked to $v_2$ and $v_1$ has only one simple branch in $T$, which we denote by $T_0$.
Moreover, we have $s_1 = -1$ and by symmetry, $v_1$ is also linked to $v_3$.
Observe that, we must have all vertices in $T_0$ having self-intersection $-2$, otherwise $\pi_1(\Gamma)$ is still not cyclic.

Since $T$ is negative definite, both $s_2$ and $s_3$ are not $-1$. 
If $v_2$ or $v_3$ has three or more simple branches, then we can apply Lemma \ref{key}(ii) at $v_1$ and Lemma \ref{T shape} to the branch at $v_1$ which is minimal and having exactly one branch point to get a contradiction.
In other words, $v_2$ or $v_3$ has exactly two simple branches.
Therefore, we have $T$ is of the following form with all $T_i$ being simple branches.

$
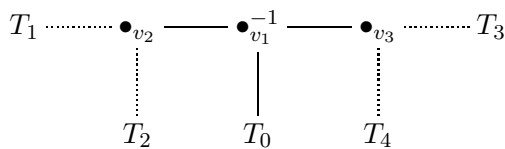
\xymatrix{
        T_1 \ar@{.}[r] & \bullet_{v_2} \ar@{-}[r] \ar@{.}[d]& \bullet^{-1}_{v_1} \ar@{-}[d] \ar@{-}[r]  & \bullet_{v_3} \ar@{.}[d] \ar@{.}[r] & T_3 \\
		      &		T_2 		&		T_0		&	T_4
} 
$

Notice that, both $\Gamma$ and $\Gamma'$ are equivalent to a linear graph because all vertices in $T_0$ have self-intersection $-2$.
Let $\gamma=\xymatrix {T_1 \ar@{-}[r] & \bullet_{v_2} \ar@{-}[r] & T_2}$.

Since $s_2 \neq -1$, we have $\delta_{\gamma} \neq 0$ and $\pi_1(\gamma)$ nontrivial.
Hence, we can apply Lemma \ref{lens space analysis} for $v_0=v_1$ and $v_r=v_3$.
This is because the $T_i$ in Lemma \ref{lens space analysis} are not assumed to be simple branches and $\Gamma'$ is equivalent to a linear graph.
Hence, we get a contradiction.
This finishes the study of $k=3$.

In general, we assume the statement is true for $n <k$ and we deal with the case with $T$ having $k \ge 4$ branch points.
Let $v_1$ be an extremal branch point and $\Gamma$ is the non-simple branch.
The induction hypothesis and Lemma \ref{observation} imply that $\Gamma$ is not minimal.
In particular, there is a branch point of $T$, say $v_2$, is linked to $v_1$ with $s_2=-1$.
If none of the three branches of $v_2$ in $T$ is simple, no matter how we blow down $\Gamma$, its minimal model has at least two branch points, contradicting to the above $k=2$ case or induction hypothesis.
Hence, $v_2$ has a simple branch.

Since $T$ is negative definite and $s_2=-1$, we must have $s_1 \le -2$.
Thus, the branch at $v_2$ containing $v_1$ is not spherical.
Let $\Gamma'$ be the non-simple branch at $v_2$ not containing $v_1$.
Applying Lemma \ref{observation} at $v_2$, we get that $\Gamma'$ is not minimal and hence there exist $v_3$, which is linked to $v_2$ with $s_3=-1$.
The existence of two adjacent vertices, $v_2$ and $v_3$, having self-intersections $-1$ contradicts to $T$ being negative definite.

\end{proof}

Finally, we can complete the proof of Proposition \ref{main classification theorem_convex} by Lemma \ref{simple branch}, Lemma \ref{T shape}, Lemma \ref{Tech} and the classification of isolated quotient surface singularities in \cite{Br68}.

In \cite{Li08} and \cite{BhOz14}, they study the filling of the lens spaces with the canonical contact structure.
These correspond to the graphs in (N2).
It is proved that the divisor filling is the maximal one among all the fillings and all other fillings can be obtained by rational blow downs of the divisor filling \cite{BhOz14}.
Therefore, divisor filling is interesting to  investigate.

\subsection{More Topological Restrictions}

\begin{lemma}\label{0-0}
 Let $T$ be a tree and $v$  a vertex of $T$.
Then, $\pi_1(T)=\pi_1(T^{(v)})$.
\end{lemma}

\begin{proof}
 It is a direct computation using Lemma \ref{representation}.
Label the vertices of $T$ as $v_1,\dots,v_n$ and let $v=v_1$.
Label the two additional self-intersection $0$ vertices in $T^{(v)}$ as $v_{-1}$ and $v_0$, where $v_0$ is the one linked to $v_1$.
We compare $\pi_1(T^{(v)})$ and $\pi_1(T)$.
In terms of generators,  $\pi_1(T^{(v)})$ has two additional generators, namely $e_{-1}$ and $e_0$.
In terms of relations, there are two new relation and one of the relation in $\pi_1(T)$ is changed.
The two new relations are given by $1=e_0$ and $1=e_{-1}e_1$.
The relation $1=\prod_{1 \le j\le n}e_j^{q_{1j}}$ is changed to $1=e_0\prod_{1 \le j\le n}e_j^{q_{1j}}$.
However, we have $1=e_0$, which means the changed relation is actually unchanged.
Moreover, adding the generator $e_{-1}$ with the relation $1=e_{-1}e_1$ is doing nothing to the group so we arrive the conclusion.
\end{proof}

It is time to state the following fact. 

\begin{lemma} \label{P graph}  Type P graphs have finite boundary fundamental groups. More precisely, 
(P1) graph is spherical, (P2) and (P4) graphs are finite cyclic,   (P3) and (P5) graphs are finite and non-cyclic. 
\end{lemma}

\begin{proof}
Clear for the (P1) graph.   Since (N2) graph is finite cyclic, so are (P2) and (P4) graphs by  by Lemma \ref{0-0}.
(P3) and (P5)  graphs and  finite and non-cyclic by Lemmas \ref{T shape} and \ref{0-0}.

\end{proof} 

In \cite{Ne81}, it is mentioned that the dual blow up (which is called blow up there) does not change the oriented diffeomorphism type of the boundary of the plumbing.

\begin{lemma} \label{00}
 Let $T$ be a minimal tree and $v$ a vertex in $T$.
Suppose $\gamma$ is a simple branch at $v$ with $\gamma$ not equivalent to $\xymatrix{\bullet^0 \\} $ and some of the vertices having non-negative self-intersection.
Then, $T$ is equivalent to a minimal tree $T'$ obtained by replacing the branch $\gamma$ by an equivalent branch  $\xymatrix@R=1pc @C=1pc{
        \bullet_{v'} \ar@{-}[r] & \dots \ar@{-}[r]  & \bullet^{0} \ar@{-}[r] & \bullet^{0}_{x} \\
	} $.
where $v'$ is the vertex linked to $v$, $x$ is an end vertex and the self-intersection of $v$ may possibly be changed. 
\end{lemma}

\begin{proof}
 We first make the following observations.
If $T$ has a sub-tree of the form $\xymatrix@R=1pc @C=1pc{
        \dots \ar@{-}[r] & \bullet^{q} \ar@{-}[r] & \bullet^{0} \ar@{-}[r] & \bullet^{b} \ar@{-}[r]& \dots \\
	} $, then $T$ is equivalent to $T'$ where $T'$ is obtained by changing the sub-tree to  $\xymatrix@R=1pc @C=1pc{
        \dots \ar@{-}[r] & \bullet^{q+1} \ar@{-}[r] & \bullet^{0} \ar@{-}[r] & \bullet^{b-1} \ar@{-}[r]& \dots \\
	} $.
Also, if $T$ has a sub-tree of the form $\xymatrix@R=1pc @C=1pc{
        \dots \ar@{-}[r] & \bullet^{q} \ar@{-}[r] & \bullet^{0} \\
	} $, then $T$ is equivalent to $T'$ where $T'$ is obtained by changing the sub-tree to  $\xymatrix@R=1pc @C=1pc{
        \dots \ar@{-}[r] & \bullet^{q+1} \ar@{-}[r] & \bullet^{0} \\
	} $.

 Let $u$ be the vertex in $\gamma$ having self-intersection non-negative,
By possibly blowing up successively at edges linked to $u$, we assume that $u$ has self-intersection $0$.
By the first observation, we can make the end vertex $w$ of $\gamma$ to have self-intersection $0$.
Since $\gamma$ is not equivalent to $\xymatrix{
        \bullet_{0} \\
	} $, there is a vertex $w'$ which is different from $v$ and is linked to the end vertex $w$.
Now, we blow down all $-1$ vertices in $\gamma-\{w,w'\}$.
By the second observation and the fact that $w$ has self-intersection $0$, we obtain $T'$ as we want.
For details, see Lemma 4.1 of \cite{Sh85}.
\end{proof}

\begin{lemma}\label{negative definite spherical}
 A spherical negative definite tree is not minimal.
In other words, it is equivalent to an empty graph.
\end{lemma}

\begin{proof}
 See Section $3$ of \cite{Hi66}
\end{proof}

\subsection{Symplectic Restrictions}\label{symplectic}
We are going to provide symplectic input to give more constraints on the trees that we are interested in.
In this subsection, $L_i$ appearing as a superscript of a vertex represents the homology class of the corresponding sphere. 
First, we recall a theorem of McDuff.

\begin{thm}\label{McDuff}(\cite{Mc90})
 Let $M$ be a closed symplectic 4-manifold. Suppose there exist an embedded symplectic sphere $C$ with positive (resp. zero) self-intersection.
Then, $M$ is symplectic rational (resp. ruled).
Moreover, if $(M,C)$ is relatively minimal and $[C]^2=0$, then there exists a symplectic deformation equivalent from $(M,C)$ to $(N,F)$, where $N$ is a symplectic sphere bundle over a closed symplectic surface and $F$ is a fibre.
\end{thm}

\begin{corr}\label{E}
 Let $M$ be a symplectic 4-manifold. Suppose there exist an embedded symplectic sphere $C$ with

(i) positive self-intersection, or

(ii) zero self-intersection and there exists another embedded symplectic sphere $C'$ that intersect $C$ transversally once.

Then $H_1(M)=0$, $M$ is rational and $b_2^+=1$.
\end{corr}

\begin{lemma}\label{R}
Let $D$ be a symplectic divisor in a closed symplectic manifold $W$.
Then, the graph of $D$ does not have a sub-tree of the form,
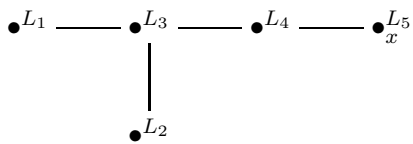
\begin{displaymath}
    \xymatrix{
        \bullet^{L_1} \ar@{-}[r] & \bullet^{L_3} \ar@{-}[d] \ar@{-}[r] & \bullet^{L_4} \ar@{-}[r] & \bullet^{L_5}_x\\
                             & \bullet^{L_2} \\
	}
\end{displaymath}
with $[L_3]^2=-1$ and $[L_5]^2 \ge 0$, and all vertices having genera zero.
\end{lemma}

\begin{proof}
By Corollary \ref{E}, we get $H_1(W)=0$ and $W$ is rational.

Without loss of generality, we may assume $[L_5]^2=0$ (by possibly blowing up at regular points for the sphere representing $x$).
Since $[L_3]^2=-1$, $W$ is not minimal and thus symplectomorphic deformation equivalent to blown-up of a Hizerburch surface, and $[L_5]$ is the fibre class of the Hizerburch surface by Theorem \ref{McDuff}.
Let $\{f,s,e_1,\dots,e_N\}$ be a basis for $H_2(W)$ such that $f$, $s$ and $e_i$ correspond to the fibre class, section class and the exceptional classes, respectively.
Suppose $s^2=n$.
Then, we recall the first chern class of the Hirzeburch surface is $(2-n)f+2s$, thus the first chern class of $W$ is $c_1(W)=(2-n)f+2s-e_1-\dots-e_N$. 
Moreover, $e_i$ can be a prior chosen so that $[L_3]=e_1$.

Suppose there is an embedded symplectic sphere in $W$ with class $[S]=\alpha f+\beta s +a_1e_1+\dots+a_Ne_N$.
Then, adjunction formula gives $$(2-n)\beta+2\alpha +2n\beta+a_1+\dots+a_N=2\alpha \beta+\beta^2n-a_1^2-\dots-a_N^2+2.$$
Suppose $[S]f=1$.
Then $\beta=1$ and the formula reduces to 
$$a_1^2+\dots+a_N^2+a_1+\dots+a_N=0$$ and hence $a_i=0$ or $-1$ for all $i$.
Suppose on the contrary $[S]f=0$.
Then $\beta=0$ and the formula reduces to 
$$2\alpha+a_1^2+\dots+a_N^2+a_1+\dots+a_N=2.$$

Now, we want to study the homology of $L_i$ and draw contradiction.
We recall that $[L_5]=f$ and $[L_3]=e_1$.
Since $1=[L_4][L_5]=[L_4]f$, we apply the adjunction formula derived above and get
$[L_4]=\alpha f+s+\epsilon_1 e_1+\dots+\epsilon_N e_N$ for some $\alpha$, where $\epsilon_i$ equals $0$ or $-1$ for all $i$.
Since $1=[L_4][L_3]=[L_4]e_1$, $[L_4]$ is of the form $\alpha f+s-e_1+\epsilon_2 e_2+\dots+\epsilon_N e_N$.

If we write $[L_1]=\overline{\alpha} f+\beta s +a_1e_1+\dots+a_Ne_N$, then $[L_1]f=0$, $[L_1]e_1=1$ and adjunction imply
$[L_1]=\overline{\alpha} f-e_1+a_2e_2+\dots+a_Ne_N$ and
\begin{eqnarray}\label{1}
 2\overline{\alpha}+a_2^2+\dots+a_N^2+a_2+\dots+a_N=2
\end{eqnarray}
Moreover, by $[L_1][L_4]=0$, we have
\begin{eqnarray}\label{2}
 \overline{\alpha}-1-\epsilon_2a_2-\dots-\epsilon_N a_N=0
\end{eqnarray}
By equations \ref{1} and \ref{2}, we have 
\begin{eqnarray*}
 &&a_2^2+\dots+a_N^2+(1+2\epsilon_2)a_2+\dots+(1+2\epsilon_N)a_N \\
 &=& a_2(a_2+(-1)^{\epsilon_2})+\dots+a_N(a_N+(-1)^{\epsilon_N})  \\
 &=&0
\end{eqnarray*}
Therefore, $a_i=0,-1$ if $\epsilon_i=0$ and $a_i=0,1$ if $\epsilon_i=-1$, for all $i$.

Similarly, $[L_2]=\overline{\beta}f-e_1+b_2e_2+\dots+b_Ne_N$ with 
$b_i=0,-1$ if $\epsilon_i=0$ and $b_i=0,1$ if $\epsilon_i=-1$, for all $i$.

By $[L_1][L_2]=0$, we have $-1-a_2b_2-\dots-a_Nb_N=0$.
When $\epsilon_i=0$, we have both $a_i$ and $b_i$ equals $0$ or $-1$, thus $a_ib_i \ge 0$.
Similarly, if $\epsilon_i=-1$, we still have $a_ib_i \ge 0$.
Therefore, $-1-a_2b_2-\dots-a_Nb_N=0$ gives a contradiction.
\end{proof}

By the previous two Lemmas, we can now state one more basic consequence.

\begin{lemma} \label{0}
 Let $T$ be a minimal graph of an embeddable divisor with finite $\pi_1$.
Then, at any branch point $v$, no branch can be a single vertex $u$ with self-intersection zero.
\end{lemma}

\begin{proof}
 Suppose there is such a vertex  $u$.
Since $\pi_1(u)$ is infinite and $v$ has at least three branches, by Lemma \ref{key}(ii), there is a spherical branch at $v$, which we call $\gamma$.
By Theorem \ref{McDuff} again, the homology class $[u]$ represents a fiber class.
Suppose $b_2^+(\gamma) \neq 0$, then there exists a class $[P]$ which is a linear combination of homology classes of vertices of $\gamma$ such that $[P]^2 >0$.
Moreover, $[P][u]=0$.
A basis for blown-up Hizerburch surface is given by $\{[u],s,e_1,\dots,e_n\}$, where $s$ is a section class and $e_i$ are the exceptional classes.
Therefore, $[P][u]=0$ implies $[P]$ is a linear combination of $[u]$ and $e_i$.
As a result, we have $[P]^2 \le 0$, which is a contradiction.
Therefore, $b_2^+(\gamma)=0$.
Since $\delta_{\gamma} \neq 0$, $b_2^+(\gamma)=0$ implies $\gamma$ is negative definite (Lemma \ref{order}).
Hence, by Lemma \ref{negative definite spherical}, $\gamma$ is not minimal.
Therefore, $T$ has a sub-tree of the form,
\begin{displaymath}
    \xymatrix{
        \bullet^{L_1} \ar@{-}[r] & \bullet^{-1}_w \ar@{-}[d] \ar@{-}[r] & \bullet^{L_4}_v \ar@{-}[r] & \bullet^{0}_u\\
                             & \bullet^{L_2} \\
	}
\end{displaymath}
Contradiction to Lemma \ref{R}.
\end{proof}

\begin{lemma}
Suppose $T$ is a realizable minimal tree with $\pi_1(T)$ being finite.
Suppose also that there is a non-negative self-intersection vertex in a simple branch of $T$.
Then, $T$ is equivalent to $T'^{(v)}$ for another minimal tree $T'$.
\end{lemma}

\begin{proof}
 We can apply Lemma \ref{0} to ensure that the proof of Lemma \ref{00} goes through and hence the result follows.  
\end{proof}

By Corollary \ref{E}, $T$ has $b_2^+=1$ and hence $T'$ is negative definite.
Moreover, by Lemma \ref{representation}, we can see that $\pi_1(T)=\pi_1(T')$ and hence finite.
Therefore, what we need to do next is the classification minimal tree $T$ with $\pi_1(T)$ being finite and all self-intersection in simple branches being negative.
By Lemma \ref{T shape}, we just need to consider the case that there are more than $1$ branch point.

\begin{lemma}\label{Tech2}
 Let $T$ be a minimal realizable graph.
Suppose $k \ge 2$ be the number of branch points of $T$.
Suppose also all self-intersection of vertices in simple branches are less than $-1$.
Then, $\pi_1(T)$ is non-cyclic and infinite.
 
In particular, if $T$ is minimal, non-negative definite and $\pi_1(T)$ is finite, then $T$ is equivalent to a type (P) graph.
\end{lemma}

\begin{proof}
 The essence of the proof is the same as in Lemma \ref{Tech} (See \cite{Sh85}).
Since we do not assume our tree $T$ satisfies $b_2^+=1$, we cannot apply the result in \cite{Sh85} directly.
Instead, we need to use Corollary \ref{E} to guarantee $b_2^+=1$ whenever it is needed in the proof.
This is required when we study the case that $T$ has $k=3$ branch points.
We remark that in that case, there are two $-1$ vertices linked to each other so that blowing down one of them gives us an embedded symplectic sphere with self-intersection $0$.
Therefore, we can apply Corollary \ref{E} in that case to finish the proof of the first assertion.
Moreover, by Lemma \ref{T shape}, the second assertion also follows.
\end{proof}

\subsection{Proof of Theorem \ref{main classification theorem}}
 Having Lemma \ref{Tech2}, we can now focus on the study of type (P) graphs.
Type (P1), (P2), (P3) are relatively easy to study and we are going to first go through it.
Then, complete classification of realizability of type (P4) and (P5) graphs are given, which in turn completes the proof of Theorem \ref{main classification theorem}.
Finally, we are going to show that many graphs in type (P5) do not have their conjugate. 

\subsubsection{Type (P1), (P2), (P3)}\label{Type (P1) to (P3)}

We start with type (P1).
By Example \ref{0-0>1}, we have $\xymatrix{ \bullet^{0} \ar@{-}[r] & \bullet^{0}\\ } $ is equivalent to $\xymatrix{ \bullet^{1}} $.
Then, by Example \ref{single vertex}, it is strongly realizable.
Moreover, it corresponds to a capping divisor of the empty graph.

Instead of answering the realizability of graphs in type (P2) and (P3) directly, we observe that graphs in type (2) and type (3) are all considered in \cite{BhOn12}.
The  (P2) graphs correspond to compactifying divisors for cyclic quotient singularities and the  (P3) graphs correspond to compactifying divisors for the dihedral, tetrahedral, octahedral and icosahedral singularities.
In particular, all  (P2) and (P3) graphs are strongly realizable.

The only less obvious correspondence between (P2), (P3) graphs and the graphs considered in \cite{BhOn12} are (P3) graphs of the form  $<y;2,1;n_2,\lambda_2;n_3,\lambda_3>$  with $y \le 1$ and $(n_2, n_3)=(2, n)$.
We denote the following graph in \cite{BhOn12} as $(c,c_1, \dots, c_k)$, where $[c,c_1,\dots,c_k]=\frac{n}{n-q}>1$ and $c,c_i \ge 2$ for all $i$.
These are the graphs of compactifying divisors of dihedral singularities used in \cite{BhOn12}.  

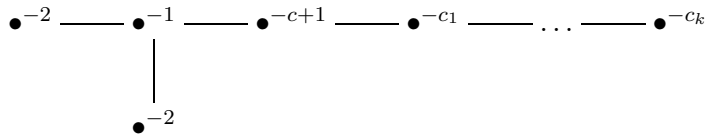
\begin{displaymath}
    \xymatrix{
        \bullet^{-2} \ar@{-}[r]& \bullet^{-1} \ar@{-}[r] \ar@{-}[d]& \bullet^{-c+1} \ar@{-}[r] & \bullet^{-c_1} \ar@{-}[r] & \dots \ar@{-}[r] & \bullet^{-c_k}\\
			&\bullet^{-2} \\
	}
\end{displaymath}

Observe that $(c,c_1, \dots, c_k)$ is the same as $<1;2,1;2,1;q,n-q>$ if one extends the definition to the case that $q < n-q$.

\begin{lemma}
 Every graph $(c,c_1, \dots, c_k)$  in \cite{BhOn12} with $c,c_i \ge 2$ is equivalent to a (P3) graph  $<y;2,1;2,1;n,\lambda>$ with $y \le 1$ and $0 < \lambda < n$ and vice versa.
\end{lemma}

\begin{proof}
Observe that $<y;2,1;2,1;n,\lambda>$ is equivalent to
$$    
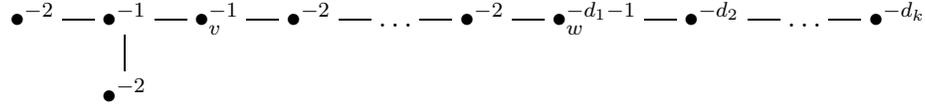
\xymatrix@R=1pc @C=1pc{
        \bullet^{-2} \ar@{-}[r]& \bullet^{-1} \ar@{-}[r] \ar@{-}[d]& \bullet^{-1}_{v} \ar@{-}[r] & \bullet^{-2} \ar@{-}[r] & \dots \ar@{-}[r] & \bullet^{-2}  \ar@{-}[r] & \bullet^{-d_1-1}_{w} \ar@{-}[r] & \bullet^{-d_2} \ar@{-}[r] & \dots \ar@{-}[r] & \bullet^{-d_k}\\
			&\bullet^{-2} \\
	}
$$
where there are $-y$ many self-intersection $-2$ spheres between the spheres named $v$ 
and $w$ as subscript and $\frac{n}{\lambda}=[d_1,\dots,d_k]$.

Hence, it is of the form $(c,c_1, \dots, c_{-a+k-1})$ with 
$$[c,c_1,\dots,c_{-d},c_{-d+1},c_{-d+2}, \dots,c_{-d+k-1}]
=[2,2,\dots,2,d_1+1,d_2,\dots,d_k].$$
This defines a map from the set of $<y;2,1;2,1;n,\lambda>$ with $y \le 1$ and $0 < \lambda < n$ to the set of $(c,c_1, \dots, c_k)$ and $c,c_i \ge 2$.
Moreover, the inverse exists.
\end{proof}

Knowing that the graphs in type (P1) to type (P3) are realizable, as remarked before,  we can determine whether a divisor with its graph 
being in type (P1) to type (P3) is a capping divisor or not, by Proposition \ref{reduce to graph}.

We remarked that for any graph $T$ in (P2), there is a unique (N2) graph $T'$ such that the dual blow up of $T'$ at the left-end vertex is $T$.
In fact, by symmetry, there is also a unique (N2) graph $T"$ such that the dual blow up of the right-end vertex  of $T"$ is $T$.
To be more precise, $T"$ is obtained from $T'$ by rewriting the self-intersections from left to right to from right to left.

\subsubsection{Type (P4) and (P5)}\label{Type (P4) and (P5)}

Suppose a graph $\overline{T^{(v)}}$ in type (P4) or (P5) admits a realization $D$ in a closed symplectic manifold $W$.
By Theorem \ref{McDuff} again, the existence of the self-intersection $1$ sphere implies that $W$ is rational.

After preparation, now we are ready to study the realizability of type (P4) and (P5) graphs.
We recall that for a type (P4) graph, $T^{(v)}$, $v$ is not an end vertex of $T$.
We show that all (P4) graphs are realizable but we some graphs in type (P5) are not realizable. 

\begin{lemma}\label{standard realization}
  Suppose $T=<n,\lambda>= \xymatrix{ \bullet^{-d_1} \ar@{-}[r] & \bullet^{-d_2} \ar@{-}[r] & \dots \ar@{-}[r] & \bullet^{-d_k}\\ }$ 
and $v_1$ is the vertex with self-intersection $-d_j$ and $j \neq 1,k$.
Then, $\overline{T^{(v_1)}}$ and hence the (P4) graph $T^{(v_1)}$ is realizable by some symplectic divisor $D$.

On the other hand, suppose $T_2=<y;2,1;n_2,\lambda_2;n_3,\lambda_3>$ is a graph in type (N3).
Then, $T_2^{(v)}$ is realizable if $y \neq 2$.

\end{lemma}

\begin{proof}
 We start with two algebraic lines in $\mathbb{CP}^2$ and call it $C_0$ and $C_1$.
We blow up $\mathbb{CP}^2$ at $d_j$ distinct regular points at $C_1$ away from the intersection point of $C_0$ and $C_1$.  
Label the exceptional spheres as $E_1, E_{j+1}, E_1^{2}, \dots, E_1^{d_j-1}$.
Call the proper transform of $C_0$ and $C_1$ as $C_0$ and $C_1$ again.
Moreover, we call $E_1$ and $E_{j+1}$ to be $C_2$ and $C_{j+1}$, respectively.
Then, we blow up $d_{j-1}-1$ many distinct regular points on $C_2$ that are away from the intersection points.
Denote the exceptional spheres as $E_2, E_2^{2}, E_2^{3}, \dots, E_2^{d_{j-1}-1}$.
Call the proper transform of $C_i$'s as $C_i$'s again and we call $E_2$ as $C_3$.
We keep blowing up at regular points inductively and similarly on $C_3$ up to $C_{j-1}$ and denotes $E_{j-1}$ as $C_j$.
Now,we blow up $C_j$ at $d_1-1$ many distinct regular points and call the exceptional spheres as $E_j, E_j^{2}, E_j^{3}, \dots, E_j^{d_{1}-1}$.
This time, we do not let $C_{j+1}$ to be $E_j$ (we actually defined $C_{j+1}=E_{j+1}$).
We get the second branch $C_2 \cup \dots \cup C_j$ of $C_1$ ($C_0$ is the first branch of $C_1$).
Now, we blow up similarly for $E_{j+1}=C_{j+1}$ and we can get the last branch $C_{j+1} \cup \dots \cup C_k$ of $C_1$.
This gives an embeddable symplectic divisor $D=C_0 \cup \dots \cup C_k$ that realize $\overline{T^{(v_1)}}$.

For the type (P5) graph, we also consider $\overline{T_2^{(v)}}$ instead of $T_2^{(v)}$.
We assume that $v$ is a vertex with two branches.
The case when $v$ is the vertex with three branches is similar.

We start with $\mathbb{CP}^2$ with $D$ being union of two distinct $\mathbb{CP}^1$, denoted by $C_1$ and $C_2$.
Without loss of generality, we can assume $\overline{T_2^{(v)}}$ is of the form
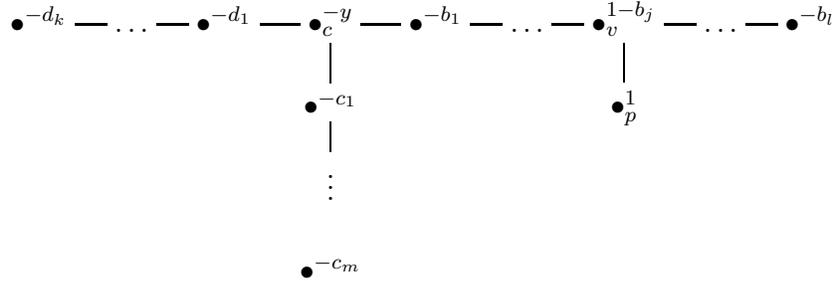
\begin{displaymath}
    \xymatrix@R=1pc @C=1pc{
        \bullet^{-d_k} \ar@{-}[r] & \dots \ar@{-}[r] & \bullet^{-d_1} \ar@{-}[r]& \bullet^{-y}_c \ar@{-}[r] \ar@{-}[d]& \bullet^{-b_1} \ar@{-}[r] & \dots \ar@{-}[r] & \bullet^{1-b_j}_v \ar@{-}[r] \ar@{-}[d]& \dots \ar@{-}[r] & \bullet^{-b_l}\\
				  &			&		&\bullet^{-c_1} \ar@{-}[d]                  &                           &                  & \bullet^{1}_p                           &\\
				  &			&		&\vdots \\
				  &			&		&\bullet^{-c_m}\\
	}
\end{displaymath}

Let we denote the sphere with self-intersection $-d_s$, $-b_s$ ($s \neq j$) and $-c_s$ as $C^\alpha_s$, $C^\beta_s$ ($s \neq j$) and $C^\gamma_s$, respectively.
Also, denote the sphere with self-intersection $-y$, $1-b_j$ and $1$ as $C^c$ $C^v$ and $C^p$, respectively.

It suffices to consider the case that $d_s=b_s=c_s=2$ for all $s$ because we can obtain the other cases by extra blow-ups.
It is possible to obtain $\overline{T_2^{(v)}}$ with homology of the spheres indicated below by iterative blow-ups starting from $D$ similar to that in Lemma \ref{standard realization}.
Here $h$ is the hyperplane class and $e_i$'s are the exceptional classes resulting from single blow-ups.

$[C^\alpha_s]=e_{i^{\alpha_{s-1}}_2}-e_{i^{\alpha_{s}}_2}$ for $2 \le s \le k$;

$[C^\alpha_1]=e_{i^{\alpha_1}_1}-e_{i^{\alpha_1}_2}$;

$[C^\gamma_s]=e_{i^{\gamma_{s-1}}_2}-e_{i^{\gamma_{s}}_2}$ for $2 \le s \le m$;

$[C^\gamma_1]=e_{i^{\gamma_1}_1}-e_{i^{\gamma_1}_2}$;

$[C^c]=e_{i^{\beta_1}_2}-e_{i^{\alpha_1}_1}-e_{i^{\gamma_1}_1}$;

$[C^\beta_s]=e_{i^{\beta_{s+1}}_2}-e_{i^{\beta_{s}}_2}$ for $1 \le s \le j-2$;

$[C^\beta_{j-1}]=e_{i^{\beta_{j-1}}_1}-e_{i^{\beta_{j-1}}_2}$;

$[C^v]=h-e_{i^{\beta_{j-1}}_1}-e_{i^{\beta_{j+1}}_1}$;

$[C^\beta_{j+1}]=e_{i^{\beta_{j+1}}_1}-e_{i^{\beta_{j+1}}_2}$;

$[C^\beta_s]=e_{i^{\beta_{s-1}}_2}-e_{i^{\beta_{s}}_2}$ for $j+2 \le s \le l$, and

$[C^p]=h$.

This shows the existence of a realization for the type (P5) graph.

\end{proof}

To aid the non-realizablility study of some type (P5) graphs, we recall a combinatorical argument given by Lisca (See Proposition 4.4 of \cite{Li08}).

\begin{prop}\label{combinaotric Lisca}
 Suppose $W=\mathbb{CP}^2 \# N \overline{\mathbb{CP}^2}$ equipped with a symplectic form $\omega$ coming from blown-up of the Fubini-Study form $\omega_{FS}$.
Let $D=C_0 \cup C_1 \cup \dots \cup C_k$ be a symplectic divisor with linear graph and $C_0$ corresponds to one of the two end vertices.
Suppose the self-intersection of $C_i$ is $-b_i$ for $2 \le i \le k$, $[C_0]^2=1$ and $[C_1]^2=1-b_1$, where $b_i \ge 1$ for all $i$.

Suppose $\{h,e_1,\dots,e_N\} \subset H_2(\mathbb{CP}^2\#N\overline{\mathbb{CP}^2};\mathbb{Z})$ forms an orthogonal basis with $h$ being the line class and $e_i^2=-1$.
Assume also $[C_0]=h$.

Then, $[C_1]=h-e_{i^1_1}-e_{i^1_2}-\dots-e_{i^1_{b_1}}$ and $[C_j]=e_{i^j_1}-e_{i^j_2}-\dots-e_{i^j_{b_j}}$ for $2 \le j\le k$,
where, for any $\alpha$, $e_{i^\alpha_m} \neq e_{i^\alpha_n}$ for $m \neq n$.  
\end{prop}

\begin{rmk}
 From now on, when we write the homology of a sphere, say $C$, we might simply write $h-e.-e.-\dots-e.$ and $e.-e.-\dots-e.$ to represent the homology class of $C$.
In this case, the different $e.$'s in $[C]$ are understood to be distinct exceptional classes as in the conclusion of the Proposition \ref{combinaotric Lisca}.
\end{rmk}

\begin{lemma}\label{realizable}
 Suppose $T=<y;2,1;n_2,\lambda_2;n_3,\lambda_3>$ is a graph in type (N3).
If $v$ is the vertex with three branches, then $T^{(v)}$ is realizable if and only if $y \neq 2$.
\end{lemma}

\begin{proof}
The realizability part is already covered by Lemma \ref{standard realization} so we are going to show the other direction.
Suppose $y=2$ and, on the contrary, there were a realization of $T^{(v)}$ in a closed symplectic manifold.
Then, we have $\overline{T^{(v)}}$ is also realizable and we have the following graph.

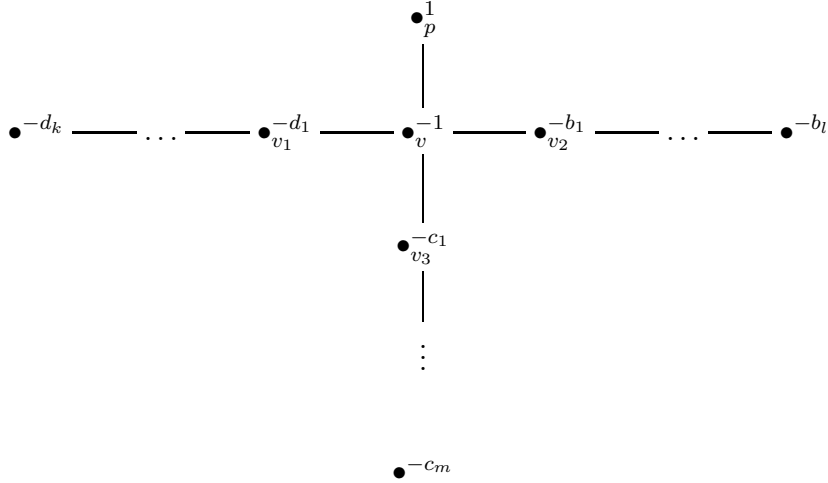
\begin{displaymath}
    \xymatrix{
                                  &                  &                            &  \bullet^{1}_p \ar@{-}[d] \\
        \bullet^{-d_k} \ar@{-}[r] & \dots \ar@{-}[r] & \bullet^{-d_1}_{v_1} \ar@{-}[r]& \bullet^{-1}_v \ar@{-}[r] \ar@{-}[d]& \bullet^{-b_1}_{v_2} \ar@{-}[r] & \dots \ar@{-}[r] & \bullet^{-b_l}\\
				  &			&		&\bullet^{-c_1}_{v_3} \ar@{-}[d]    \\
				  &			&		&\vdots \\
				  &			&		&\bullet^{-c_m}\\
	}
\end{displaymath}

By Theorem \ref{McDuff} we can assume the positive sphere (the one with subscript $p$) has homology class $h$.
Then, the only vertex with 4 branches ($v$) has to have homology class of the form $h-e_1-e_2$, by Proposition \ref{combinaotric Lisca}.
Here, as usual, $e_1$ and $e_2$ are exceptional classes formed by blowups.

We recall that Proposition \ref{combinaotric Lisca} ensure that the vertices $v_i$ has homology of the form $e_{j_1}-e_{j_2}-\dots-e_{j_t}$ for some distinct $e_{j_s}$ $1 \le s \le t$.
To give the positive one contribution of the intersection of vertex $v_i$ with $v$, modulo symmetry, two of three vertices, $v_1$, $v_2$ and $v_3$ has homology class of the form $e_1-e.-\dots-e.$, where $e.$ are distinct exceptional classes not equal to $e_1$.
However, it contradict to the zero intersection of any pair of $v_i$, $i=1,2,3$.
\end{proof}

 Using the same line of reasoning, one can determine completely which graph is realizable and which is not and we put the results in the following.
Therefore, the proof of Theorem \ref{main classification theorem} is completed.

A vertex with subscribe $Y$ indicates that if it is $v$, then the corresponding $T^{(v)}$ is realizable.
Otherwise, the subscribe is $X$.

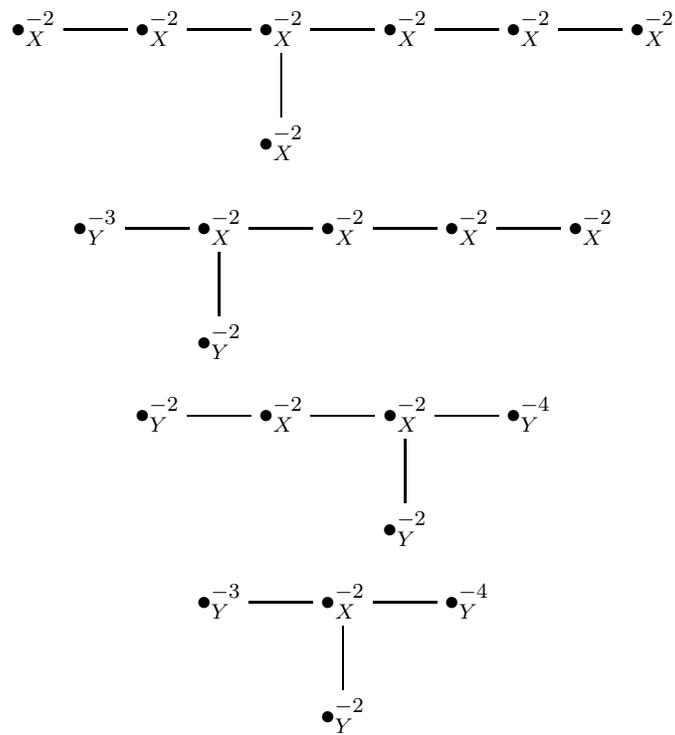
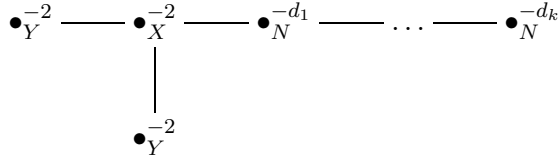
\begin{fig}\label{Tetrahedral}
 Tetrahedral
$$    \xymatrix{
       \bullet^{-2}_X \ar@{-}[r] & \bullet^{-2}_X \ar@{-}[r]& \bullet^{-2}_X \ar@{-}[r] \ar@{-}[d]& \bullet^{-2}_X \ar@{-}[r] & \bullet^{-2}_X \\
			        &                       &\bullet^{-2}_X \\
	}
$$

$$    \xymatrix{
       \bullet^{-2}_Y \ar@{-}[r] & \bullet^{-2}_X \ar@{-}[r]& \bullet^{-2}_X \ar@{-}[r] \ar@{-}[d]& \bullet^{-3}_Y \\
			        &                       &\bullet^{-2}_Y \\
	}
$$

$$    \xymatrix{
        \bullet^{-3}_Y \ar@{-}[r]& \bullet^{-2}_X \ar@{-}[r] \ar@{-}[d]& \bullet^{-3}_Y \\
			            &\bullet^{-2}_Y \\
	}
$$
\end{fig}

\begin{fig}\label{Octahedral}
 Octahedral
$$    \xymatrix{
       \bullet^{-2}_X \ar@{-}[r] & \bullet^{-2}_X \ar@{-}[r]& \bullet^{-2}_X \ar@{-}[r] \ar@{-}[d]& \bullet^{-2}_X \ar@{-}[r] & \bullet^{-2}_X \ar@{-}[r] & \bullet^{-2}_X \\
			        &         &              \bullet^{-2}_X \\
	}
$$

$$    \xymatrix{
        \bullet^{-3}_Y \ar@{-}[r]& \bullet^{-2}_X \ar@{-}[r] \ar@{-}[d]& \bullet^{-2}_X \ar@{-}[r] & \bullet^{-2}_X \ar@{-}[r] & \bullet^{-2}_X \\
			                              &\bullet^{-2}_Y \\
	}
$$

$$    \xymatrix{
       \bullet^{-2}_Y \ar@{-}[r] & \bullet^{-2}_X \ar@{-}[r]& \bullet^{-2}_X \ar@{-}[r] \ar@{-}[d]& \bullet^{-4}_Y \\
			        &                     &\bullet^{-2}_Y \\
	}
$$

$$    \xymatrix{
        \bullet^{-3}_Y \ar@{-}[r]& \bullet^{-2}_X \ar@{-}[r] \ar@{-}[d]& \bullet^{-4}_Y \\
			                              &\bullet^{-2}_Y \\
	}
$$
 
\end{fig}

\begin{fig}\label{Icosahedral}
 Icosahedral
$$    \xymatrix{
       \bullet^{-2}_X \ar@{-}[r] & \bullet^{-2}_X \ar@{-}[r]& \bullet^{-2}_X \ar@{-}[r] \ar@{-}[d]& \bullet^{-2}_X \ar@{-}[r] & \bullet^{-2}_X \ar@{-}[r] & \bullet^{-2}_X \ar@{-}[r] & \bullet^{-2}_X \\
			        &                       &\bullet^{-2}_X \\
	}
$$

$$    \xymatrix{
       \bullet^{-2}_X \ar@{-}[r] & \bullet^{-2}_X \ar@{-}[r]& \bullet^{-2}_X \ar@{-}[r] \ar@{-}[d]& \bullet^{-2}_X \ar@{-}[r] & \bullet^{-3}_X \\
			        &                       &\bullet^{-2}_X \\
	}
$$

$$    \xymatrix{
        \bullet^{-3}_Y \ar@{-}[r]& \bullet^{-2}_X \ar@{-}[r] \ar@{-}[d]& \bullet^{-2}_X \ar@{-}[r] & \bullet^{-2}_X \ar@{-}[r] & \bullet^{-2}_X \ar@{-}[r] & \bullet^{-2}_X \\
			                               &\bullet^{-2}_Y \\
	}
$$

$$    \xymatrix{
       \bullet^{-2}_X \ar@{-}[r] & \bullet^{-2}_X \ar@{-}[r]& \bullet^{-2}_X \ar@{-}[r] \ar@{-}[d]& \bullet^{-3}_Y \ar@{-}[r] & \bullet^{-2}_Y \\
			        &                       &\bullet^{-2}_Y \\
	}
$$

$$    \xymatrix{
        \bullet^{-3}_Y \ar@{-}[r]& \bullet^{-2}_X \ar@{-}[r] \ar@{-}[d]& \bullet^{-2}_X \ar@{-}[r] & \bullet^{-3}_Y \\
			        &                       \bullet^{-2}_Y \\
	}
$$

$$    \xymatrix{
       \bullet^{-2}_Y \ar@{-}[r] & \bullet^{-2}_X \ar@{-}[r]& \bullet^{-2}_X \ar@{-}[r] \ar@{-}[d]& \bullet^{-5}_Y\\
			        &                       &\bullet^{-2}_Y \\
	}
$$

$$    \xymatrix{
        \bullet^{-3}_Y \ar@{-}[r]& \bullet^{-2}_X \ar@{-}[r] \ar@{-}[d]& \bullet^{-3}_Y \ar@{-}[r] & \bullet^{-2}_Y \\
			                              &\bullet^{-2}_Y \\
	}
$$

$$    \xymatrix{
        \bullet^{-3}_Y \ar@{-}[r]& \bullet^{-2}_X \ar@{-}[r] \ar@{-}[d]& \bullet^{-5}_Y \\
			                              &\bullet^{-2}_Y \\
	}
$$
 
\end{fig}

\begin{fig}\label{Dihedral}
 Dihedral
 $$    \xymatrix{
        \bullet^{-2}_Y \ar@{-}[r]& \bullet^{-2}_X \ar@{-}[r] \ar@{-}[d]& \bullet^{-d_1}_N \ar@{-}[r] & \dots  \ar@{-}[r] & \bullet^{-d_k}_N \\
			        &                       \bullet^{-2}_Y \\
	}
$$
where $N=Y$ if $d_1 \ge 3$ and $N=X$ if $d_1=2$.

\end{fig}

This completes the classification of realizable graph with finite boundary fundamental group.

\subsection{Remarks on Fillings}\label{Fillings}
In this subsection, we study the fillings for a given capping divisor $D$.
First, we sketch the proof of the finiteness of fillings.
Then, we study the conjugate phenomena.
Finally, Liouville domain as a filling is considered.

\subsubsection{Finiteness}\label{Finiteness}

\begin{prop}\label{bounds}
 Suppose $D$ is a capping divisor with finite boundary fundamental group.
Then,   up to diffeomorphism,  only finitely many minimal symplectic manifolds can be compactified by $D$.
\end{prop}

\begin{proof} [Stetch of proof] We follow  the strategy   in \cite{Li08},\cite{BhOn12} and \cite{St13}.
We remark that this question is answered in \cite{BhOn12} for graphs in type (P1), (P2), (P3).
Therefore, it suffices  to consider the case that the graph of $D$ is a graph in type (P4) or (P5), which are all dual blown up graphs. 

Suppose $D$ is a capping divisor for a symplectic manifold $Y$, and let $W$ be the resulting closed manifold. 
By Theorem \ref{McDuff}, $W$ is rational since there is a positive sphere $Q$ in $D$ due to the dual blow up.
We can pick an orthonormal basis for $\{h,e_1, \dots, e_n \} $ for $H_2(W)$ such that $h^2=1$, $e_i^2=-1$ and $\omega(e_i) > 0$ for all $1 \le i \le N$.
Moreover, we can assume the positive sphere $Q$   is of class $h$  by Theorem \ref{McDuff}.

Let $C_j$ be the sphere corresponding to the vertex $v$ and suppose that its self-intersection is $1-d_j$. 
By Proposition \ref{combinaotric Lisca}, the homology of $C_j$ is $[C_j]=h-e_{i^j_1}-\dots-e_{i^j_{d_j}}$ for some $i^j_1, \dots, i^j_{d_j}$ distinct.
Moreover, we know that for the other spheres, the homology is of the form $e.-e.-\dots-e.$.

If we do iterative symplectic blow-downs away from the positive sphere $Q$, we will end up with $(\mathbb{CP}^2,\omega_0)$, and the image of $D$ under the blow-down maps can be made to be union of exactly $2$  $J$-holomorphic spheres for some $\omega_0$-tamed $J$ if the blow-down maps are carefully chosen. 

By keeping track of the homological effects of the blow-downs, one can  classify all possible $Y$  using the same reasoning as in \cite{Li08}, \cite{OhOn05} and \cite{BhOn12}.
In particular, one can obtain  finiteness.
\end{proof} 

Fixing a capping divisor $D$, we would like to investigate whether there are bounds for topological complexity among all minimal symplectic manifold that can be compactified by $D$.

The answer is no in general. 
By Donaldson's celebrated construction of Lefschetz pencil,  any closed symplectic 4 manifold can be decomposed into a disc bundle over a closed symplectic surface $\Sigma$ glued with a Stein domain.
In this case, we can view the symplectic surface $\Sigma$ as a symplectic capping divisor for the Stein domain.
One of the interesting problems in this case is to bound the topological complexity for a given genus $g$.
Some finiteness results of the topological complexity are obtained in \cite{Sm01} when the Lefschetz pencil has small genus.
However, it is proved in \cite{BaMo12} that there is no bound of the Euler number of the filling when the genus is greater than $10$.

In our setting, we allow 'reducible' symplectic capping divisor so one should hope for obtaining some finiteness results when the symplectic capping divisor has small geometric genus. 
By Proposition \ref{bounds}, bounds for diffeomorphic invariants are obtained when $D$ has finite boundary fundamental group, which is a special case of geometric genus being zero.

\subsubsection{Non-Conjugate Phenomena}\label{Non-Conjugate Phenomena}
Graphs in type (N) are  resolution graphs of distinct quotient singularities.
In particular, if $T_1$ and $T_2$ are graphs in (N), then they have different boundary fundamental groups except both $T_1$ and $T_2$ are resolution graphs of cyclic singularities (See e.g. \cite{Br68} Satz $2.11$, fourth column of the table).

When two graphs $T$ and $T^c$ admit strong realizations $D$ and $D^c$ respectively such that $D$ and $D^c$ are conjugate to each other, we say  that $T$ is conjugate to $T^c$. 
In Section \ref{Type (P1) to (P3)}, we mentioned  that each graph $T^N$ in type (N) has a conjugate graph $T^P$ in type (P1), (P2) or (P3), and vice versa.  
We are going to show that many  type (P5) graphs do not share this phenomena. 

There should be many ways to do it and we would like to use the first Chern class.
When $T$ admits a realization $D=C_1 \cup \dots \cup C_k$ inside a closed symplectic manifold $W$, the first Chern class of $W$ descends to the first Chern class $c_1^D$ for $P(D)$, where $P(D)$ is a plumbing of $D$.
Since $\partial P(D)$ is a rational homology sphere, $c_1^D$ lifts to a class uniquely in $H^2(P(D),\partial P(D), \mathbb{Q})$ by the Mayer-Vietoris sequence, which we still denote as $c_1^D$.
Then, by the Lefschetz-Poincare duality, we can identify it with a class in $H_2(P(D), \mathbb{Q})$, which is generated by $[C_1],\dots,[C_k]$.

\begin{defn}
 Keeping the notations as in the previous paragraph, we call $(c_1^T)^2+k$ the characterizing number of $T$ and denote it by $n^T$.
For $Y=W-P(D)$, we define the characterizing number of $Y$ to be $n^Y=(c_1^Y)^2+b_2(Y)$, where $c_1^Y$ and $b_2(Y)$ are the first Chern class and second Betti number of $Y$, respectively.
\end{defn}

\begin{lemma}\label{first chern class}
 Suppose $T$ is a graph in special types that admits a realization $D=C_1 \cup \dots \cup C_k$ in $W$.
Let $b=(s_1+2,\dots,s_k+2)^T$ and write $c_1^T=\sum\limits_{i=1}^k w_i[C_k]$.
Then, 

(i) $w=(w_1,\dots,w_k)^T$ satisfies $Q_Tw=b$,

(ii) if $T$ is of type (P) and $Y=W-P(D)$, then we have $n^T+n^Y=10$.

(iii) if $T^{(v)}$ is of type (P5) and $n^{T}+n^{T^{(v)}} \neq 10$, then there is no graph conjugate to $T^{(v)}$.

\end{lemma}

\begin{rmk}\label{non-standard contact structure}
Suppose  $T$ is a type (N3) graph and $T^{v}$ is  a dual blow up of $T$.  If $n^{T}+n^{T^{(v)}} \neq 10$ and $T^{(v)}$ is realizable, then it follows from (iii) that,   on  
the diffeomorphic boundaries of plumbings, 
the contact structure $\xi^{T^{v}}$ induced by the positive GS criterion on $T^{(v)}$ is not contactomorphic
to the canonical  contact structure $\xi^T$, which is  induced by the negative GS criterion on $T$.
In this case, we can actually  use the capping divisor $T^{(v)}$ to classify the symplectic fillings of the non-standard contact structure $\xi^{T^{v}}$ on the boundary of plumbing of $T^{(v)}$ (See also subsection \ref{Finiteness}).  
\end{rmk}

\begin{proof}
 Since the first Chern class is induced by a symplectic form, adjunction formula works in $P(D)$.
Therefore, we have $2+s_i=c_1^T[C_i]$, where $s_i$ is the self-intersection of $C_i$.
Hence, (i) follows immediately.

For (ii), since $W$ is rational, we have $(c_1^W)^2+b_2(W)=10$.
By the Mayer-Vietoris sequence, we have $H_2(P(D),\mathbb{Q}) \bigoplus H_2(Y,\mathbb{Q})= H_2(W, \mathbb{Q})$.
Thus, $b_2(W)=b_2(P(D))+b_2(Y)$ and $(c_1^W)^2=(c_1^T)^2+(c_1^Y)^2$, which proves (ii).

Finally, if the complement of plumbing of $T^{(v)}$ is  a plumbing of a symplectic divisor, say $D'$, 
then  $D'$ must be negative definite.
By Lemma \ref{0-0}, we know that $\pi_1(D')=\pi_1(T^{(v)})=\pi_1(T)$.
Among the type (N3) graphs, the boundary fundamental group uniquely characterize the graph (See \cite{Br68} Satz $2.11$, fourth column of the table).
Therefore, by the classification Theorem \ref{main classification theorem}, the graph of $D'$ must be $T$ and hence, (ii) implies (iii).
\end{proof}

\begin{eg}
Consider the resolution graph of $E_8$ singularities, which is given by
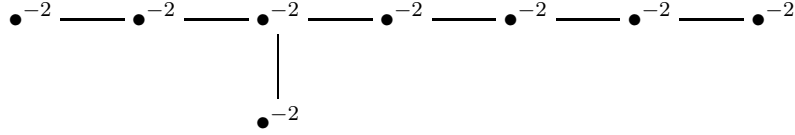
\begin{displaymath}
    \xymatrix{
        \bullet^{-2} \ar@{-}[r]& \bullet^{-2} \ar@{-}[r]& \bullet^{-2} \ar@{-}[r] \ar@{-}[d]& \bullet^{-2} \ar@{-}[r] & \bullet^{-2} \ar@{-}[r] &\bullet^{-2} \ar@{-}[r] &\bullet^{-2}\\
		               &	&\bullet^{-2} \\
	}
\end{displaymath}
By  Lemma \ref{first chern class}(i), the first Chern class is $c_1^{E_8}=0$

The following graph is a symplectic capping divisor of a plumbing of $E_8$, which we call $E_8^c$.
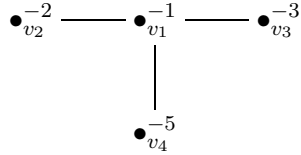
\begin{displaymath}
    \xymatrix{
        \bullet^{-2}_{v_2} \ar@{-}[r]& \bullet^{-1}_{v_1} \ar@{-}[r] \ar@{-}[d]& \bullet^{-3}_{v_3} \\
		               	&\bullet^{-5}_{v_4} \\
	}
\end{displaymath}
Then, by  Lemma \ref{first chern class}(i), we have $c_1^{E_8^c}=2[C^{v_1}]+[C^{v_2}]+[C^{v_3}]+[C^{v_4}]$, where $C^{v_i}$ is the sphere corresponding to $v_i$.
Direct calculation gives $(c_1^{E_8})^2=-2$ which is also predicted by Lemma \ref{first chern class}(ii).
 
\end{eg}

By a direct computation using mathematica, if $T$ in (N3) does not correspond to dihedral singularity, then there are only seven different (P5) graphs $T^{(v)}$ that satisfy $n^{T}+n^{T^{(v)}}=10$.
Moreover, by Theorem \ref{main classification theorem}, only four of them are realizable and they are given by the followings.
Therefore, these four graphs are the only exception that Lemma \ref{first chern class}(iii) cannot conclude anything among all graphs in (P5) not arising from dihedral resolution graph. 

\begin{fig}
For the following four type (N3) graphs $T$, the corresponding four type (P5) graphs $T^{(v)}$ satisfy $n^{T}+n^{T^{(v)}}=10$.

$$    
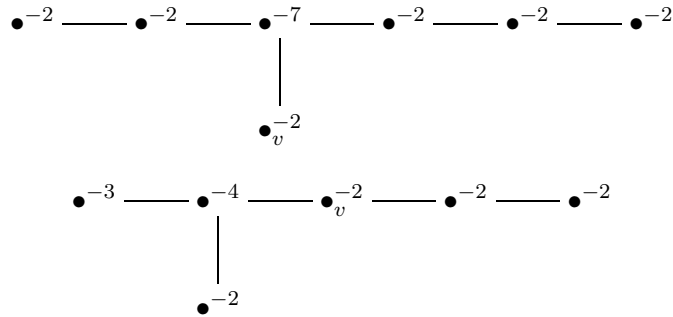
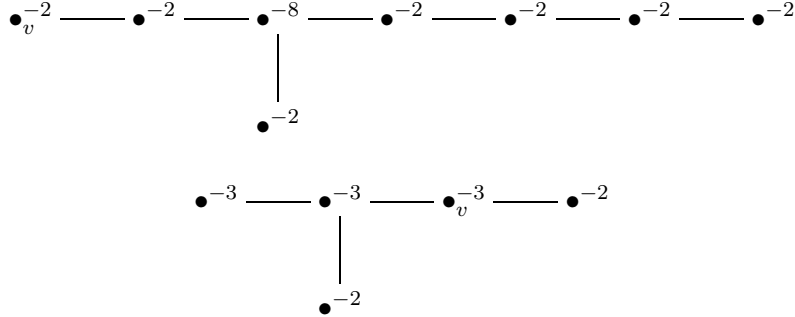
\xymatrix{
       \bullet^{-2} \ar@{-}[r] & \bullet^{-2} \ar@{-}[r]& \bullet^{-7} \ar@{-}[r] \ar@{-}[d]& \bullet^{-2} \ar@{-}[r] & \bullet^{-2} \ar@{-}[r] & \bullet^{-2} \\
			        &         &              \bullet^{-2}_v \\
	}
$$

$$    \xymatrix{
        \bullet^{-3} \ar@{-}[r]& \bullet^{-4} \ar@{-}[r] \ar@{-}[d]& \bullet^{-2}_v \ar@{-}[r] & \bullet^{-2} \ar@{-}[r] & \bullet^{-2} \\
			                              &\bullet^{-2} \\
	}
$$

$$    \xymatrix{
       \bullet^{-2}_v \ar@{-}[r] & \bullet^{-2} \ar@{-}[r]& \bullet^{-8} \ar@{-}[r] \ar@{-}[d]& \bullet^{-2} \ar@{-}[r] & \bullet^{-2} \ar@{-}[r] & \bullet^{-2} \ar@{-}[r] & \bullet^{-2} \\
			        &                       &\bullet^{-2} \\
	}
$$

$$    \xymatrix{
        \bullet^{-3} \ar@{-}[r]& \bullet^{-3} \ar@{-}[r] \ar@{-}[d]& \bullet^{-3}_v \ar@{-}[r] & \bullet^{-2} \\
			                              &\bullet^{-2} \\
	}
$$
\end{fig}

\subsubsection{Liouville Domain}
It is also interesting to know when a symplectic manifold, in particular, Liouville domain, can be compactified by a symplectic capping divisor.
Affine varieties are this kind of    Liouville domains.
Thus, in some sense we can regard  such Liouville domains  as symplectic analogues of  affine varieties.

For an affine surface $X$, log Kodaira dimension can be defined for the pair $(V,D)$, where $V$ is the completion of $X$ and $V-D=X$.
Moreover, this holomorphic invariant is independent of the compactification.
When the affine surface $X$  is a homology plane (also called affine acyclic),   McLean actually showed in  \cite{McL12-2}  that the log Kodaira dimension  is also a symplectic invariant.
Therefore, among all the Louville domains,  (rational) homology planes are particularly interesting.

It is a classical question in algebraic geometry to classify all (rational) homology planes (See the last Section of \cite{Mi00} and \cite{Za98}).
A common feature for such an  affine variety  is that its completion is a rational surface.
As we have seen, symplectic $4-$manifolds that can be compactified by a capping divisor with finite boundary fundamental group also share this phenomena. 
In particular, it would be interesting to know what symplectic capping divisors can compactify a Liouville domain that is a rational homology disk but the completion is not a rational surface.

Another classical question is to determine all singularities that admits a rational homology disks smoothing.
If the resolution graph for one such singularity is $\Gamma$, then in particular, the plumbing of $\Gamma$ can be symplectically filled by the rational homology disk.
We remark that this question is completely answered using techniques ranging from smooth topology, symplectic topology and algebraic geometry (See \cite{StSzWa08}, \cite{BhSt11} and \cite{PaShSt14}).

Using the same reasoning as in Lemma \ref{first chern class}, we have the following Lemma.

\begin{lemma}\label{rational homology disks filling}
 Suppose $T$ is a realizable type (P) graph.
If $T$ can symplectic divisorial compactify a rational homology disk, then we have $n^T=10$.
\end{lemma}

By  mathematica, we find that there are only four type (P5) graphs $T^{(v)}$ that are not arising from dihedral resolution graph and satisfy $n^{T^{(v)}}=10$.
Among these four, only three of them are realizable and they are listed in the following.
In particular, it means that apart from these three, all other strongly realizable graphs in type (P5) not arising from dihedral resolution graph cannot sympelctic divisorial compactify a rational homology disk.

\begin{fig}
For the following three  (N3) graphs $T$, the corresponding three  (P5) graphs $T^{(v)}$ satisfy $n^{T^{(v)}}=10$.

 $$    
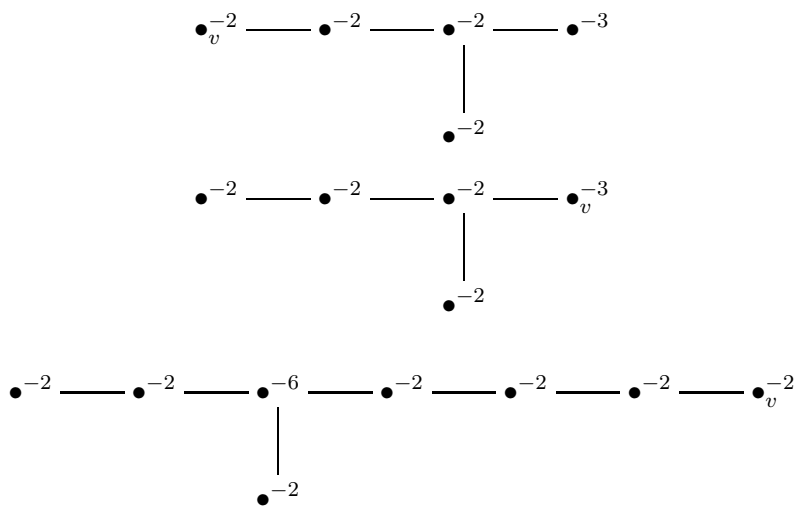
\xymatrix{
       \bullet^{-2}_v \ar@{-}[r] & \bullet^{-2} \ar@{-}[r]& \bullet^{-2} \ar@{-}[r] \ar@{-}[d]& \bullet^{-3} \\
			        &                       &\bullet^{-2} \\
	}
$$

 $$    \xymatrix{
       \bullet^{-2} \ar@{-}[r] & \bullet^{-2} \ar@{-}[r]& \bullet^{-2} \ar@{-}[r] \ar@{-}[d]& \bullet^{-3}_v \\
			        &                       &\bullet^{-2} \\
	}
$$

$$    \xymatrix{
       \bullet^{-2} \ar@{-}[r] & \bullet^{-2} \ar@{-}[r]& \bullet^{-6} \ar@{-}[r] \ar@{-}[d]& \bullet^{-2} \ar@{-}[r] & \bullet^{-2} \ar@{-}[r] & \bullet^{-2} \ar@{-}[r] & \bullet^{-2}_v \\
			        &                       &\bullet^{-2} \\
	}
$$

\end{fig}

Note that since most of the graphs in type (P5) do not have conjugate, the contact structure on the boundary is non-standard.
Therefore, the consideration above is not covered by \cite{StSzWa08}, \cite{BhSt11} and \cite{PaShSt14} (See Remark \ref{non-standard contact structure}).

\end{document}